\documentclass[10pt]{article}
\usepackage{hyperref}
\usepackage{amsmath, graphicx, latexsym, amssymb, amsthm, amsfonts}
\usepackage[mathscr]{eucal}
\usepackage{graphics}
\usepackage{color}
\usepackage{epsfig}
\usepackage{enumerate}
\usepackage{mathtools}
\usepackage{tabularx,makecell}

\newtheorem{lemma}{Lemma}[section]
\newtheorem{theorem}[lemma]{Theorem}
\newtheorem{corollary}[lemma]{Corollary}

\newtheorem{proposition}[lemma]{Proposition}

\newtheorem{assumption}[lemma]{Assumption}

\newtheorem{remark}[lemma]{Remark}

\newcommand{\eps}{{\varepsilon}}
\newcommand{\La}{{\Lambda}}
\newcommand{\Prob}{P}

\newcommand{\Id}{\mbox{Id}}

\newcommand{\MM}{{\mathbb{M}}}

\newcommand{\cla}{{\mathcal{A}}}

\newcommand{\clc}{{\mathcal{C}}}
\newcommand{\clt}{{\mathcal{T}}}
\newcommand{\clr}{{\mathcal{R}}}
\newcommand{\clu}{{\mathcal{U}}}
\newcommand{\clg}{{\mathcal{G}}}
\newcommand{\clv}{{\mathcal{V}}}
\newcommand{\clx}{{\mathcal{X}}}
\newcommand{\clh}{{\mathcal{H}}}
\newcommand{\cll}{{\mathcal{L}}}
\newcommand{\cld}{{\mathcal{D}}}
\newcommand{\clf}{{\mathcal{F}}}
\newcommand{\cli}{{\mathcal{I}}}

\newcommand{\one}{{\boldsymbol{1}}}

\newcommand{\Om}{\mathnormal{\Omega}}
\newcommand{\veps}{\mathnormal{\varepsilon}}

\newcommand{\E}{{E}}

\newcommand{\calB}{{\cal B}}

\newcommand{\calP}{{\cal P}}

\newcommand{\be}{\begin{equation}}
\newcommand{\ee}{\end{equation}}

\newcommand{\RR}{{\mathbb R}}
\newcommand{\clp}{{\cal P}}
\newcommand{\clm}{{\cal M}}
\newcommand{\clb}{{\cal B}}
\newcommand{\NN}{{\mathbb N}}
\newcommand{\dbl}{d_{\tiny{\mbox{bl}}}}
\newcommand{\pconv}{\xrightarrow{P}}
\newcommand{\BM}{\mbox{BM}}
\newcommand{\tr}{\mbox{tr}}

\newcommand{\vs}{\varsigma}
\newcommand{\uu}{\varrho}
\newcommand{\yy}{\mathbf{y}}
\newcommand{\PZ}{{}}
\mathtoolsset{showonlyrefs}

\numberwithin{equation}{section}
\begin{document}

\title{Large Deviations for Small Noise Diffusions Over Long Time}
\author{Amarjit Budhiraja, Pavlos Zoubouloglou}
\maketitle

\begin{abstract}
We study two problems. First, we consider the large deviation behavior of empirical measures of certain diffusion processes as, simultaneously, the time horizon becomes large and noise becomes vanishingly small. The law of large numbers (LLN) of the empirical measure in this asymptotic regime is given by the unique equilibrium of the noiseless dynamics. Due to degeneracy of the noise in the limit, the methods of Donsker and Varadhan (1976) are not directly applicable and new ideas are needed. Second, we study a system of slow-fast diffusions where both the slow and the fast components have vanishing noise on their natural time scales. This time the LLN is governed by a degenerate averaging principle in which local equilibria of the noiseless system obtained from the fast dynamics describe the asymptotic evolution of the slow component. We establish a large deviation principle that describes probabilities of divergence from this behavior. On the one hand our methods require stronger assumptions than the nondegenerate settings, while on the other hand the  rate functions take  simple and  explicit forms that have striking differences from their nondegenerate counterparts.

\noindent\newline


\end{abstract}

\section{Introduction}
\label{sec1}
In this work we study the large deviation behavior of certain stochastic dynamical systems with small noise over long time horizons. In order to motivate the problem of interest we begin with the following classical setting of Donsker-Varadhan large deviation theory \cite{donvar1, donvar2,donvar3} for ergodic diffusions.
Let $Z$ be a $\RR^d$-valued continuous stochastic process given as the solution of the following stochastic differential equation (SDE)
\begin{equation}
	dZ(t) = - \nabla \phi(Z(t)) dt +  dB(t), \; Z(0)= z_0
\end{equation}
where $B$ is a $d$-dimensional Brownian motion given on some probability space $(\Om, \clf, P)$, $z_0 \in \RR^d$,
 and $\phi:\RR^d \to \RR$ is a twice continuously differentiable function.
Suppose in addition that $\phi$ is bounded from below, has a bounded Hessian, and  $\|\nabla\phi(x)\| \to \infty$ as $\|x\| \to \infty$.
Consider the empirical measure process associated with $Z$ defined as
\begin{equation}\label{eq:mut}
	\mu_t(A) \doteq \frac{1}{t} \int_0^t \delta_{Z(s)}(A) ds, \; t >0, A \in \clb(\RR^d),
\end{equation}
where $\delta_x$ denotes the Dirac probability measure at the point $x$ and $\clb(S)$ for a topological space $S$ denotes the associated Borel $\sigma$-field. From \cite{donvar3} it follows that under the above conditions on $\phi$, the collection $\{\mu_t\}$ of $\clp(\RR^d)$ valued random variables, where $\clp(\RR^d)$ is the space of probability measures on $\RR^d$ equipped with the topology of weak convergence, satisfies a large deviation principle (LDP) with rate function $I^Z: \clp(\RR^d) \to [0,\infty]$ and speed $t$, namely for all continuous and bounded $F: \clp(\RR^d) \to \RR$
\begin{equation}
	-t^{-1} \log E\left[e^{-tF(\mu_t)}\right] \to \inf_{\mu \in \clp(\RR^d)}[F(\mu) + I^Z(\mu)],
\end{equation}
where the rate function $I^Z$ is given as
\begin{equation}\label{eq:1205}
	I^Z(\mu) \doteq \sup_{g \in \cld^+}  \left(-\int_{\RR^d} \frac{(\cll g)(x)}{g(x)} \mu(dx)\right), \; \mu \in \clp(\RR^d),
\end{equation}
where $\cll$ is the infinitesimal generator of the Markov process $Z$, whose evaluation for $g \in C_b^2(\RR^d)$ (the space of twice continuously differentiable bounded functions with bounded derivatives) is given as
$$(\cll g) \doteq - \nabla \phi \cdot \nabla g + \frac{1}{2}\Delta g,$$
where $\Delta$ is the $d$-dimensional Laplacian, and $\cld^+$ is the space of  functions $g$ in the domain of $\cll$  that are uniformly bounded below by a positive constant. The above large deviation principle gives asymptotics of probabilities of deviations of the empirical measure process $\mu_t$ from its law of large numbers (LLN) limit, which is the unique stationary distribution of the Markov process $Z$, for large values of $t$.
This basic result has been extended in subsequent works in many different directions (see e.g.\cite{gart, chi, dea1, dinney, jai, neynum2,  flesheson, fenkur}).

Our first interest in the current work is in the study of analogous long-time behavior for small noise diffusions.
Specifically, we consider the following setting.
Let $Z^{\eps}$ be a $\RR^d$-valued continuous stochastic process given as the solution of the following SDE
\begin{equation}\label{eq:1140}
	dZ^{\eps}(t) = - \nabla \phi(Z^{\eps}(t)) dt + s(\eps) dB(t), \; Z^{\eps}(0)= z_0
\end{equation}
where $B$ and $\phi$ are as before and
$s:(0,\infty)\to (0, \infty)$ satisfies $s(\eps)\to 0$ as $\eps \to 0$.
Due to the degeneracy of the noise in the limit as $\eps \to 0$, we will need stronger conditions than those needed for a LDP for $\{\mu_t\}$ defined by \eqref{eq:mut}.
Specifically, we assume that, in addition to $\phi$ being twice continuously differentiable with a bounded Hessian, $\phi$ is strongly convex and $\nabla \phi(0)=0$.
Under these assumptions on $\phi$, $0$ is the unique equilibrium of the ordinary differential equation (ODE):
\begin{equation}\label{eq:nlessod}
	\dot{z} = \nabla \phi(z).
\end{equation}
Also, it follows (see e.g. proof of Lemma \ref{theorem:auxillary}) that, as $\eps \to 0$,
\begin{equation}\label{eq:muepsdef}
	\mu^{\eps} \doteq \eps \int_0^{1/\eps} \delta_{Z^{\eps}(t)} dt \pconv \delta_0
\end{equation}
in $\clp(\RR^d)$.  Long time behavior of SDE with small noise  as in \eqref{eq:1140} is of interest, for example, in study of stochastic approximation schemes for approximating zeroes of a nonlinear function (cf. \cite{bor09, kus03, BMP12}).

One of the crucial ingredients in the proofs of \cite{donvar3} and other works on related themes is the nondegeneracy of the noise in the dynamics. This property is key in the proof of the lower bound where one invokes an ergodic theorem in order to suitably approximate near optimal paths in the variational problem describing the large deviation rate function. This feature of nondegeneracy is the main point of departure in the current work, as instead of empirical measures converging to the stationary distribution of an ergodic nondegenerate diffusion, in the current setting, these measures converge to a point mass given by the fixed point of the noiseless ODE in \eqref{eq:nlessod}.  The usual methods of studying empirical measure large deviations for Markov processes exploit nondegeneracy in the dynamics by considering relative entropies of near optimal measures with respect to the stationary distribution of the given diffusion. However these methods  are not applicable  here as typical measures of interest in our setting will be mutually singular, and therefore one needs new tools. We also remark that, if on the right side of \eqref{eq:1205} one naively replaces the second order operator with the limiting first order operator associated with the diffusion in \eqref{eq:1140}, namely $\cll^0g \doteq -\nabla \phi \cdot \nabla g$, the maximization in \eqref{eq:1205} gives $+\infty$ for any $\mu$ with a compact support and so  even a candidate for the rate function for $\mu^{\eps}$ is not immediate from \eqref{eq:1205} in this degenerate setting.

Our first result (Theorem \ref{theorem:ldp}) shows that  under the assumptions on $\phi$ made above, $\mu^{\eps}$ satisfies a large deviation principle  with speed $(\eps s^2(\eps))^{-1}$, and the associated rate function $I: \clp(\RR^d) \to [0,\infty]$ takes a particularly simple form given as
\begin{equation}
	\label{eq:ratefnsimpmod}
	I(\mu) = \frac{1}{2}\int_{\RR^d} \|\nabla \phi(y)\|^2 \mu(dy), \; \mu\in \clp(\RR^d).
\end{equation}
We note that unlike the rate function $I^Z$ associated with the ergodic diffusion $Z$, given in \eqref{eq:1205}, which is described through a variational formula,  the rate function in this small noise setting takes a surprisingly explicit form. 
The precise result we establish allows for a somewhat more general drift function and a state-dependent diffusion coefficient. The  conditions on the coefficients and the form of the rate function in this more general setting are given in Section \ref{sec:mainresult}.  The above result gives a LDP when simultaneously time becomes large and the noise intensity becomes small. A similar theme has recently been considered in \cite{dupwu}, where,  motivated by the problem of design of Monte-Carlo schemes, certain large deviation estimates have been established for suitable integrals  in the specific case where $s(\eps) = (\log(1/\eps))^{-1/2}$. In this case the relevant  techniques are those based on the Freidlin-Wentzell theory of quasipotentials of small noise diffusions \cite{frewen}. Note that in our result we do not make any assumptions on how $s(\eps)$ approches $0$. Furthermore, the paper \cite{dupwu} does not give a LDP for the empirical measure $\mu^{\eps}$.

The second focus of this work is the study of asymptotic behavior of fast-slow diffusions, when both slow and fast components have small noise in their natural time scales. The precise model of interest is described by a \PZ{$m+ d$} dimensional diffusion $(X^{\eps}, Y^{\eps})$ given as follows.
\begin{equation}
	\begin{aligned}
dX^\eps (t) &= b(X^\eps(t), Y^\eps(t))dt + s(\eps) \sqrt{\eps} \alpha(X^\eps(t))  dW(t), \quad X^\eps(0) = x_0, \, 0\le t \le T\\
 dY^\eps (t)& = -\frac{1}{\eps} \nabla_y U(X^\eps (t), Y^\eps (t))dt + \frac{s(\eps)}{\sqrt{\eps}}dB(t), \quad   Y^\eps(0) = y_0, \, 0\le t  \le T,
\end{aligned}
\label{eq:systemor}
\end{equation}
 Here, $T \in (0,\infty)$ is some fixed time horizon, $b: \mathbb{R}^{m+d} \to \mathbb{R}^m$, $\alpha:\mathbb{R}^{m} \rightarrow \mathbb{R}^{m \times k}$, $U: \mathbb{R}^{m+d} \rightarrow \mathbb{R} $
 are suitable coefficient functions, 
 $s(\eps)$ is as before,  and $W, B$ are $k$ and $d$ dimensional mutually independent Brownian motions respectively. Note that the natural time scale for $Y^{\eps}$ is $O(\eps)$ while that of $X^{\eps}$ is $O(1)$.  On these natural time scales the noise variances of the two processes are $O(s^2(\eps))$
 and $O(\eps s^2(\eps))$ respectively, both of which converge to $0$ as $\eps \to 0$.
 
 In contrast to the setting considered here, when $s(\eps) =1$, the above multiscale system falls within the framework of (nondegenerate) stochastic averaging principles, for which the associated large deviations theory has been well developed (cf. \cite{frewen, ver3, puhal1, dupspi, buddupgan2, puhal1}). Under appropriate conditions on the coefficient functions, these large deviations results give probabilities of deviations of the trajectory $X^{\eps}$, regarded as a random variable in the space $C([0,T]:\RR^m)$ (the space of continuous functions from $[0,T]$ to $\RR^m$ equipped with the usual uniform convergence topology), from its law of large number limit $X^0$ given as the solution of the following ODE:
 \begin{equation}
 	\dot{X}^0 = \bar b(X^0), \mbox{ where } \bar b(x) = \int_{\RR^d} b(x,y)\mu_x(dy), \; x \in \RR^m
 \end{equation}
 and for $x \in \RR^m$, $\mu_x$ is the unique stationary distribution of the diffusion
 $$dY_x (t) = - \nabla_y U(x, Y_x(t))dt + dB(t).$$
 The main insight that emerges from this LLN behavior is that the slow process, over the time scales at which the fast process equilibrates  towards its stationary distribution, stays approximately unchanged and its limit is governed by a parametrized family of stationary distributions associated with the fast diffusion where each stationary distribution corresponds to the local equilibrium of the fast process for a given value of the state of the slow process.
 
 In the setting of the current work ($s(\eps)\to 0$) the fast process on its natural time scale is driven by a small noise and thus in the scaling regime we consider, the asymptotics of the slow process  (under suitable conditions) are governed by the family of equilibria for the parametrized family of ODE:
 \begin{equation}
 	\dot{Y}_x = - \nabla_y U(x,Y_x).
 \end{equation}
 Under our assumptions, for each $x \in \RR^m$ the above ODE will have a unique equilibrium point $y(x) \in \RR^d$ and the LLN of $X^{\eps}$ defined in \eqref{eq:systemor} is given by 
 \begin{equation}
 	\dot{X}^0 = \bar b(X^0), \mbox{ where } \bar b(x) = b(x, y(x)), \; x \in \RR^m.
 \end{equation}
 The goal of this work is to study the behavior of  probabilities of large deviations of the  process $X^{\eps}$ from its LLN limit given by $X^0$. The key challenge in the degenerate setting considered here is that, unlike the case $s(\eps)=1$ where the local equilibria are mutually absolutely continuous, when $s(\eps)\to 0$ as $\eps \to 0$, the family of equilibria  are mutually singular (except the trivial case when they are the same). Once again the proof techniques used in the nondegenerate setting are not applicable here and different ideas are needed.
 In Theorem \ref{theorem:ldp-2} we establish a large deviation principle for $X^{\eps}$ under appropriate conditions on the coefficient functions. Our results in fact give a stronger result which provides a LDP for the pair 
 $(X^{\eps}, \Lambda^{\eps})$ in $C([0,T]: \RR^m) \times \clm_1$, where
 $\clm_1$ is the  space of finite measures $\nu$ on $\mathbb{R}^d \times [0,T]$ such that $\nu(\mathbb{R}^d \times [0,t]) =
  t$ for all $t \in [0,T]$, equipped with the weak convergence topology, and
 $\Lambda^{\eps}$ is a $\clm_1$ valued random variable defined as
 \begin{equation}\label{eq:557}
	 \Lambda^{\eps}(A \times [0,t]) \doteq \int_{[0,t]} 1_{A}(Y^{\eps}(s)) ds, \; t \in [0,T], \; A \in \clb(\RR^d).\end{equation}
 The precise rate function governing the LDP can be found in Section \ref{sec:m2} (see \eqref{eq:rate-function-2}) but we note here that in the special case where $m=k$ and $\alpha$ is the identity matrix, the rate function takes a simple explicit form as
 \begin{equation}
 	I(\xi, \mu) = \frac{1}{2} \int_{ \RR^d \times [0,T]} \|\nabla_y U(\xi(s), y)\|^2 \mu(dy\, ds) +
	\frac{1}{2} \int_{[0,T]} \left\|\dot{\xi}(t)-\int_{\RR^d} b(\xi(t), y)\mu(t,\, dy)\right\|^2 dt,
 \end{equation}
 for $(\xi, \mu) \in C([0,T]: \RR^m) \times \clm_1$, where $\mu(dy\, ds) = \mu(s, dy)\, ds$.
 Roughly, the second term in the rate function arises from the large deviations of the Brownian motion $W$ whereas the first term captures the deviations of the fast process from the collection of its local equilibria. More precisely, the first term can be interpreted as the instantaneous cost associated with the deviations of a set of points described by the measure $\mu(s,dy)$, for each time instant $s$, from its equilibrium point $y(\xi(s))$. Once again the form of the rate function has striking differences from that in the nondegenerate setting (cf. \cite{frewen, ver3, dupspi}).
 
 We remark that the proof of a LDP for $\mu^{\eps}$ defined in \eqref{eq:muepsdef}  is a simpler analogue of the proof of Theorem  \ref{theorem:ldp-2}  and in fact can be deduced from it. However we present these results separately for two reasons. First, the basic idea of the proof (particularly of the LDP lower bound) is significantly  simpler  and clearer to see in the setting of Theorem \ref{theorem:ldp}  and sets the general framework  for the more involved setting in Theorem \ref{theorem:ldp-2}. Second, in Theorem \ref{theorem:ldp} we treat a more general setting than the one discussed in the Introduction because of which this result cannot be immediately deduced from Theorem \ref{theorem:ldp-2}.
 
 We now make comments on proof ideas.
 
 \subsection{Proof Strategy.}
 The starting point for the proofs of both Theorem \ref{theorem:ldp}
 and \ref{theorem:ldp-2} is a variational formula for moments of nonnegative functionals of finite dimensional Brownian motions due to Bou\'{e} and Dupuis\cite{boudup} (see Theorem \ref{Variational-1}). Using this formula the basic problem of large deviations reduces to establishing convergence of costs associated with certain stochastic control problems to those associated with suitable deterministic optimization problems. This convergence is shown by establishing a complementary set of asymptotic inequalities between the costs, one giving the large deviation upper bound (see \eqref{ineq-single-upper} and \eqref{ineq-two-upper}) while the other giving the large deviation lower bound 
 (see \eqref{ineq-single-lower} and \eqref{ineq-two-lower}). Proof of the upper bound proceeds by weak convergence arguments that also reveal the precise form of the large deviations rate function. This form emerges from a key orthogonality property (see \eqref{eq:eq321} for Theorem \ref{theorem:ldp} and \eqref{eq:eq321b} for Theorem \ref{theorem:ldp-2}) that is behind the inequalities in Lemma \ref{lem:ineqincos} and Lemma \ref{system-upper-inequality-representation} and which in turn give the large deviation upper bound. Proofs of the lower bounds are somewhat long and involved and require several approximating constructions. We only comment on the arguments for Theorem \ref{theorem:ldp-2} as those used for Theorem \ref{theorem:ldp} are simpler analogues. The basic approach in the proof of the lower bound is the construction of simple form near optimal paths $\xi^*$ and occupation measures $\nu^*$ for the deterministic optimization problem on the right side of \eqref{ineq-two-lower} (see Lemma \ref{lem:approxdisc}). This is then used to construct suitable controls and controlled processes for the prelimit stochastic system which appropriately converge to the chosen near optimum. In doing so one needs to ensure that the corresponding prelimit occupation measure $\bar \Lambda^{\eps}$ in \eqref{eq:laeps} charges the asymptotically correct periods of time in the correct regions of the state space of the fast process that are dictated by the near optimum occupation measure $\nu^*$. One also needs to ensure that the cost incurred in doing so is suitably close to the cost associated with the near optimum $(\xi^*, \nu^*)$. In achieving these dual goals one needs to design suitable controls that appropriately modify the dynamics of the fast process so that the local equilibria of the process are sufficiently close to $\nu^*(s, dy)$ at all time instants $s$. This construction, which is given in Section \ref{sec:conscont}, is at the heart of the lower bound proof. The idea is for the control to move the state process from one point in the support of $\nu^*(s,dy)$ to the next in a very small amount of time with negligible cost and then keep the process near this latter point for the correct amount of time as dictated by $\nu^*(s, dy)$ while incurring the optimum amount of cost. This basic idea takes a somewhat simpler form for the proof of Theorem \ref{theorem:ldp} and we refer the reader to Section \ref{subsec:single-lower} for a more detailed outline of the strategy for this setting.
 
The Rest of the paper is organized as follows. We close this section by summarizing the basic notation and terminology used. Sections \ref{sec:pfthmone} and \ref{sec:pfthmtwo} contain the proofs of Theorems \ref{theorem:ldp} and \ref{theorem:ldp-2} respectively. Organizations of these proofs are summarized at the beginning of the corresponding sections.

\subsection{Notation and Terminology.}

The following notation will be used.  
For a Polish space $\mathcal{X}$, we will denote the space of continuous, real-valued functions on $\mathcal{X}$ by $\mathcal{C}(\mathcal{X})$ and by $\mathcal{C}_c(\mathcal{X})$ (resp. $\mathcal{C}_b(\clx)$) the subset of
 $\mathcal{C}(\mathcal{X})$ consisting of functions with compact support (resp. that are bounded). We say  a function $f:\RR^d \to \RR$ is $\mathcal{C}^k, k \in \NN$
  if $f$ is continuously differentiable  $k$-times. Such a function is said to be in $\clc_b^k$ if the function and all its derivatives up to the $k$-th order are bounded.
  We denote by $\mathcal{C}^\infty(\mathbb{R}^d)$ the space of infinitely differentiable functions from $\mathbb{R}^d$ to $\mathbb{R}$.
   $\mathcal{C}([0,T]:\mathbb{R}^d)$ will denote the space of continuous functions from $[0,T]$ to $\mathbb{R}^d$ which will be equipped with the usual uniform topology induced by the sup-norm.  We denote by $L^2([0,T]:\mathbb{R}^d)$  the space of square integrable functions from $[0,T]$ to $\mathbb{R}^d$. For $v \in L^2([0,T]:\mathbb{R}^d)$, we write its $L^2$-norm
   $(\int_{[0,T]} \|v(s)\|^2 ds)^{1/2}$ as $\|v\|_2$.
   For $m, d \in \NN$ and a $\clc^2$
$f: \RR^{m+d} \to \RR$,  $\mathcal{H}f$ will denote the $(m+d)\times(m+d)$-dimensional Hessian matrix of $f$ with regard to all variables, and for $(x,y) \in \RR^m\times \RR^d$, $\mathcal{H}_{x}f(x,y)$ will denote the $m \times m$ Hessian matrix of $f$ with regard to $x \in \RR^m$, and $\mathcal{H}_{y} f$ is defined analogously. Similarly $\nabla f$ denotes the $(m+d)$-dimensional vector that is the gradient of $f$, $\nabla_{x} f$ the $m-$dimensional vector that is the gradient of $f$ with regard to the variables in $x$, and $\nabla_y f$ is defined similarly. For a matrix $a$, we denote its transpose by $a^T$ and its trace (when meaningful) by $\tr(a)$. $\Id$ will denote the  identity matrix with dimension clear from the context.
We denote by $\clb(\clx)$ the Borel $\sigma$-field on $\clx$.  $\mathcal{P(\mathcal{X})}$ will denote the space of probability measures on ($\mathcal{X}, \clb(\clx)$), equipped with the topology of weak convergence. This topology can be metrized using the bounded-Lipschitz distance defined as: for $\mu, \nu \in \clp(\clx)$
\begin{equation} \label{eq:510n} \dbl (\mu,\nu) \doteq \sup_{f \in BL_1(\clx)} | \int f d\mu - \int f d\nu | ,\end{equation}
where $BL_1(\clx)$ is the space of all  Lipschitz functions from $\clx$ to $\RR$ that are bounded by $1$ and have Lipschitz constant bounded by $1$. For $x \in \clx$, $\delta_x$ will denote the Dirac probability measure concentrated at the point $x$. For $\clx$ valued random variables $X_n, X$, we denote the convergence in distribution (resp. in probability) of $X_n$ to $X$ as $X_n \Rightarrow X$ (resp.
$X_n \xrightarrow{P} X$).
$\BM(\clx)$ will  denote the set of bounded and measurable real valued functions on $\mathcal{X}$.
 For a bounded $\mathbb{R}^d$ valued function $f$ on $\mathcal{X}$, we denote $\|f\|_\infty \doteq  \sup_{x \in \mathcal{X}} \|f(x)\|$.
 A function $I: \clx \to [0, \infty]$ is called a $\textit{rate function}$ (on $\mathcal{X}$) if it  has compact level sets, i.e. for each $M \in (0, \infty)$ the level set $\{x \in \mathcal{X}: I(x) \leq M\}$ is a compact subset in $\mathcal{X}$. As a convention, infimum over an empty set is taken to be $\infty$. We will consider collections indexed by a positive parameter $\eps$, and by convention, $\eps$ will always take values in $(0,1)$. 

A collection of $\mathcal{X}-$valued stochastic processes $\{X^\eps\}$  is said to satisfy the $\textit{Laplace Principle}$ on $\mathcal{X}$ with rate function $I$ and speed $\alpha(\eps)$, where
$ \alpha(\eps) \rightarrow \infty$ as $\eps\rightarrow 0$ if for every $F \in \mathcal{C}_b(\mathcal{X})$
\begin{equation}
    \lim_{\eps\rightarrow 0} - \frac{1}{\alpha(\eps)} \log \E e^{-\alpha(\eps) F(X^\eps)} = \inf_{x \in \mathcal{X}} \bigg( F(x) + I(x)  \bigg).
\end{equation}
We say the Laplace upper (resp. lower) bound holds if the left side is bounded below (resp. above) by the right side.  We recall that the collection $\{X^\eps\}$  satisfies the $\textit{Large Deviation Principle}$ on $\mathcal{X}$ with rate function $I$ and speed $\alpha(\eps)$, if and only if it satisfies the Laplace principle.
%

\section{Main Results}
\label{sec:mainresult}
In this section we present our two main results. The first result  concerns the LDP for the empirical measure of certain small noise diffusions while the second result studies large deviations for a class of slow-fast system of diffusions with vanishing noise. The results are described below in Sections \ref{sec:m1} and \ref{sec:m2} respectively.
\subsection{Empirical measure for Small Noise Diffusions.}
\label{sec:m1}
We will consider a somewhat more general setting than the one considered in the Introduction and after stating the main result we remark on how the model considered in the Introduction is covered by this result.
The collection of diffusions we study takes the form
\begin{equation}
    dY^\eps(t) =  -\frac{1}{\eps}  \psi(Y^\eps (t))dt + \frac{s(\eps)}{\sqrt{\eps}}  \sigma(Y^\eps (t))dB(t), \quad Y^\eps(0) = y_0, \, 0 \le t \le 1,
    \label{eq:SDE_single}
\end{equation}
where  $B$ is a $r$-dimensional $\{\clf_t\}_{0\le t \le 1}$ standard Brownian motion  given on some 
filtered probability space $(\Omega, \mathcal{F}, \{\clf_t\}, P)$ satisfying the usual conditions, $y_0 \in \RR^d$,    and $s(\eps) \to 0$ as $\eps \to 0$. Throughout, without loss of generality, we assume that $s(\eps) \in (0,1)$. We will make the following assumptions on the coefficient functions  $\psi$ and $\sigma$.

\begin{assumption}
\begin{enumerate}
Let $a = \sigma\sigma^T$.
\item $[\textit{Diffusion Coefficient}]$ $\sigma: \RR^d \to \RR^{d\times r}$ is a bounded Lipschitz map. The matrix function $a$ is uniformly nondegenerate: for some $c_a \in (0, \infty)$
$$v^Ta(y)v \ge c_a \|v\|^2 \mbox{ for all } y \in \RR^d, \; v \in \RR^d.$$
\item $[\textit{Drift Coefficient}]$ There is a $\mathcal{C}^2$ function $\phi: \mathbb{R}^d \rightarrow \mathbb{R}$ such that
\begin{enumerate}
    \item $\psi(y) = a(y) \nabla \phi (y)$ for all $y \in \mathbb{R}^d$,
    \item $\sup_{y \in \mathbb{R}^d} \|\mathcal{H} \phi (y) \| < \infty $,
    \item $\nabla \phi (y) = 0$ if and only if $y = 0$,
    \item 
	$\|\nabla \phi (y)\|^2 \rightarrow \infty $ as $\|y\| \rightarrow \infty$.
\end{enumerate}
\item $[\textit{Asymptotic Stability}]$ 
For every $x \in \mathbb{R}^d$ with 
$$\mathcal{V}_x(y) \doteq  a(x+y) \nabla \phi(x+y) - a(x) \nabla \phi(x), \, y \in \mathbb{R}^d,$$
 the ODE $\dot{\xi} = - \mathcal{V}_x (\xi)$ has 0 as the unique fixed point, which is globally asymptotically stable.

%
%

\item $[\textit{Lyapunov Function}]$ For each $x \in \mathbb{R}^d$, there exists a $\clc^2$ function $V_x: \mathbb{R}^d \rightarrow \mathbb{R}$ such that for some $\alpha_i(x) \in (0,\infty), i = 1,2$ and $c_i(x) \in (0,\infty), i = 1,2$, the following hold:
$ \| \mathcal{H} V_x\|_\infty < \infty$,
 \begin{align*}
	  \alpha_2(x) \|\xi\|^2 + 1 &\le V_x(\xi) \leq \alpha_1 (x) ( 1 + \|\xi\|^2),  \\
  \mathcal{V}_x(\xi) \nabla V_x(\xi) &\geq c_1(x)\|\xi\|^2 -c_2(x),  
 \end{align*}
for all $\xi \in \mathbb{R}^d$.
\end{enumerate}
\label{assumptions-1}
\end{assumption}
The following is our first main result.
\begin{theorem}
Suppose Assumption \eqref{assumptions-1} holds. Then, the map $I_1: \mathcal{P}(\mathbb{R}^d) \rightarrow [0,\infty]$ defined as 

\begin{equation}
I_1(\gamma) = \frac{1}{2} \int_0^1 \|\sigma^T (y) \nabla \phi(y)\|^2 \gamma(dy)
=\frac{1}{2} \int_0^1 \|\sigma^T (y) a^{-1}(y)\psi(y)\|^2 \gamma(dy),
 \quad \gamma \in \mathcal{P}(\mathbb{R}^d)
\label{eq:rate-function}
\end{equation}
is a rate function. Additionally, the collection $\{\mu^{\eps}\}$ of $\clp(\RR^d)$ valued random variables defined as
\begin{equation} \label{eq:occupational-1} 
	\mu^{\eps}(A) \doteq \int_0^1 \delta_{Y^{\eps}(t)}(A) dt, \, A \in \clb(\RR^d)
\end{equation}
 satisfies a LDP on $\mathcal{P}(\mathbb{R}^d)$ with rate function $I_1$ and speed $(\eps s^2(\eps))^{-1}$.
\label{theorem:ldp}
\end{theorem}

\begin{remark}\label{rem:diffscal}
	Define $Z^{\eps}(t) \doteq Y^{\eps}(t\eps)$, $t \ge 0$. Then $Z^{\eps}$ satisfies
	$$dZ^{\eps}(t) = -\psi(Z^\eps (t))dt + s(\eps) \sigma(Z^\eps (t))d\tilde B(t), \quad Z^\eps(0) = y_0,$$
	where $\tilde B$ is a $r$-dimensional Brownian motion and $\mu^{\eps}$ can be rewritten as
	\begin{equation}\label{eq:115nn}
		\mu^{\eps}(A) \doteq \eps \int_0^{1/\eps} \delta_{Z^{\eps}(t)}(A) dt, \, A \in \clb(\RR^d).\end{equation}
	Thus the above theorem gives a LDP for the $1/\eps$-time horizon empirical measure of $Z^{\eps}$ with speed $(\eps s^2(\eps))^{-1}$.
\end{remark}

%


\begin{remark}
	We now give examples where Assumption \ref{assumptions-1} is satisfied. 
	\begin{enumerate}[(a)]
		\item Suppose that $\sigma(y) = \sigma$ for all $y \in \mathbb{R}^d$ and $a = \sigma \sigma^T$ is invertible. Suppose further that there is a $\mathcal{C}^2$ strongly convex function $\Tilde{\phi}:\mathbb{R}^d \rightarrow \mathbb{R}$ such that $\| \mathcal{H} \Tilde{\phi}\|_\infty < \infty$, $\nabla \Tilde{\phi} (0) = 0$, and $\psi(y) = \nabla \tilde{\phi}(y)$. Then Assumption \ref{assumptions-1} is satisfied. This is checked as follows. \\
		Part 1  and 2(a,b,c) of the assumption are clearly satisfied.\\
		Consider now  part 2(d). Strong convexity  implies that there exists some $m \in (0,\infty)$, such that for all $x,y \in \mathbb{R}^d$ 
		$$\tilde{\phi}(y) \geq \tilde{\phi}(x) + \nabla \tilde{\phi}(x) (y-x) + m \|y-x\|^2.$$
		 Taking $x = 0$, we see that $\tilde{\phi}(y) \geq \tilde{\phi}(0) + m \|y\|^2$.
		  Also, $\nabla\tilde{\phi}(y) = \int_0^1 \mathcal{H}\tilde{\phi}(ty) \cdot y dt$ and so $\|y\| \|\nabla\tilde{\phi}(y)\| \geq y \cdot \nabla \tilde{\phi}(y)  = \int_0^1 y \cdot \mathcal{H} \tilde{\phi}(ty) \cdot y dt \geq m \|y\|^2$. It then follows $\|\nabla\tilde{\phi}(y)\| \geq m \|y\|$, for all $y \in \mathbb{R}^d$. Thus 
		  since $a$ is invertible,
		  $\|\nabla \phi(y)\|^2 \rightarrow \infty $, as $\|y\| \rightarrow \infty$.
		   This shows that part 2(d) holds as well. Thus we have shown that part 2 of the assumption holds.\\
Now consider part 3. Note that  for all $x,y \in \mathbb{R}^d$, $$\mathcal{V}_x(y) = a \left( a^{-1} \nabla \tilde{\phi}(x+y) - a^{-1} \nabla \tilde{\phi}(x)  \right) = \nabla \tilde{\phi}(x+y) - \nabla \tilde{\phi}(x).$$
 Since for every $x,$ the map $\eta_x$ defined as $\eta_x(y) \doteq \tilde{\phi}(x+y) - y \cdot \nabla \tilde{\phi}(x)$ is strongly convex (in $y$) and $\nabla \eta_x (0) = \clv_x(0) = 0$, we have part 3.\\
Finally consider part 4. Let 
$$V_x(\xi) \doteq  \tilde{\phi}(\xi + x) - \tilde \phi(x) - \xi \cdot \nabla \tilde{\phi}(x)+ 1.$$ 
By the assumption on the boundedness of $\mathcal{H} \tilde{\phi}$ we have that $\|\mathcal{H} V_x(\xi)\|_\infty <\infty$.
Also, by strong convexity, $V_x(\xi) \ge m \|\xi\|^2+1$ which gives the lower bound in the first inequality in part 4. The upper bound is an immediate consequence of $\|\mathcal{H} V_x(\xi)\|_\infty <\infty$.
For the second inequality in part 4, once more using strong convexity,
 $$\mathcal{V}_x(\xi) \cdot \nabla V_x(\xi) =  \| \nabla \tilde{\phi}(x+\xi) - \nabla \tilde{\phi}(x)\|^2\ge m^2 \|y\|^2.$$
 Thus we have shown that Assumption \ref{assumptions-1} is satisfied. Combining this observation with Remark \ref{rem:diffscal}, we see that, with $Z^{\eps}$ defined by \eqref{eq:1140}, the collection $\{\mu^{\eps}\}$ defined by \eqref{eq:115nn}  satisfies a LDP with rate function $I$ as in \eqref{eq:ratefnsimpmod}.

 \item Assumption \ref{assumptions-1} also holds for  diffusions with a multiplicative noise as the following example illustrates.\\
  Let $d = m = 1$ and suppose that $\psi: \RR \to \RR$ is such that $\psi(0) = 0$ and $\tilde{\phi}(x) = \int_0^x \psi(y) dy$ is a $\mathcal{C}^2$ strongly convex map with $\sup_x |\tilde\phi''(x)| < \infty$.\\
   Also suppose that for some $c_1, c_2>0$, $c_1 \leq |\sigma(x)| \leq c_2$ for all $x \in \mathbb{R}$,  $x \mapsto \sigma(x)$ is a differentiable Lipschitz function, and  $y^2 |\sigma'(y)|^2 \rightarrow 0,$ as $|y| \rightarrow \infty$.\\
    Then, Assumption \ref{assumptions-1} holds with $\phi(x) = \int_0^x \frac{1}{a(y)} \tilde{\phi}'(y) dy,$ where $a(y) = \sigma^2(y)$. To see this note the following.
	
Part 1 holds by assumption. Part 2 (a) follows from the identity  $a(x) \phi'(x) = \tilde{\phi}'(x) = \psi(x)$.  Part 2(b) is immediate from our assumption and boundedness of $\tilde \phi''$.  For 2(c), from our condition on $\psi$, we have  $\phi'(x) = 0$ iff $x = 0$ and that $|\sigma(x)|>0$ for all $x$. Finally, part 2(d) follows from the strong convexity of $\phi$ and the bounds on $|\sigma(x)|$.\\
 For part 3, note that
 $\mathcal{V}_x(y) = 
  \tilde{\phi}'(x+y) - \tilde{\phi}'(x).$
  Since for all $x$, $\eta_x(y) = \tilde{\phi}(x+y) -  y \tilde{\phi}'(x)$ is strongly convex in $y$, $\mathcal{V}_x(y) = 
  \eta_x'(y)$, and $\mathcal{V}_x(0) = a(x) \phi'(x) - a(x) \phi'(x) =0$, we have that part 3 is satisfied. \\
  Finally part 4 is verified as in (a).
Thus we have verified all statements in Assumption \ref{assumptions-1}.\\
   In a similar manner one can construct multidimensional examples with multiplicative noise as well.
\end{enumerate}

\end{remark}

\subsection{Multiscale System of Diffusions with Vanishing Noise.}
\label{sec:m2}
In this section we present our main result for the multiscale system $(X^{\eps}, Y^{\eps})$ introduced in \eqref{eq:systemor}.
We begin with our main assumption on the coefficients.

\begin{assumption}
	The functions $b, \alpha$ and $U$ satisfy the following.
\begin{enumerate}
    \item $[\textit{Coefficients of the slow component}]$
    The coefficients $b: \mathbb{R}^m \times \mathbb{R}^d \rightarrow \mathbb{R}^m$ and $\alpha: \mathbb{R}^m \rightarrow \mathbb{R}^{m \times k}$ are Lipschitz: there exist some $L_b, L_\alpha \in (0,\infty)$ such that for all $x,x' \in \mathbb{R}^m, y,y' \in \mathbb{R}^d$ 
    \[ \|b(x',y') -b(x,y)\| \leq L_b ( \|x'-x\| + \|y'-y\|), \;\;
     \|\alpha(x') -\alpha(x)\| \leq L_\alpha \|x'-x\|.\]
  Furthermore  $\alpha$ is a bounded function.
    
    \item $[\textit{Coefficients of the fast component}]$ 
    $U: \mathbb{R}^{m+d} \rightarrow \mathbb{R}$ is a $\mathcal{C}^2$ function such that the following hold:
    \begin{enumerate}
\item $[\textit{Growth of Hessian}]$
        \[ \sup_{(x,y) \in \mathbb{R}^m \times \mathbb{R}^d} \left( \|\mathcal{H}_y U(x,y)\| + \frac{\|\mathcal{H}_x U(x,y)}{1+\|x\|+\|y\|}  \right) \doteq L_{\mathcal{H}U} < \infty.  \]
            
        \item $[\textit{Growth of $x$-gradient}]$ 
        \[ \sup_{(x,y) \in \mathbb{R}^m \times \mathbb{R}^d} \frac{\| \nabla_x U(x,y)\|}{1 + \|x\| + \|y\|} < \infty. \]

        \item $[\textit{Lower bound on $U$ and its $y$-gradient}]$ There exist constants $L^1_{low}, L^2_{low} \in (0, \infty)$ such that
        
        \begin{equation}
        \nonumber
            \inf_{x \in \mathbb{R}^m} \bigg( U(x,y) + \|\nabla_y U(x,y) \|^2   \bigg) \geq L^1_{low} \|y\|^2 - L^2_{low}.
        \end{equation} 
        
        \item $[\textit{Stability of the fast component}]$ For each $(x,z) \in \mathbb{R}^{m+d}$, the ODE
        \begin{equation}
            \nonumber
            \dot{u} = - \mathcal{V}_{x,z}(u)
        \end{equation}
        where the function $\mathcal{V}_{x,z}: \RR^d \to \RR^d$ defined as $\mathcal{V}_{x,z}(y) = \nabla_y U(x,y+z) - \nabla_y U(x,z)$ for $y \in \mathbb{R}^d$ has $0$ as the unique fixed point, which is globally asymptotically stable.
        
        \item $[\textit{Growth and smoothness of $y$-gradient }]$  The map $(x,y) \mapsto \nabla_y U(x,y)$ is Lipschitz. For each $y \in \mathbb{R}^d$,
        \[ \sup_{x \in \mathbb{R}^m} \|\nabla_y U(x,y)\| < \infty.    \]
        
        \item $[\textit{Fixed point}]$ There exists a  Lipschitz map $\theta: \mathbb{R}^m \rightarrow \mathbb{R}^d$ such that, for each $x \in \mathbb{R}^m$, $\nabla_yU(x,\theta(x)) = 0$. 
            \end{enumerate}

 %
    
\end{enumerate}
\label{Assumptions:system}
\end{assumption}

\begin{remark}
Let $\theta: \mathbb{R}^m \rightarrow \mathbb{R}^d$ be in $\mathcal{C}_b^2$. 
Let $\phi: \RR^d \to \RR$ be a $\clc^2$ strongly convex function with bounded Hessian such that $\nabla \phi(0) =0$. Then $U(x,y) \doteq \phi(y - \theta(x))$, $(x,y) \in \RR^{m+d}$, is one basic example that satisfies  Assumption \ref{Assumptions:system}(2). 
\end{remark}

Recall the space $\clm_1$ introduced above  \eqref{eq:557} which is equipped with the weak convergence topology. Once more we will metrize it using the bounded-Lipschitz distance as in \eqref{eq:510n}.
Also recall the collection of $\clm_1$ valued random variables $\{\Lambda^{\eps}\}$ defined in \eqref{eq:557}.
We now introduce the rate function associated with a LDP for $(X^{\eps}, \Lambda^{\eps})$. 
Let $\clx \doteq \mathcal{C}([0,T]:\mathbb{R}^m)$.
Note that a $\nu \in \clm_1$ can be disintegrated as $\nu(dy\, ds) = \hat{\nu}_s(dy) ds$ where
   $s \mapsto \hat{\nu}_s$ is a measurable map from $[0,T]$ to $\mathcal{P}(\mathbb{R}^d)$. 
 
 Now, for $(\xi, \nu) \in \mathcal{X}\times \clm_1  $, let $\mathcal{U}(\xi, \nu)$ be the class of all $v \in  L^2([0,T]:\mathbb{R}^k)$ such that $\xi$ solves
\begin{equation}
    \nonumber
    \xi(t) = x_0 + \int_0^t \int_{\mathbb{R}^d} b(\xi(s),y) \hat{\nu}_s(dy) ds + \int_0^t \alpha(\xi(s)) v(s) ds, \; t \in [0,T].
\end{equation}
 Define $I_2 : \mathcal{X} \times \clm_1 \rightarrow [0,\infty]$ as

\begin{equation}
    I_2(\xi, \nu) = \inf_{ v \in \mathcal{U}(\xi,\nu)}  \frac{1}{2} \int_0^T \left( \|v(s)\|^2 +  \int_{\mathbb{R}^d} \|\nabla_y U (\xi(s),y) \|^2 \hat{\nu}_s(dy)  \right) ds, \; (\xi, \nu) \in \clx\times \clm_1.
    \label{eq:rate-function-2}
\end{equation}
The following is the second main result of this work.
\begin{theorem}
Suppose Assumption \ref{Assumptions:system} holds. Then $I_2$ is a rate function on $\mathcal{X}\times \clm_1$. Furthermore, $(X^{\eps}, \Lambda^{\eps})$ defined by \eqref{eq:systemor} and \eqref{eq:557} satisfy a LDP on $\mathcal{X}\times \clm_1$ with rate function $I_2$ and speed $(\eps s^2(\eps))^{-1}$.
\label{theorem:ldp-2}
\end{theorem}

\section{A Variational Formula.}
\label{sec:varformula}
In this section we recall a basic variational formula for exponential moments of functionals of finite dimensional Brownian motions that was established in \cite{boudup}. In the form stated below, the result can be found in  \cite[Theorem 8.3]{buddupbook}. 

Let $(\Omega, \mathcal{F}, \Prob,\mathcal{F}_t)_{0 \leq t \leq T}$ be a filtered probability space where the filtration $\{\mathcal{F}_t\}_{0 \leq t \leq T}$ satisfies the usual conditions. For $p \in \mathbb{N}$, denote by $\mathcal{A}^p$ the collection of all $\mathcal{F}_t-$progressively measurable, $\mathbb{R}^p-$valued stochastic processes $\{u(t)\}_{0\leq t \leq T}$ that satisfy $\E \int_0^T \|u(s)\|^2 ds < \infty$. Also for $M \in (0,\infty)$ we denote by 
\begin{equation}\label{eq:954}
	S_M^p \doteq \{ h \in L^2([0,T]: \RR^p): \int_0^T \|h(s)\|^2 ds \leq M\}.
\end{equation}
This space will be equipped with the inherited weak topology on $L^2([0,T]: \RR^p)$, under which it is a compact space.
Let $\mathcal{A}_{b,M}^p = \{ u \in \mathcal{A}^p : u \in  S_M^p, a.s.\}$ and let $\mathcal{A}_b^p = \cup_{M=1}^\infty \mathcal{A}_{b,M}^p$. When clear from the context, we will drop $p$ from the notation in $\mathcal{A}^p,\mathcal{A}_{b,M}^p$, $\cla_b^p$, and $S_M^p$. 

Let $\beta$ be a $p-$dimensional, standard, $\{\mathcal{F}_t\}-$Brownian motion on this filtered probability space. 

\begin{theorem}
\label{Variational-1}
Let $G \in \BM(\mathcal{C}([0,T]:\mathbb{R}^p))$. Then,
    \begin{equation}
    \nonumber
    - \log \E \exp \{ - G(\beta) \} = \inf_{v \in \mathcal{R}} \E \left( G \left(\beta + \int_0^\cdot v(s) ds\right) + \frac{1}{2} \int_0^T \|v(s)\|^2 ds \right)
\end{equation}
where $\mathcal{R}$ can be either $\mathcal{A}$ or $\mathcal{A}_b$.
\end{theorem}

The above theorem will be used in the proofs of both Theorems \ref{theorem:ldp} and \ref{theorem:ldp-2}.  In the first case, $T=1$ and the role of $\beta$ will be played by the Brownian motion $B$ (in particular $p= r$), while in the second case $\beta = (W,B)$ and $p= k+d$.

\section{Proof of Theorem \ref{theorem:ldp}.}
\label{sec:pfthmone}
In order to prove the theorem, we will first show in Section \ref{subsec:single-upper} the LDP upper bound, which in terms of Laplace asymptotics corresponds to the statement: for every $F \in \mathcal{C}_b(\mathcal{P}(\mathbb{R}^d))$
\begin{equation}
\label{ineq-single-upper}
    \liminf_{\eps \rightarrow 0}  - \eps s^2(\eps) \log \E e^{-\frac{F(\mu^\eps)}{\eps s^2(\eps)}}  \geq \inf_{\gamma \in \mathcal{P}(\mathbb{R}^d)} \bigg( F(\gamma) + I_1(\gamma)  \bigg).
\end{equation}
Then, in Section \ref{subsec:single-lower} we will prove the complementary lower bound: for every $F \in \mathcal{C}_b(\mathcal{P}(\mathbb{R}^d))$
\begin{equation}
\label{ineq-single-lower}
    \limsup_{\eps \rightarrow 0}  - \eps s^2(\eps) \log \E e^{-\frac{F(\mu^\eps)}{\eps s^2(\eps)}} \leq \inf_{\gamma \in \mathcal{P}(\mathbb{R}^d)} \bigg( F(\gamma) + I_1(\gamma)  \bigg).
\end{equation}
Finally, in Section \ref{sec:cptlvthm1} we show that the function $I_1$ has compact level sets.
Together these three results will complete the proof of Theorem \ref{theorem:ldp}.

Assumption \ref{assumptions-1} will be taken to hold throughout this Section.

\subsection{LDP Upper Bound}
\label{subsec:single-upper}

In this Section, we prove the inequality in \eqref{ineq-single-upper}.  Fix $F \in \mathcal{C}_b(\mathcal{P}(\mathbb{R}^d))$. 
Note that the coefficients $\psi$ and $\sigma$ are Lipschitz maps and thus the SDE in \eqref{eq:SDE_single} has a unique pathwise solution.
This says that, there exists a measurable map $\tilde{\mathcal{G}}^\eps: \mathcal{C}([0,1]:\mathbb{R}^r) \rightarrow \mathcal{C}([0,1]:\mathbb{R}^d)$ such that $Y^\eps = \tilde{\mathcal{G}}^\eps(B)$ and consequently, there is a measurable map $\mathcal{G}^\eps: \mathcal{C}([0,1]:\mathbb{R}^r) \rightarrow \mathcal{P}(\mathbb{R}^d)$ such that $\mu^\eps = \mathcal{G}^\eps(B)$ where $\mu^{\eps}$ is as in \eqref{eq:occupational-1}.

Fix $\eps > 0$ and apply Theorem \ref{Variational-1} with $p = r$, $\beta = B$, and $G$ replaced by $G^\eps = F \circ \mathcal{G}^\eps$. Then, we have

\begin{equation}
\label{eq:238}
\begin{aligned}
            - \eps s^2(\eps) \log \E e^{-\frac{F(\mu^\eps)}{\eps s^2(\eps)}} &= - \eps s^2(\eps) \log \E e^{-\frac{G^\eps(B)}{\eps s^2(\eps)}} \\
           & = \inf_{v \in \mathcal{A}_b} \E \left[ \frac{1}{2}\eps s^2 (\eps) \int_0^1 \|v(s)\|^2 ds + G^\eps \left( B + \int_0^\cdot v(s)ds \right)   \right] \\
            &= \inf_{v \in \mathcal{A}_b} \E \left[ \frac{1}{2}\int_0^1 \|v(s)\|^2 ds + G^\eps \left( B + \frac{1}{\sqrt{\eps} s(\eps)}\int_0^\cdot v(s)ds \right)   \right].
\end{aligned}
\end{equation}

Fix $\delta > 0$ and choose for each $\eps > 0$ a $\tilde{v}^\eps \in \mathcal{A}_b$ that is $\delta-$optimal for the right side. Then, for all $\eps>0$

\begin{equation}
    \nonumber
    - \eps s^2(\eps) \log \E e^{-\frac{F(\mu^\eps)}{\eps s^2(\eps)}} 
    \geq \E \left[ \frac{1}{2}\int_0^1 \|\tilde{v}^{\eps}(s)\|^2 ds + G^\eps \left( B + \frac{1}{\sqrt{\eps} s(\eps)}\int_0^\cdot \tilde{v}^{\eps}(s)ds \right)   \right] - \delta.
\end{equation}

Since $F$ is bounded, by a standard localization argument (see  \cite[Theorem 3.17 ]{buddupbook}) it follows that there is a $M \in (0,\infty)$, and for each $\eps >0, v^\eps \in \mathcal{A}_{b,M}$ such that 
\begin{equation}
    \nonumber
              - \eps s^2(\eps) \log \E e^{-\frac{F(\mu^\eps)}{\eps s^2(\eps)}} 
              \geq \E \left[ \frac{1}{2}\int_0^1 \|{v}^{\eps}(s)\|^2 ds + G^\eps \left( B + \frac{1}{\sqrt{\eps} s(\eps)}\int_0^\cdot {v}^{\eps}(s)ds \right)   \right] - 2 \delta.   
\end{equation}

Also by an application of Girsanov's Theorem, it is easy to see that $G^\eps(B + \frac{1}{\sqrt{\eps}s(\eps)} \int_0^\cdot v^\eps(s) ds) = \Bar{\mu}^\eps$, a.s., where 
\begin{equation}
        \Bar{\mu}^\eps(A) = \int_0^1 1_A(\Bar{Y}^\eps(s))ds, \quad A \in \mathcal{B}( \mathbb{R}^d)
    \label{eq:occupational-bar-1}
\end{equation}
and $\Bar{Y}^\eps$ solves 

\begin{equation}
    d\Bar{Y}^\eps (t) = -\frac{1}{\eps} \psi(\Bar{Y}^\eps (t))dt  + \frac{s(\eps)}{\sqrt{\eps}} \sigma(\Bar{Y}^\eps (t)) dB(t) +\frac{1}{\eps} \sigma(\Bar{Y}^\eps (t)) v^\eps(t) dt, \quad \Bar{Y}^\eps (0) = y^0.
    \label{eq:Y-bar}
\end{equation}
In particular,
\begin{equation}\label{eq:438}
              - \eps s^2(\eps) \log \E e^{-\frac{F(\mu^\eps)}{\eps s^2(\eps)}} 
              \geq \E \left[ \frac{1}{2}\int_0^1 \|{v}^{\eps}(s)\|^2 ds + F(\bar\mu^\eps)  \right] - 2 \delta .
\end{equation}
We begin with the following moment estimate.
\begin{lemma}
\label{lemma-single-exp-Y}
 We have that 
 \begin{equation}
     \nonumber
     \sup_{\eps} \E \int_0^1 \|\Bar{Y}^\eps(s)\|^2 ds < \infty, \quad \text{and} \quad \sup_\eps \sup_{0\le s \le 1} \eps \E\|\Bar{Y}^\eps(s)\|^2  < \infty.
 \end{equation}
\end{lemma}

\begin{proof}
We note that from Assumption \ref{assumptions-1} parts 2 and 3, $\psi = \clv_0$.
Using this fact and applying It\^{o}'s lemma to $V_0(\Bar{Y}^\eps (t))$ we obtain, for $0\le t \le 1$
\begin{equation}
\label{single-upper-ito}
\begin{gathered}
    V_0(\Bar{Y}^\eps(t)) = V_0(y_0)  - 
    \frac{1}{\eps} \int_0^t \nabla V_0(\Bar{Y}^\eps(s)) \cdot \left(\mathcal{V}_0(\Bar{Y}^\eps(s)) - \sigma(\Bar{Y}^\eps(s)) v^\eps(s)) \right) ds \\
    + \frac{s(\eps)}{\sqrt{\eps}}  \int_0^t  \nabla V_0 (\Bar{Y}^\eps(s))^T \sigma(\Bar{Y}^\eps(s))  dB(s) 
+ \frac{s^2(\eps)}{2\eps} \int_0^t \tr( [\sigma^T \mathcal{H} V_0  \sigma](\Bar{Y}^\eps(s))) ds.   \end{gathered}
\end{equation} 
Let, for $m \in \NN$, $\tau_m = \inf\{t: \Bar{Y}^\eps(t) \geq m\}$. Taking expectations and rearranging terms we obtain
\begin{equation}
    \begin{aligned}
       \E \int_0^{ t\wedge \tau_m} \nabla V_0(\Bar{Y}^\eps(s)) \cdot \mathcal{V}_0(\Bar{Y}^\eps(s)) ds &\le   \eps V_0 (y_0)  
        +  \E \int_0^{t\wedge \tau_m} \nabla V_0 (\Bar{Y}^\eps(s))^T  \sigma (\Bar{Y}^\eps(s)) v^\eps(s)ds\\
		&  +  \frac{s^2(\eps)}{2}  \E \int_0^{t \wedge \tau_m} \tr( [\sigma^T \mathcal{H} V_0  \sigma](\Bar{Y}^\eps(s)))ds,
    \end{aligned}
    \label{eq:phi-tilde-bound}
\end{equation} 
where we have used the nonnegativity of $V_0$.

With $c_1(\cdot), c_2(\cdot)$ as in  Assumption \ref{assumptions-1} (part 4), we have
\begin{equation}
\nonumber
    \begin{gathered}
            c_1(0) \E \int_0^{t \wedge \tau_m} \|\Bar{Y}^\eps(s)\|^2 ds \leq  \E \int_0^{t \wedge \tau_m} \nabla V_0(\Bar{Y}^\eps(s)) \cdot \mathcal{V}_0 (\Bar{Y}^\eps(s)) ds + c_2(0) \\
            \leq c_2(0) + \eps V_0(y_0)  + \|\sigma\|_\infty \E \int_0^1 \|\nabla V_0(\Bar{Y}^\eps(s))\| \|v^\eps(s)\| ds + \frac{s^2(\eps)r}{2} \|\sigma\|_\infty^2 \|\mathcal{H} V_0\|_\infty.
    \end{gathered}
\end{equation}

Using the linear growth of $\nabla V_0$ (which follows from $\|\clh V_0\|_{\infty} <\infty$) and Young's inequality, we can find $\kappa_1 \in (0,\infty)$ such that for all $\eps > 0$
\begin{equation}
\label{single-lemma-exp-Ito}
  \E \int_0^1 \|\nabla V_0(\Bar{Y}^\eps(s))\| \|v^\eps(s)\| ds \leq \frac{c_1(0)}{2(\|\sigma\|_\infty+1)} \E \int_0^1 (1+\|\Bar{Y}^\eps(s)\|^2) ds + \kappa_1 \E \int_0^1 \|v^\eps(s)\|^2 ds .
\end{equation}
Then,  by sending $m \rightarrow \infty$ in the previous display, and recalling that $v^{\eps} \in \cla_{b,M}$, we have,
\begin{equation}
    \label{eq:130}
    \frac{c_1(0)}{2} \E \int_0^1 \|\Bar{Y}^\eps(s)\|^2 ds \leq \kappa_2
	\end{equation}
where $\kappa_2 = c_2(0) +c_1(0) + \kappa_1 M + r\|\sigma\|^2_\infty \|\mathcal{H}V_0\|_\infty
+ V_0(y_0)$. This proves the first statement in the lemma.

Finally, using Assumption \ref{assumptions-1} (part 4) in \eqref{single-upper-ito} again and taking expectations, we have, for $0\le t \le 1$,
\begin{equation}
\begin{aligned}
\alpha_2(0) \eps E \|\Bar{Y}^\eps(t)\|^2	 &\le
\eps E V_0(\Bar{Y}^\eps(t)) \le \eps V_0(y_0) + c_2(0) + \frac{r}{2}\|\sigma\|^2_\infty \|\mathcal{H}V_0\|_\infty\\
& + \|\sigma\|_\infty E\int_0^1 \|\nabla V_0(\Bar{Y}^\eps(s))\| \|v^\eps(s)\| ds,
\end{aligned}
\end{equation}
where we have used the observation that the expected value of the stochastic integral in \eqref{single-upper-ito} is $0$ in view of \eqref{eq:130} and linear growth of $\nabla V_0$.
The second statement in the lemma is now immediate from \eqref{single-lemma-exp-Ito},\eqref{eq:130} and on  recalling that $v^{\eps} \in \cla_{b,M}$.
\end{proof}

We now introduce certain occupation measures which will play an important role in the proof of the upper bound.
For $\eps>0$, define a $\mathcal{P}(\mathbb{R}^{d+r} )-$valued random variable $\Bar{Q}^\eps$ as
\begin{equation}
    \Bar{Q}^\eps (A \times B) = \int_0^1 1_A (\Bar{Y}^\eps(s)) 1_B( v^{\eps}(s)) ds, \, A\in \mathcal{B}( \mathbb{R}^d), B\in \mathcal{B}( \mathbb{R}^r)
    \label{eq:occupational-2}
\end{equation}
The following proposition gives the tightness of the collection $\{\Bar{Q}^\eps\}$.
\begin{proposition}
\label{single-prop-tightness-Q}
The family of $\mathcal{P}(\mathbb{R}^{d+r})-$valued random variables $\{\Bar{Q}^\eps\}$ defined in \eqref{eq:occupational-2} is tight. Furthermore, 
\begin{equation}\label{eq:350}
	\sup_{\eps} \E \int_{\mathbb{R}^{d+r} } (\|y\|^2 + \|z\|^2) \Bar{Q}^\eps(dy\, dz) < \infty .
\end{equation}
\end{proposition}

\begin{proof}
	Define $q^{\eps} \in \clp(\RR^{d+r})$ as \PZ{$q^{\eps}(C) \doteq E(\Bar{Q}^\eps(C))$, for $C \in \clb(\RR^{d+r})$}.
	It suffices to prove that
	\begin{equation}\label{eq:149}
		\sup_{\eps}\int_{\RR^{d+r}} (\|y\|^2 + \|z\|^2) q^\eps(dy\, dz) < \infty.
	\end{equation}
Note that
$$\int_{\RR^{d+r}} (\|y\|^2 + \|z\|^2) q^\eps(dy\, dz) = E \int_0^1 \|\Bar{Y}^{\eps}(s)\|^2 ds  + E \int_0^1 \|v^{\eps}(s)\|^2 ds.$$
The estimate in \eqref{eq:149} now follows from Lemma \ref{lemma-single-exp-Y} and on using the fact that
$v^\eps \in \mathcal{A}_{b,M}$ for every $\eps \in (0,1)$.
\end{proof}
The next step will be to give a suitable characterization of the weak limit points of $\bar Q^{\eps}$. In order for that we present the following approximation lemma which will also be used in the proof of Theorem \ref{theorem:ldp-2}.
\begin{lemma}
\label{single-upper-lemma-orthogonality-extension}
Let $f: \mathbb{R}^d \rightarrow \mathbb{R}$ be a $\mathcal{C}^2$ function such that $\|\mathcal{H} f\|_\infty < \infty$. Then, there exists a sequence $\{f_M\}_{M \in \mathbb{N}}$ of $\mathcal{C}_c^2$ functions from $\mathbb{R}^d$ to $\mathbb{R}$ such that 

\begin{enumerate}
    \item for each $M \in \mathbb{N}$, $f_M(x) = f(x)$ for all $\|x\| \leq M$.
    \item $\sup_{M \in \mathbb{N}} \sup_{x\in \mathbb{R}^d} \frac{\|\nabla f_M(x)\|}{1 + \|x\|} < \infty$.
\end{enumerate}
\end{lemma}

\begin{proof}
Define for $r \in \NN$, $\tilde g_r : \mathbb{R} \to \mathbb{R}$ as
\[ \tilde g_r(y) = \begin{cases} 
      1 & y\leq r \\
      \frac{2r-y}{r} & r < y\leq 2r \\
      0 & 2r < y
   \end{cases}
\]
We now suitably mollify the above piecewise smooth function. Let $\xi: \RR \to \RR$ be defined as
\[ \xi(y) = \begin{cases} 
      c e^{-\frac{1}{1 - y^2}} & |y|< 1 \\
      0 &  |y| \geq 1
   \end{cases}, 
\]
where $c$ is a normalization constant that makes $\xi$ a probability density. 
Now define for $M \in \NN$ 
\begin{equation}
    g_M (y) = (\tilde g_{M+1} * \xi) (y) = \int_\mathbb{R} \tilde g_{M+1} (y -x) \xi (x) dx = \int_\mathbb{R} \tilde g_{M+1} (x) \xi (y-x) dx.
    \label{eq:smoothed}
\end{equation}
Note that $g_M(y) = 1$  for $y \le M$, $g_M(y)=0$ for $y\ge 2M+3$ and $g_M(y) \in [0,1]$ for all $y \in \RR$. Furthermore $g_M$ is smooth and 
\begin{equation}
	|g'_M(y)| \le \frac{1}{M},\, \mbox{ for } y \in \RR.
\end{equation}
This in particular says that, for all $x \in \RR^d$
\begin{equation}
	\|\nabla(g_M(\|x\|))\| = \left \| g'_M(\|x\|) \frac{x}{\|x\|}\right\| \le \frac{1}{M}.
\end{equation}
Define 
$$f_M(x) \doteq f(x) g_M(\|x\|), \; x \in \RR^d.$$
	Then note that $f_M$ is in $C^2_c$ and $f_M(x)=f(x)$ for $\|x\| \le M$. 
	Also, since $\sup_{x \in \RR^d} |\clh f(x)| < \infty$, there is a $C_1 \in (0, \infty)$ (independent of $M$) such that
	$$ |f(x)| \le C_1(1+\|x\|^2), \;\;  \|\nabla f(x)\| \le C_1(1+\|x\|) \; \mbox{ for all } x \in \RR^d.$$
	Finally note that, for $\|x\| < M$
	$$\|\nabla f_M(x)\| = \|\nabla f(x)\| \le C_1(1+\|x\|)$$
	and for $\|x\| \ge M$
	\begin{align*}
		\|\nabla f_M(x)\| &=  \left\|g_M(\|x\|)\nabla f(x) + f(x) \nabla g_M(\|x\|)\right\|\\
		&\le \|\nabla f(x)\| + \sup_{\|x\| \le 2M+3} \frac{|f(x)|}{M}\\
		& \le C_1(1+\|x\|) + C_1 \frac{1+(2M+3)^2}{M} \le 23 C_1(1+\|x\|).
	\end{align*}
	The result follows. 
\end{proof}
We now proceed with obtaining a characterization for the weak limit points of $\{\Bar{Q}^\eps\}$.
Recall from Proposition \ref{single-prop-tightness-Q} that this collection is tight.

\begin{lemma}\label{lem:ineqincos}
	Let $\Bar{Q}$ be a weak limit point of $\{\Bar{Q}^\eps\}$. Then, a.s.,
	$$\int_{\RR^{d+r}}  \|z\|^2 \Bar{Q}(dy\, dz) \ge \int_{\RR^{d+r}}  \|\sigma^T(y) \nabla \phi(y)\|^2 \Bar{Q}(dy\, dz).$$
	
\end{lemma}
\begin{proof}
Let $\eta \in C_c^2(\RR^d)$. Then by It\^{o}'s formula
\begin{align*}
	\eta(\Bar{Y}^{\veps}(1)) - \eta(y_0) &= - \frac{1}{\veps}\int_0^1 \left[\psi(\Bar{Y}^{\veps}(s)) - 
	\sigma(\Bar{Y}^{\veps}(s))v^{\veps}(s)\right]\cdot\nabla \eta(\Bar{Y}^{\veps}(s)) ds\\
	&+ \frac{s(\veps) }{\veps^{1/2}}\int_0^1 \nabla \eta(\Bar{Y}^{\veps}(s))^T \sigma(
	\Bar{Y}^{\veps}(s))dB(s)
	+ \frac{s^2(\veps) }{2\veps} \int_0^1 \tr[\sigma \sigma^T\clh \eta](\Bar{Y}^{\veps}(s)) ds.
\end{align*}
Multiplying with $\veps$ in the above equation
\begin{align*}
	\veps\left[\eta(\Bar{Y}^{\veps}(1)) - \eta(y_0)\right] &= - \int_0^1 \left[\psi(\bar Y^{\veps}(s)) - 
	\sigma(\Bar{Y}^{\veps}(s))v^{\veps}(s)\right]\cdot \nabla \eta(\Bar{Y}^{\veps}(s)) ds\\
	&+ s(\veps)\veps^{1/2} \int_0^1 \nabla \eta(\bar Y^{\veps}(s))^T \sigma(\bar Y^{\veps}(s))dB(s)
	+ \frac{1}{2}s^2(\veps) \int_0^1 \tr[\sigma \sigma^T\clh\eta](\Bar{Y}^{\veps}(s)) ds.
\end{align*}
Sending $\veps $ to $0$  and using the fact that $\eta$, $\nabla\eta$
and $\clh \eta$ are bounded (as $\eta$ has compact support), we have, in probability,
\begin{equation}
	\lim_{\eps \to 0}\int_{\RR^{d+r}}\left[\psi(y) - \sigma(y)z\right]\cdot \nabla \eta(y) \Bar{Q}^{\veps}(dy\, dz) =
	\lim_{\eps \to 0}\int_0^1 \left[\psi(\bar Y^{\veps}(s)) - 
		\sigma(\Bar{Y}^{\veps}(s))v^{\veps}(s)\right]\cdot \nabla \eta(\Bar{Y}^{\veps}(s)) ds = 0.
\end{equation}

We relabel the subsequence along which $\Bar{Q}^\eps \Rightarrow \Bar{Q}$, as $\Bar{Q}^\eps$.
Then from the square integrability in \eqref{eq:350}, and since  $\eta$ has compact support we have that as $\eps \to 0$,
\begin{equation}
	\lim_{\veps \to 0} \int_{\RR^{d+r}} \left[\psi(y) - \sigma(y)z\right]\cdot \nabla \eta(y) \Bar{Q}^{\veps}(dy\, dz)
	= \int_{\RR^{d+r}} \left[\psi(y) - \sigma(y)z\right]\cdot \nabla \eta(y) \Bar{Q}(dy\, dz).
\end{equation}
Combining the above two displays we have, a.s.,

\begin{align}
0 
&= \int_{\RR^{d+r}} \left[\psi(y) - \sigma(y)z\right]\cdot \nabla \eta(y) \Bar{Q}(dy\, dz).\label{eq:e8}
\end{align}
Disintegrate $\Bar{Q}$ as $\Bar{Q}(dy\, dz) = q(y,\, dz) \hat Q(dy)$, \PZ{where $\hat{Q}(A) \doteq \Bar{Q}(A \times \RR^r), A \in \clb(\RR^d)$ is the first marginal of $\Bar{Q}$ and $q$ is the regular conditional probability distribution (r.c.p.d.) on the second coordinate given the first coordinate.
} Now define
$$u(y)= \int_{\RR^r} z \, q(y,\, dz).$$
Note that the above integral is well defined a.s. for $\hat Q$ a.e. $y$, since from \eqref{eq:350} and Fatou's lemma 
\begin{equation}\label{eq:408}
	\int_{\mathbb{R}^{d+r} } (\|y\|^2 + \|z\|^2) \Bar{Q}(dy\, dz) < \infty \mbox{ a.s.}\end{equation}
Then, from \eqref{eq:e8}, and a standard separability argument, a.s., for all $\eta \in C_c^2(\RR^d)$,
\begin{equation}
	0= \int_{\mathbb{R}^{d} }  [\psi(y)- \sigma(y)u(y)]\cdot \nabla \eta(y) \hat Q(dy).\label{eq:etaisze}\end{equation}
We now argue that, although $\phi$ does not have compact support, we can replace $\eta$ by $\phi$ in the above identity, namely,
\begin{equation}
	0= \int_{\mathbb{R}^{d} }  [\psi(y)- \sigma(y)u(y)]\cdot \nabla \phi(y) \hat Q(dy). \label{eq:eq321}
\end{equation}
From Lemma \ref{single-upper-lemma-orthogonality-extension} there exists a sequence $\eta_M$ of functions in $\clc_c^2(\RR^d)$ such that
$\eta_M(y)= \phi(y)$ for all $\|y\|\le M$ and
\begin{equation}\label{eq:unifbdonnaeta}
	\sup_{M\in \NN} \frac{\|\nabla \eta_M(y)\|}{1+\|y\|} \doteq C_1 <\infty .
\end{equation}
Using the linear growth of $\nabla\phi$, and the boundedness of $a, \sigma$, we have for some $\kappa_1 \in (0, \infty)$
\begin{align}
	 \int_{\mathbb{R}^{d} }  |(\psi(y)- \sigma(y)u(y))\cdot\nabla \phi(y)| \hat Q(dy)
	 &\le \kappa_1 \int_{\mathbb{R}^{d} }  (1+ \|y\|^2 + \|u(y)\|^2) \hat Q(dy)\nonumber\\
	 &\le \kappa_1 \int_{\mathbb{R}^{d+r} }  (1+ \|y\|^2 + \|z\|^2)  \Bar{Q}(dy\, dz) < \infty\label{eq:333}
\end{align}
where we have used \eqref{eq:408}, the definition of $u(y)$, and Jensen's inequality.
Now
\begin{align}
\int_{\mathbb{R}^{d} } [\psi(y)- \sigma(y)u(y)]\cdot \nabla \phi(y) \hat Q(dy) &=
\int_{\|y\|<M} [\psi(y)- \sigma(y)u(y)]\cdot \nabla \phi(y) \hat Q(dy) \nonumber\\
&\quad + \int_{\|y\|\ge M} [\psi(y)- \sigma(y)u(y)]\cdot\nabla \phi(y) \hat Q(dy)\nonumber\\
&= \int [\psi(y)- \sigma(y)u(y)]\cdot\nabla \eta_M(y) \hat Q(dy)\nonumber\\
&\quad + \int_{\|y\|\ge M} [\psi(y)- \sigma(y)u(y)]\cdot(\nabla \phi(y) - \nabla \eta_M(y)) \hat Q(dy)\nonumber\\
&=\int_{\|y\|\ge M} [\psi(y)- \sigma(y)u(y)]\cdot(\nabla \phi(y) - \nabla \eta_M(y)) \hat Q(dy) \label{eq:849n}
\end{align}
where the last line is from \eqref{eq:etaisze} applied with $\eta = \eta_M$.
The last term converges to $0$ as $M \to \infty$, since  as in \eqref{eq:333} we have that, for some $\kappa_2 \in (0,\infty)$,
$$\int_{\|y\|\ge M} |(\psi(y)- \sigma(y)u(y))(\nabla \phi(y) - \nabla \eta_M(y))| \hat Q(dy)
\le \kappa_2 \int_{\|y\|\ge M} (1+ \|y\|^2 + \|u(y)\|^2) \hat Q(dy),$$
and the last term, due to \eqref{eq:333}, converges to $0$ as $M\to \infty$.
 This proves \eqref{eq:eq321}.
Finally note that
\begin{align*}
	\int_{\mathbb{R}^{d+r} }  \|z\|^2 \Bar{Q}(dy\, dz) &= \int_{\mathbb{R}^{d+r} }  \|z\|^2 q(y\, dz) \hat Q(dy)
	 \ge \int_{\mathbb{R}^{d} }  \|u(y)\|^2 \hat Q(dy)\\
	&= \int_{\mathbb{R}^{d} }  \|u(y) - \sigma^T(y) \nabla \phi(y)\|^2 \hat Q(dy)
	+ \int_{\mathbb{R}^{d} }  \|\sigma^T(y) \nabla \phi(y)\|^2 \hat Q(dy)\\
	&\quad + 2\int_{\mathbb{R}^{d} }  (u(y) - \sigma^T(y) \cdot \nabla \phi(y))\cdot (\sigma^T(y) \nabla \phi(y)) \hat Q(dy).
\end{align*}
The last term equals $0$ from \eqref{eq:eq321} since
$$\int_{\mathbb{R}^{d} }  (u(y) - \sigma^T(y) \nabla \phi(y))\cdot (\sigma^T(y) \nabla \phi(y)) \hat Q(dy)
= -\int_{\mathbb{R}^{d} }  [\psi(y)- \sigma(y)u(y)]\cdot \nabla \phi(y) \hat Q(dy).$$
This completes the proof of the lemma. 
\end{proof}

\subsubsection{Proof of the LDP upper bound.}
 We now complete the proof of the upper bound, namely of the inequality in \eqref{ineq-single-upper}.
 Denote by $[\Bar{Q}^{\veps}]_i$, $i=1,2$ the two marginals of $\Bar{Q}^{\veps}$ on $\RR^d$ and $\RR^r$ respectively. Let $\Bar{Q}$ be a weak limit point of $\Bar{Q}^{\veps}$. By a usual subsequential argument we can assume that the convergence $\Bar{Q}^{\veps} \Rightarrow \Bar{Q}$ holds along the full sequence.
Note that
\begin{align*}
	\liminf_{\veps \to 0}- s^2(\veps)\veps E \left[ \exp \left\{-\frac{1}{s^2(\veps)\veps} F(\mu^{\veps})\right\} \right]
	& \ge \liminf_{\veps \to 0} E \left [ F(\bar \mu^{\veps}) + \frac{1}{2}\int_0^1 \|v^{\veps}(t)\|^2 dt\right] - 2\delta\\
	& = \liminf_{\veps \to 0} E \left [ F([Q^{\veps}]_1) + \frac{1}{2}\int_{\RR^{d+r}} \|z\|^2 \Bar{Q}^{\veps}(dy\,dz)\right] - 2\delta\\
	&\ge E \left [ F(\hat Q) + \frac{1}{2}\int_{\RR^{d+r}} \|z\|^2  \Bar{Q}(dy\, dz)\right] - 2\delta\\
	&\ge E \left [ F(\hat Q) + \frac{1}{2}\int_{\RR^d} \|\sigma^T(y) \nabla \phi(y)\|^2 \hat Q(dy)\right] - 2\delta\\
	&= E \left [ F(\hat Q) + I_1(\hat Q)\right]- 2\delta\\
	&\ge \inf_{\gamma \in \calP(\RR^d)} \left [ F(\gamma) + I_1(\gamma)\right]- 2\delta.
\end{align*}
where the first inequality is from \eqref{eq:438},  the second line uses the definition of $\bar{Q}^{\eps}$, the third line uses the convergence of $\Bar{Q}^{\eps}$ to $\Bar{Q}$, the  lower semicontinutiy  of the $L^2$ norm and Fatou's lemma, and the decomposition $\Bar{Q}(dy\,dz) = \hat Q(dy) q(y, dz)$, the fourth line uses
Lemma \ref{lem:ineqincos}, and the fifth uses the definition of $I_1$. Since $\delta>0$ is arbitrary, the proof of the Laplace upper bound is complete. \hfill \qed

\subsection{Construction of a  Stabilizing Control}

In proving the lower bound we will need to construct certain controlled versions of \eqref{eq:SDE_single} that stay in the neighborhood of a specified state in $\RR^d$ for a given length of time. The following lemma will be key in such constructions.

\begin{lemma} \label{theorem:auxillary}
Fix $x \in \mathbb{R}^d$ and a \PZ{collection of points} $\{ x^\eps \}$ in $\mathbb{R}^d$. Let $\tilde{Y}^\eps (s) = Y^\eps (s,x,x^\eps)$ solve the equation
\[d\tilde{Y}^\eps (t) = -\frac{1}{\eps} \mathcal{V}_x (\tilde{Y}^\eps(t))dt + \frac{s(\eps)}{\sqrt{\eps}} \sigma(\tilde{Y}^\eps (t)+x) dB(t), \quad \tilde{Y}^\eps(0) = x^{\eps},\]
where $\mathcal{V}_x (y)$ is as in Assumption \ref{assumptions-1}(3). Then, the following hold.
\begin{enumerate}
    \item There exists  $\kappa_1 = \kappa_1(x) \in (0,\infty)$ such  that $\sup_{0 \leq s \leq 1} \E \|\tilde{Y}^\eps(s)\|^2 \leq \kappa_1 (1 + \|x^\eps\|)^2$ for all $\eps$.
    \item Suppose  $\sup_{\eps} \eps \|x^\eps\|^2 < \infty$. Then, for every $\kappa \in (0,1]$ \begin{equation} \sup_{t \in [\kappa,1]} \E \dbl \left(\frac{1}{t} \int_0^t \delta_{\tilde{Y}^\eps(s)} ds, \delta_0\right) \rightarrow 0, \quad \text{as} \quad \eps \rightarrow 0
\label{eq:auxilary-converg}
\end{equation}
\end{enumerate} 
\end{lemma}

%
%

\begin{proof}
Let $x \in \mathbb{R}^d$ be as in the statement of the theorem. In the following, dependence of various bounds on $x$ will not be noted explicitly. Let for $y \in \mathbb{R}^d, \sigma_x(y)\doteq \sigma(x+y)$. Applying It\^{o}'s formula to $V_x(\tilde{Y}^\eps(t))$, where $V_x$ is as in Assumption \ref{assumptions-1}(4), we have
\begin{equation}
\begin{aligned}
   V_x (\tilde{Y}^\eps(t)) &= V_x (x^\eps) - \frac{1}{\eps} \int_0^t \nabla V_x(\tilde{Y}^\eps (s)) \cdot \mathcal{V}_x (\tilde{Y}^\eps (s)) ds \\
   &+ \frac{s^2(\eps)}{2 \eps} \int_0^t \tr([\sigma_x^T\mathcal{H} V_x \sigma_x](\tilde{Y}^\eps (s))) ds + \frac{s(\eps)}{\sqrt{\eps}} \int_0^t  \nabla V_x (\tilde{Y}^\eps (s))^T  \sigma_x(\tilde{Y}^\eps (s)) dB(s).
  \end{aligned}
 \label{eq:Ito-phi}
\end{equation}

For $m > 0$ let $\tau_m = \inf\{t \ge 0: \|\tilde{Y}^\eps(t)\| \geq m\}$. Since by assumption  $V_x$ is nonnegative, 
\begin{equation}
       \E  \int_0^{t \wedge \tau_m} \nabla V_x(\tilde{Y}^\eps (s)) \cdot \mathcal{V}_x (\tilde{Y}^\eps (s)) ds 
       \leq \eps V_x(x^\eps)  + \frac{s^2(\eps)}{2} \E  \int_0^{t \wedge \tau_m} \tr([\sigma_x^T\mathcal{H} V_x \sigma_x](\tilde{Y}^\eps (s))) ds.
\end{equation}

Using Assumption \ref{assumptions-1}(4) again and sending $m \rightarrow \infty$, we have that for some $\kappa_2 \in (0,\infty)$,
\begin{equation}\label{eq:1055}
    \E \int_0^1 \|\tilde{Y}^\eps(s)\|^2 ds \leq \kappa_2(\eps \|x^\eps\|^2 + 1)
\end{equation}
for all $\eps>0$. This estimate together with the linear growth of $\nabla V_x$ says in particular that the expected value of the stochastic integral in \eqref{eq:Ito-phi} is $0$. From \eqref{eq:Ito-phi} and Assumption \ref{assumptions-1}(4) we also see that for some $a_i \in (0,\infty), i = 1, 2, 3,4$, and for all $t \in [0,1]$
\begin{equation}
    \label{eq:546n}
    \begin{gathered}
       a_1 \E \|\tilde{Y}^\eps(t)\|^2  \leq \E V_x (\tilde{Y}^\eps(t)) \leq a_2( 1 + \|x^\eps\|^2) - \frac{1}{\eps} \int_0^t \left( a_3\E \|\tilde{Y}^\eps(s)\|^2 -a_4\right) ds.
    \end{gathered}
\end{equation}
A similar argument shows that for $0 \le s \le t \le 1$ and with $b^\eps(t) = \E \|\tilde{Y}^\eps(t)\|^2$
\begin{equation}\label{eq:932}
    a_1 b^\eps(t) \leq a_2 (1 + b^\eps(s)) - \frac{1}{\eps} \int_s^t \left( a_3 b^\eps(u)- a_4  \right) du.
\end{equation}

Define $\kappa_1 \doteq  1\vee \frac{4 a_2}{a_1} \vee \frac{2a_2}{a_1}(\frac{a_4}{a_3} + 1)$. We claim that $b^\eps(t) \leq \kappa_1 (1 + \|x^\eps\|^2)$ for all $t \in [0,1]$ and $\eps \in (0,1)$. Suppose that the claim is false, then there is an $\eps \in (0,1)$ and a $t_0 \in [0,1]$ such that $b^\eps(t_0) = \kappa_1(1 + \|x^\eps\|^2)$. Let $k_1 \overset{.}{=} \frac{a_1 \kappa_1}{2a_2} - 1$. Note that by our choice of $\kappa_1$, $k_1\ge 1$. We can find some $s_0 \in [0,t_0)$ so that $b^\eps(s_0) = k_1 (1 +\|x^\eps\|^2)$ and $b^\eps(s) \geq k_1 (1 + \|x^\eps\|^2)$ for all $s \in [s_0,t_0]$. Then, from \eqref{eq:932},
\begin{equation*}
\begin{split}
a_1 \kappa_1 (1 + \|x^\eps\|^2) = a_1 b^\eps (t_0) &\leq a_2 (1 + k_1 (1 + \|x^\eps\|^2)) -\frac{1}{\eps} \int_{s_0}^{t_0} \left(a_3 b^\eps(s) - a_4 \right) ds \\
&\leq a_2 (1 + k_1(1 + \|x^\eps\|^2)) - \frac{1}{\eps} \int_{s_0}^{t_0} \left( a_3 k_1 (1 + \|x^\eps\|^2) - a_4 \right) ds \\
&\leq a_2 (k_1 + 1)(1 + \|x^\eps\|^2) - \frac{1}{\eps} \int_{s_0}^{t_0} \left( a_3 k_1 (1 + \|x^\eps\|^2) - a_4 \right) ds.
\end{split}
\end{equation*}
 By our choice of $k_1$ we see that $\left( a_3 k_1  - a_4 \right)  \geq 0$ and so, from the above display, and using the definition of $\kappa_1$ and $k_1$,
 $$a_1 \kappa_1 (1 + \|x^\eps\|^2) \leq a_2 (k_1 + 1) (1 + \|x^\eps\|^2) = \frac{a_1 \kappa_1}{2}(1 + \|x^\eps\|^2)$$
  which is clearly false. This gives us a contradiction, which proves the claim and completes the proof of part (1) of the proposition.  

Now we prove part (2) of the proposition. Fix $\kappa \in (0,1)$ and a collection $\{x^\eps\} \subset \mathbb{R}^d$ such that $\sup_{\eps} \eps \|x^\eps\|^2 < \infty$. 
Let, for $0 \leq t \leq 1$,
\begin{equation}
    \theta^\eps_t \doteq \frac{1}{t} \int_0^t \delta_{\tilde{Y}^\eps(s)} ds.
    \label{def:occupational-tilde}
\end{equation}
We argue by contradiction. Suppose that \eqref{eq:auxilary-converg} is false. Then, there exist some $\delta > 0,\kappa >0$, a sequence $\eps_n \rightarrow 0$, and $t_n \in [\kappa,1]$ such that for every $n \geq 1$, with $\pi^n \doteq \theta^{\eps_n}_{t_n}$
\begin{equation}\label{eq:145}
     \E \dbl ( \pi^n, \delta_0) \geq \delta.
\end{equation}
Using  \eqref{eq:1055} and our assumption on $\{x^\eps\}$ we see that $\{\pi^n\}$ is tight as a sequence of $\mathcal{P}(\mathbb{R}^d)$-valued random variables. Then, along a subsequence (labeled again by $n$), $\pi^n \Rightarrow \pi$, for some $\mathcal{P}(\mathbb{R}^d)$-valued random variable $\pi$.  Along the lines of Lemma \ref{lem:ineqincos}  we now see that, for every $\eta: \RR^d \to \RR$ in $\clc_b^2$, a.s.,
\begin{equation}\label{eq:552n}
    \int_{\mathbb{R}^d} \nabla \eta (y) \clv_x (y) \pi (dy) = 0.
\end{equation}
Indeed,  by It\^{o}'s formula, we have for all $n \in \mathbb{N}$,
\begin{equation}
\begin{aligned}
            \eta (\tilde{Y}^{\eps_n} (t_n)) &= \eta ( x^{\eps_n}) - \frac{1}{\eps_n}  \int_0^{t_n} \nabla \eta (\tilde{Y}^{\eps_n} (s)) \mathcal{V}_x (\tilde{Y}^{\eps_n} (s)) ds  \\
    &+ \frac{s^2 (\eps_n)}{2\eps_n} \int_0^{t_n} \tr([\sigma_x^T \mathcal{H} \eta \sigma_x](\tilde{Y}^{\eps_n} (s)))  ds + \frac{s(\eps_n)}{\sqrt{\eps_n}} \int_0^{t_n}  \sigma_x^T(\tilde{Y}^{\eps_n} (s)) \nabla \eta (\tilde{Y}^{\eps_n} (s)) dB(s).
\end{aligned}
\end{equation}    
Rearranging the terms,
\begin{equation}
\begin{aligned}
   \int_0^{t_n} \nabla \eta (\tilde{Y}^{\eps_n} (s)) \mathcal{V}_x (\tilde{Y}^{\eps_n} (s)) ds &= \eps_n \eta(x^{\eps_n}) - \eps_n \eta (\tilde{Y}^{\eps_n} (t_n)) 
   + \frac{s^2 (\eps_n)}{2} \int_0^{t_n} \tr([\sigma_x^T \mathcal{H} \eta \sigma_x](\tilde{Y}^{\eps_n} (s))  ds\\
   & + s(\eps_n)\sqrt{\eps_n} \int_0^{t_n} \left( \sigma_x^T(\tilde{Y}^\eps_n (s)) \nabla \eta (\tilde{Y}^{\eps_n} (s)) \right) dB(s)
\end{aligned}
\end{equation}
Taking limit as $n \rightarrow \infty$, the right hand side converges to $0$ since $\eps_n, s(\eps_n) \rightarrow 0$ and $\eta \in \clc_b^2$. Also, the left hand side can be rewritten as 
\[ \int_0^{t_n} \nabla \eta (\tilde{Y}^{\eps_n} (s)) \mathcal{V}_x (\tilde{Y}^{\eps_n} (s)) ds = t_n \int_{\mathbb{R}^d} \nabla \eta (y) \mathcal{V}_x (y) \pi^n (dy).  \]
Since $t_n \in [\kappa,1]$ for all $n$, we have, from the weak convergence of $\pi^n$ to $\pi$, the square integrability estimate in \eqref{eq:1055}, and  linear growth of $\clv_x$
from Assumption \ref{assumptions-1}, that, a.s., \eqref{eq:552n} holds.

Now, from the global asymptotic stability of the unique fixed point $0$ of the ODE 
$\dot{y} = -\mathcal{V}_x(y)$ (Assumption \ref{assumptions-1}(3)), we  see that $\pi = \delta_0$.
Indeed, denoting the solution of the above ODE with $y(0)=z \in \RR^d$ as $y_z(t)$ and defining
$$T_t\eta(z) \doteq \eta(y_z(t)), \;\; \clg \eta(z) \doteq -\nabla \eta(z) \cdot \clv_x(z), \;\; z \in \RR^d,$$
we see that 
$$\int_{\RR^d} T_t\eta(z) \pi(dz) - \int_{\RR^d}\eta(z)\pi(dz)  = \int_0^t \int_{\RR^d}  (\clg T_s\eta)(z) \pi(dz) ds =   -\int_0^t  \int_{\mathbb{R}^d} \nabla (T_s\eta) (z) \cdot \mathcal{V}_x (z) \pi (dz)\, ds = 0$$
which, on sending $t\to \infty$ and using the global asymptotic stability property, says that
$\int_{\RR^d} \eta(z) \delta_{0}(dz) = \int_{\RR^d}\eta(z)\pi(dz)$ a.s. for all $\eta \in \clc_b^2$, proving the identity $\pi = \delta_0$. However, since $\pi_n \Rightarrow \pi$, this contradicts the inequality in \eqref{eq:145}, completing the proof of part (2) of the proposition.
\end{proof}

\subsection{LDP Lower Bound}
\label{subsec:single-lower}

In this section, we prove the lower bound for the LDP, namely the inequality in \eqref{ineq-single-lower}.  Fix $\delta > 0$ and let $\gamma^* \in \mathcal{P}(\mathbb{R}^d)$ be $\delta-$optimum for the infimum on the right side of \eqref{ineq-single-lower}, namely,
\begin{equation}
    F(\gamma^*) + \int_{\mathbb{R}^d} \|\sigma^T(y) \nabla \phi(y)\|^2 \gamma^*(dy) = F(\gamma^*) + I_1(\gamma^*) \leq \inf_{\mu \in \mathcal{P}(\mathbb{R}^d)} \bigg( F(\mu) + I_1(\mu)  \bigg) + \delta.
\label{single-near-optimal-measure-selection}
\end{equation}

Denote by $\mathcal{P}_{dis}$ the class of all probability distributions on $\mathbb{R}^d$ that are supported on a finite set. Then, we can find a $\tilde{\mu} \in \mathcal{P}_{dis}$ such that
\begin{equation}
\label{single-discrete-approximation}
    |F(\gamma^*) - F(\tilde{\mu})| \leq \delta, \quad  \int_{\mathbb{R}^d} \|\sigma^T(y) \nabla \phi(y)\|^2 \tilde{\mu}(dy) \leq  \int_{\mathbb{R}^d} \|\sigma^T(y) \nabla \phi(y)\|^2 \gamma^*(dy) + \delta.
\end{equation}

Indeed,  consider an iid sequence of $\mathbb{R}^d-$valued random variables $\xi_1,\xi_2,\dots$ distributed as $\gamma^*$ on some probability space $(\tilde{\Omega},\tilde{\mathcal{F}},\tilde{P})$, then from the Glivenko-Cantelli theorem and the Strong Law of Large Numbers
\begin{equation}
    \frac{1}{n} \sum_{i=1}^n \delta_{\xi_i} \rightarrow \gamma^*, \;\;  \frac{1}{n} \sum_{i=1}^n \|\sigma^T(\xi_i) \nabla \phi(\xi_i)\|^2 \rightarrow \int_{\mathbb{R}^d} \|\sigma^T(y) \nabla \phi(y)\|^2 \gamma^*(dy),  \;  \tilde{P} \text{ a.s. }
\end{equation}

Now fix a $\tilde{\omega}$ in the set of full measure on which the above two convergence results hold. It then follows with $\mu_n = \frac{1}{n} \sum_{i=1}^n \delta_{\xi_i(\tilde{\omega})}$ that

\begin{equation}
\nonumber
    |F(\gamma^*) - F(\mu_n)| \rightarrow 0, \quad  \int_{\mathbb{R}^d} \|\sigma^T(y) \nabla \phi(y)\|^2 \mu_n(dy) \rightarrow \int_{\mathbb{R}^d} \|\sigma^T(y) \nabla \phi(y)\|^2 \gamma^*(dy).
\end{equation}
This proves the statement in \eqref{single-discrete-approximation} and gives 
\begin{equation}
    F(\tilde\mu) + \int_{\mathbb{R}^d} \|\sigma^T(y) \nabla \phi(y)\|^2 \tilde\mu(dy)  \leq \inf_{\mu \in  \mathcal{P}(\mathbb{R}^d)} \bigg( F(\mu) + I_1(\mu)  \bigg) + 3\delta.
\label{snomes2}
\end{equation}
Now suppose that $\mbox{supp}( \tilde{\mu}) = \{ x_1,\dots, x_k \}$ with $\tilde{\mu}\{x_i\} = p_i$, namely $\tilde{\mu} = \sum_{i=1}^k p_i \delta_{x_i}$. In order to prove the lower bound in \eqref{ineq-single-lower} we will once more use the variational representation in \eqref{eq:238}. With this variational representation, the proof reduces to the construction of a suitable sequence of controls $v^\eps$ and controlled  empirical measures $\Bar{\mu}^{\eps}$ such that the associated costs
$$\E \left[ \frac{1}{2}\int_0^1 \|{v}^{\eps}(s)\|^2 ds + F(\bar\mu^\eps)  \right]$$
are asymptotitcally close as $\eps \rightarrow 0$ to 
$$F(\tilde{\mu}) +  \int_{\mathbb{R}^d} \|\sigma^T(y) \nabla \phi(y)\|^2 \tilde\mu(dy).$$
 In order for this asymptotic behavior, the controls should keep the empirical measure $\bar{\mu}^\eps = \int_0^1 \delta_{\bar{Y}^\eps(s)} ds$ asymptotically close to the discrete measure $\tilde{\mu}$ and keep the associated cost $\frac{1}{2} \int_0^1 \|{v}^\eps(s)\|^2 ds$ asymptotically close to $\frac{1}{2} \int_{\mathbb{R}^d} \|\sigma^T(y) \nabla \phi(y)\|^2 \tilde{\mu} (dy)$. 

Our strategy will be to construct controls that make the stochastic proces $\bar{Y}^\eps$ visit sequentially the $k$ points in the support of $\tilde{\mu}$ and spend approximately $p_i$ units of time in the vicinity of $x_i$ by incurring a control cost of approximately $p_i \|\sigma^T(x_i) \nabla \phi(x_i)\|^2$. More precisely, the control we will construct will have the following features:
\begin{itemize}
    \item In a short time and with negligible cost the process $\bar{Y}^\eps(t)$ travels to a neighbourhood of $x_1$.
    \item For approximately $p_1$ amounts of time the state is controlled to stay in the viccinity of $x_1$ while paying a total cost that is approximately $p_1 \|\sigma^T(x_1) \nabla \phi(x_1)\|^2$.
    \item The process is then moved to $x_2$ in a short time, while expending a negligible control cost. Then, it is controlled to stay in the vicinity of $x_2$ for $p_2$ units of time while paying a costrol cost of $p_2 \|\sigma^T(x_2) \nabla \phi(x_2)\|^2$.
    \item This is continued until we finish with all the $k$ positions. 
\end{itemize}
This construction is summarized in Table \ref{table:process-travel}. For $x,y \in \RR^d$, define
$\yy(t) \equiv \yy(x,y; t)\doteq x+ t(y-x)$, $0\le t \le 1$ and define
the function $\uu(t) \equiv \uu(x,y;t) \doteq \sigma^T(\yy(t)) a^{-1}(\yy(t))(y-x)$, $0\le t \le 1$.
Note that $\uu$ and $\yy$ satisfy
$$\yy(t) = x + \int_0^t \sigma(\yy(s)) \uu(s) ds, \; 0 \le t \le 1$$
and for some $c \in (0,\infty)$
\begin{equation}\label{eq:defnuu}
	\int_0^1 \|\uu(t)\|^2 \le c (1+ \|x\|^2 + \|y\|^2) \mbox{ for all } x,y \in \RR^d.\end{equation}
Define
$\uu^{\eps}_{x,y}(t) \doteq \uu(x,y; t/\eps)$ for $0 \le t \le \eps$. Also for $i=1, \ldots, k$, let $P_i \doteq \sum_{j=1}^{i} p_j$ and set $P_0=0$.
Also, for $x,y \in \RR^d$, let 
\[\Lambda(y,x) \doteq a(y)^{-1} a(x).
\]
\begin{table}[h!]
\centering
 \begin{tabular}{|c | c| c |} 
 \hline
State of $\Bar{Y}^\eps$ & Time & Control Process $v^{\eps}(s)$  \\ [0.5ex] 
 \hline\hline
 $y_0 \rightarrow x_1$ & $ \eps$ & $\bigg(\sigma^T(\Bar{Y}^\eps(s)) \nabla \phi (\Bar{Y}^\eps(s)) + \uu^{\eps}_{y_0,x_1}(s)\bigg) 1_{(0,\eps]}(s)$ \\ 
 \hline
 $x_1$ & $p_1 - \eps$ & $\sigma^T(\Bar{Y}^\eps(s)) \Lambda(\bar{Y}^\eps(s),x_1) \nabla \phi (x_1)1_{(\eps,P_1]}(s)$  \\
 \hline
 $\bar Y^{\eps}(P_1) \rightarrow x_2$ & $\eps$ & $\bigg(\sigma^T(\Bar{Y}^\eps(s)) \nabla \phi (\Bar{Y}^\eps(s)) + \uu^{\eps}_{\Bar{Y}^\eps(P_1),x_2}(s-P_1)\bigg) 1_{(P_1,P_1+\eps]}(s)$  \\
 \hline
 $x_2$ & $p_2 - \eps$ & $\sigma^T(\Bar{Y}^\eps(s)) \Lambda(\bar{Y}^\eps(s),x_2) \nabla \phi (x_2)1_{(P_1 + \eps,P_2]}(s)$  \\
 \hline
 \dots & \dots & \dots \\ [1ex] 
 \hline
 $\bar Y^{\eps}(P_{k-1}) \rightarrow x_k$ & $\eps$ & $\bigg(\sigma^T(\Bar{Y}^\eps(s)) \nabla \phi (\Bar{Y}^\eps(s)) + \uu^{\eps}_{\Bar{Y}^\eps(P_{k-1}),x_k}(s-
 P_{K-1})\bigg) 1_{(P_{k-1},P_{k-1} + \eps]}(s)$ \\ [1ex] 
 \hline
 $x_k$ & $p_k - \eps$ & $\sigma^T(\Bar{Y}^\eps(s)) \Lambda(\bar{Y}^\eps(s),x_k) \nabla \phi (x_k) 1_{(P_{k-1}+\eps,1]}(s)$ \\ [1ex] 
\hline
\end{tabular}
\caption{Construction of the Control Process. First column gives the approximate states (or transitions) for $\Bar{Y}^\eps$ under the selection of the controls, $v^\eps$.}
\label{table:process-travel}
\end{table}

Then the state controlled process is given by the equation \eqref{eq:Y-bar} with the control process $v^\eps$ defined in state feedback form as: For $i=0, 1, \ldots , k-1$
\begin{equation}
 v^\eps(t) \doteq 
 \begin{cases}
 	\sigma^T(\Bar{Y}^\eps(t)) \nabla \phi (\Bar{Y}^\eps(t)) + \uu^{\eps}_{\Bar{Y}^\eps(P_{i}),x_{i+1}}(t-P_{i}), \;\; \mbox{ if } t \in (P_i, P_i+\eps]\\
\sigma^T(\Bar{Y}^\eps(t)) \Lambda(\bar{Y}^\eps(t),x_{i+1}) \nabla \phi (x_{i+1}), \;\; \mbox{ if } t \in (P_i+\eps, P_{i+1}].	
 \end{cases}
   \label{eq:process}
 \end{equation}
%
%
From the Lipschitz property of $\nabla \phi$, and $\Lambda$, and the boundedness and Lipschitz property of $\sigma^T$ we see that the above state feedback control is well defined and the corresponding SDE in \eqref{eq:Y-bar} has a unique solution. For the rest of the section, $\bar Y^{\eps}$ will denote the solution of this SDE  with the above feedback control.

The following lemma gives a key moment bound.
\begin{lemma}
\begin{equation}
    \sup_{0 \leq t \leq 1}   \sup_{\eps \in (0,1)} \E \|\Bar{Y}^\eps (t)\|^2 < \infty.  
    \label{bound-Y-bar}
\end{equation}
\label{lemma:bound-Y-bar}
\end{lemma}
\begin{proof}
We will show via induction that for $i=0, 1, \ldots , k$,
\begin{equation}
	 \sup_{0 \leq t \leq P_i}   \sup_{\eps \in (0,1)} \E \|\Bar{Y}^\eps (t)\|^2 < \infty.
\end{equation}
Clearly the result is true for $i=0$. Suppose now that the result is true for some $i \in \{0,1,\dots,k-1\}$ and consider $i+1$. Then, for $t \in [P_i,P_i+\eps]$
\begin{equation} 
\nonumber
\Bar{Y}^\eps (t)  = \Bar{Y}^\eps(P_i) + \frac{1}{\eps}\int_{P_i}^t \sigma(\Bar{Y}^\eps (s)) \cdot \uu\left(\Bar{Y}^\eps(P_i),x_{i+1}; \frac{s-P_i}{\eps}\right) ds  + \frac{s(\eps)}{\sqrt{\eps}} \int_{P_i}^{t} \sigma(\Bar{Y}^\eps (s)) dB(s). 
\end{equation}
Thus with $c$ as in \eqref{eq:defnuu}
\begin{equation}
\begin{aligned}
    \E ( \sup_{P_i \leq t \leq P_i + \eps} \|\Bar{Y}^\eps (t)\|^2) &
    \leq  3 \E \|\Bar{Y}^\eps(P_i)\|^2 + 3\|\sigma\|^2_\infty \E \int_0^1 \|\uu(\Bar{Y}^\eps(P_i),x_{i+1}; \cdot)\|^2 ds + 12 \frac{s^2(\eps)}{\eps} \|\sigma\|^2_\infty \eps \\
    &\leq 3(1 + c \|\sigma\|_\infty^2) \E \|\Bar{Y}^\eps(P_i)\|^2 + 3c \|\sigma\|_\infty^2 \|x_{i+1}\|^2 + 12 \|\sigma\|_\infty^2 s^2(\eps) + 3 c \|\sigma\|_\infty^2.
\end{aligned}
\end{equation} 
Using the induction hypothesis we now have 

\begin{equation}
    \label{lemma-single-induction-first}
    \sup_{0 \leq t \leq P_i + \eps} \sup_{\eps \in (0,1)} \E \|\bar{Y}^\eps(t)\|^2 < \infty.
\end{equation}
Consider now $ P_i + \eps \leq t \leq P_{i+1}$ and define 
\begin{equation}
    \tilde{Y}^\eps_{i+1}(t) = \Bar{Y}^\eps (P_i + \eps +t) - x_{i+1}, \quad 0 \leq t \leq p_{i+1} - \eps.
    \label{single-Y-tilde-def}
\end{equation}
Then, for $t \in [0,p_{i+1}-\eps]$ 

\begin{equation}
    \begin{aligned}
        \tilde{Y}_{i+1}^\eps (t) &= \Bar{Y}^\eps (P_i + \eps) - x_{i+1}  
        - \frac{1}{\eps} \int_{P_i + \eps}^{P_i + \eps + t} \sigma(\Bar{Y}^\eps (s)) \sigma^T (\Bar{Y}^\eps (s))(\nabla \phi (\Bar{Y}^\eps (s)) - \Lambda(\bar{Y}^\eps(s),x_{i+1}) \nabla \phi (x_{i+1}) ) ds\\
		 &+ \frac{s(\eps)}{\sqrt{\eps}} \int_{P_i + \eps}^{P_i + \eps + t} \sigma(\Bar{Y}^\eps (s)) dB(s)   \\ 
        &= \Bar{Y}^\eps (P_i + \eps) - x_{i+1} - \frac{1}{\eps} \int_0^{t} \mathcal{V}_{x_{i+1}} (\tilde{Y}_{i+1}^\eps(s)) ds + \frac{s(\eps)}{\sqrt{\eps}} \int_0^{t} \sigma_{x_{i+1}}(\tilde{Y}_{i+1}^\eps (s)) d\tilde B_i(s),
    \end{aligned}
\end{equation} 
where $\sigma_x(y)$ is as introduced in the proof of Lemma \ref{theorem:auxillary} and
$\tilde B_i(s) \doteq B(t+P_i+\eps) - B(P_i+\eps)$ is a Brownian motion. Thus $\tilde{Y}^\eps_{i+1}$ satisfies the SDE in Lemma \ref{theorem:auxillary} with $x= x_{i+1}$ and
$x^{\eps} = \Bar{Y}^\eps (P_i + \eps) - x_{i+1}$.
  Consequently, from Lemma \ref{theorem:auxillary} 
\begin{equation}
\begin{aligned}
     \sup_{P_i + \eps \leq t \leq P_{i+1}}   \sup_{\eps } \E \|\Bar{Y}^\eps (t) \|^2 
     &\leq 2 \|x_{i+1}\|^2 + 2 \sup_{P_i + \eps \leq t \leq P_{i+1}} \sup_{\eps } \E \|\Bar{Y}^\eps (t) - x_{i+1}\|^2  \\
    & \leq 2 \|x_{i+1}\|^2 + 2 \sup_{0 \leq t
     \leq p_{i+1}-\eps} \sup_{\eps } \E \| \tilde{Y}^\eps_{i+1}(t)\|^2 \\
     &\leq 2 \|x_{i+1}\|^2 + 2 \kappa_1(x_{i+1}) (1 + \E \|\Bar{Y}^\eps(P_i+\eps)-x_{i+1}\|^2).
\end{aligned}
\end{equation}
Combining the above with \eqref{lemma-single-induction-first} we now have that
 
\begin{equation}
        \sup_{0 \leq s \leq P_{i+1}} \sup_{\eps \in (0,1)} \E \|\Bar{Y}^\eps (s)\|^2 < \infty.
    \label{bound:Y-bar}
\end{equation}
The proof is complete by induction.
\end{proof}

The next lemma shows that the control process moves the state process to the vicinity of $x_{j+1}$ in the time interval from $P_j$ to $P_j + \eps$.

\begin{lemma}
\label{single-lemma-process-movement}
For $i = 0, 1, \dots, k-1$
\[ \Bar{Y}^\eps(P_{i} + \eps) \xrightarrow{P} x_{i+1},  \quad \text{as} \quad \eps \rightarrow 0   \]
\end{lemma}
\begin{proof}
For $i = 0,1,\dots, k-1$ and $t \in [P_i, P_i+\eps]$,
\begin{equation}
        \Bar{Y}^\eps (t) = \Bar{Y}^\eps (P_i ) + \frac{1}{\eps} \int_{P_i}^t \sigma(\Bar{Y}^\eps (s)) \uu^{\eps}_{\Bar{Y}^\eps (P_i), x_{i+1}} (s-P_i) ds 
        + \frac{s(\eps)}{\sqrt{\eps}} \int_{P_i}^t \sigma(\Bar{Y}^\eps (s)) dB(s).
    \label{eq:Y-bar-control}
\end{equation}
Define $\Bar{Z}^\eps(t) = \Bar{Y}^\eps(P_i + \eps t)$ for $0 \leq t \leq 1$. Then \eqref{eq:Y-bar-control} can be written as 
\begin{equation}
    \begin{gathered}
        \Bar{Z}^\eps(t) = \Bar{Z}^\eps(0) + \int_0^t \sigma(\Bar{Z}^\eps(s)) \uu(\Bar{Y}^\eps (P_i), x_{i+1}; s) ds + s(\eps) \int_0^t \sigma(\Bar{Z}^\eps(s)) d\tilde B^{\eps}(s),
    \end{gathered}
    \label{eq:Z-bar-control}
\end{equation}
where $\tilde B^{\eps}(s) \doteq \eps^{-1/2} (B(s\eps+P_i)- B(P_i))$, $0\le s \le 1$, is a BM starting from the origin.
Also consider the ODE given by the noiseless version of the above equation, namely,
\begin{equation}
    \tilde{Z}^{\eps}(t) = \Bar{Y}^\eps(P_i) + \int_0^t\sigma(\tilde{Z}^{\eps}(s)) \uu(\Bar{Y}^\eps (P_i), x_{i+1}; s) ds.
\end{equation}
Note that, by the definition of $\uu^\veps(\cdot)$, $ \tilde{Z}^{\eps}(1) = x_{i+1}$. Also, using Grönwalll's lemma and Lipschitz property of $\sigma$,
\begin{equation}\label{eq:425}
\|\Bar{Y}^\eps(P_i+\eps) - x_{i+1}\| =	\|\Bar{Z}^\eps(1) - \tilde{Z}^{\eps}(1)\|
\le s(\eps)\sup_{0\le t \le 1} \|M^{\eps}(t)\| \exp\{\|\sigma\|_{Lip} \int_0^1 \|\uu(\Bar{Y}^\eps (P_i), x_{i+1}; s)\|ds \},
\end{equation}
where $\|\sigma\|_{Lip} $ is the Lipschitz constant of $\sigma$ and $M^\eps(t) \doteq \int_0^t \sigma(\Bar{Z}^\eps(s)) d\tilde B^{\eps}(s)$. 
Next note that, from the properties of $\uu$,
$$\int_0^1 \|\uu(\Bar{Y}^\eps (P_i), x_{i+1}; s)\|ds \leq \left( c(\|\bar{Y}^\eps(P_i)\|^2 + \|x_{i+1}\|^2 + 1) \right)^{1/2} $$
 and that from Lemma \ref{lemma:bound-Y-bar}
 $\{\bar{Y}^\eps(P_i)\}_{\eps>0}$ is tight.
The result now follows on  sending $\eps \to 0$ in \eqref{eq:425}.
 \end{proof}

The following lemma shows that the empirical measure of $\bar{Y}^\eps(t)$ over the interval $[P_i+\eps,P_{i+1}]$ is concentrated near $x_{i+1}$.

\begin{lemma} 
As $\eps \rightarrow 0$, for each $i =0, 1, \dots, k-1$
 \[ \E \dbl \left( \frac{1}{p_{i+1}-\eps} \int_{P_{i} + \eps}^{P_{i+1}} \delta_{\Bar{Y}^\eps (s)} ds, \delta_{x_{i+1}}    \right) \rightarrow 0 . \]
 \label{lemma:occupational-conv}
\end{lemma}

\begin{proof}
Fix $i = 0, \dots, k-1$. Since for a $z \in \mathbb{R}^d$ and $f \in BL_1(\RR^d)$, the map $x \mapsto f(x-z)$ is also in $BL_1(\RR^d)$, we see that
\begin{equation}
\label{eq:expectation-bounded Lipschitz}
\begin{split}
\E \dbl \left( \frac{1}{p_{i+1}-\eps} \int_{P_{i} + \eps}^{P_{i+1}} \delta_{\Bar{Y}^\eps (s)} ds, \delta_{x_{i+1}}    \right) &= \E \dbl \left( \frac{1}{p_{i+1}-\eps} \int_{P_{i} + \eps}^{P_{i+1}} \delta_{\Bar{Y}^\eps (s) - x_{i+1}} ds, \delta_{0}    \right) \\
 &= \E \dbl \left( \frac{1}{p_{i+1}-\eps} \int_{0}^{p_{i+1}-\eps} \delta_{\tilde{Y}^\eps_{i+1} (s)} ds, \delta_{0}    \right),
\end{split}
\end{equation}
 where the last line follows on recalling the definition of the process $\tilde{Y}^\eps_{i+1}$ from \eqref{single-Y-tilde-def}. 
 The expectation on the last line can be written as 
  \begin{equation}
     \nonumber
     ED_{i+1}^\eps (\bar{Y}^\eps(P_{i}+\eps) - x_{i+1}) = ED_{i+1}^\eps (\tilde{Y}_{i+1}^\eps(0))
 \end{equation}
where for $x \in \mathbb{R}^d$
\begin{equation}
    \nonumber
    D^{\eps}_{i+1}(x) \doteq  \dbl \left( \frac{1}{p_{i+1}-\eps} \int_{0}^{p_{i+1} - \eps} \delta_{Y^\eps(s,x_{i+1},x)} ds, \delta_{0} \right)
\end{equation}
and $Y^\eps(s,x_{i+1},x)$ is the process introduced in Lemma \ref{theorem:auxillary}.

From Lemma \ref{theorem:auxillary}, for any compact set $K \subset \mathbb{R}^d$, 
\begin{equation}
    \nonumber
    \sup_{x \in K} E D_{i+1}^\eps (x) \rightarrow 0, \quad \text{as} \quad \eps \rightarrow 0.
\end{equation}
From the tightness of $\{\tilde{Y}^\eps_{i+1}(0) \}$, which follows from Lemma \ref{single-lemma-process-movement}, we now see that 
\begin{equation}
    \nonumber
    ED_{i+1}^\eps (\tilde{Y}^\eps_{i+1}(0)) \rightarrow 0, \quad \text{as} \quad \eps \rightarrow 0.
\end{equation}
The result follows.
\end{proof}
The following lemma shows that the asymptotic costs under the control $v^\eps$ defined in \eqref{eq:process} is as desired.
\begin{proposition}
We have
\begin{equation}
    \nonumber 
    \limsup_{\eps \rightarrow 0}  \frac{1}{2}  \E \int_0^1 \|v^\eps (t)\|^2 dt = \frac{1}{2} \int_{\RR^d} \|\sigma^T (y) \nabla \phi (y) \|^2 \tilde \mu(dy).
\end{equation}
\label{prop:convergence-costs}
\end{proposition}

\begin{proof}
Note that
\begin{equation}
\frac{1}{2} \E \int_0^1 \|v^\eps(s)\|^2 ds =  \frac{1}{2} \E \sum_{i=0}^{k-1} \int_{P_i}^{P_i +\eps} \|v^\eps(s)\|^2 ds   +  \frac{1}{2} \E \sum_{i=0}^{k-1} \int_{{P_i+\eps}}^{P_{i+1}}\|v^\eps(s)\|^2 ds.  
\label{single-costs-decomp}
\end{equation}
We first argue that the cost associated with travel from $x_i$ to $x_{i+1}$ is negligible, i.e.
\begin{equation}
   \limsup_{\eps \to 0} \frac{1}{2} \E \int_{P_i}^{P_i+ \eps} \|v^\eps(s)\|^2 ds = 0, \quad \text{for all} \quad i = 0, \dots, k-1.  
    \label{single-negligible-costs-travel}
\end{equation} 
Indeed, using the linear growth of $\nabla \phi$ (which follows from Assumptions 2.1(2b)), the boundedness of $\sigma$, and the estimate for $\uu$ in \eqref{eq:defnuu}, for some $\kappa_1 \in (0,\infty)$
\begin{equation}
\nonumber
\begin{gathered}
    \E \int_{P_i}^{P_i +\eps} \|v^\eps(s)\|^2 ds \leq
    2 \E \int_{P_i}^{P_i+\eps} \|\sigma^T(\bar{Y}^\eps(s)) \nabla \phi(\bar{Y}^\eps(s))\|^2 ds + 2 \E \int_{P_i}^{P_i+\eps} \|\uu^{\eps}_{\Bar{Y}^\eps(P_i),x_{i+1}}(s)\|^2 ds \\
    \leq \kappa_1 \int_{P_i}^{P_i +\eps} E(1+ \|\Bar{Y}^\eps(s)\|^2) ds + 2 \eps  E\int_{0}^{1} \|\uu(\Bar{Y}^\eps(P_i),x_{i+1}, s)\|^2 ds\\
	\leq \kappa_1 \int_{P_i}^{P_i +\eps} E(1+ \|\Bar{Y}^\eps(s)\|^2) ds + 2 c \eps E (1+ \|x_{i+1}\|^2 + \|\Bar{Y}^\eps(P_i)\|^2).
\end{gathered}
\end{equation} 
The statement in \eqref{single-negligible-costs-travel} is now immediate from the moment bound in Lemma \ref{lemma:bound-Y-bar}.

Now we consider costs over the intervals $[P_i + \eps,P_{i+1}]$. For $i = 0,1,\dots,k-1$ consider the random probability measure

\begin{equation}
    \theta_i^\eps \doteq \frac{1}{p_{i+1} - \eps} \int_{P_i + \eps}^{P_{i+1}} \delta_{\bar{Y}^\eps(s)} ds.
\end{equation}
Then 

\begin{equation}
    \frac{1}{2} \E \int_{P_i + \eps}^{P_{i+1}} \|v^\eps(s)\|^2 ds = \frac{1}{2} (p_{i+1}-\eps)\E \int_{\mathbb{R}^d} \|\sigma^T (z) \Lambda(z,x_{i+1}) \nabla \phi(x_{i+1})\|^2 \theta^\eps_i (dz).
\end{equation}

From Lemma \ref{lemma:occupational-conv}, $\theta_i^\eps \rightarrow \delta_{x_{i+1}}$ in probability as $\eps \rightarrow 0$. Since $z \mapsto \Lambda(z,x_{i+1})$ and $\sigma$ are bounded and continuous functions and $\Lambda(x_{i+1},x_{i+1}) = \Id$, we now see that the right side converges, as $\eps \rightarrow 0$, to $\frac{1}{2} p_{i+1}\|\sigma^T(x_{i+1}) \nabla \phi(x_{i+1}) \|^2$. Combining the above with \eqref{single-negligible-costs-travel}, we now have the desired result from \eqref{single-costs-decomp}.
\end{proof}
We now show the convergence of the empirical measure associated with $\bar{Y}^\eps$. Recall $\tilde \mu$ introduced above \eqref{single-discrete-approximation}.
\begin{proposition}
 Let $\bar{\mu}^\eps = \int_0^1 \delta_{\bar{Y}^{\eps}(s)} ds$. Then, $\bar{\mu}^\eps$ converges to $\tilde \mu$ in probability as $\eps \rightarrow 0.$
 \label{prop:measure-conv}
\end{proposition}

\begin{proof}
Note that one can write $\bar{\mu}^\eps$ as 

\begin{equation}
\nonumber
    \bar{\mu}^\eps =
    \sum_{j=0}^{k-1} (p_{j+1}-\eps) \frac{1}{p_{j+1}-\eps} \int_{P_j + \eps}^{P_{j+1}} \delta_{\bar{Y}^\eps(s)} ds + \eps \sum_{j =0}^{k-1} \frac{1}{\eps} \int_{P_j}^{P_j + \eps} \delta_{\bar{Y}^\eps(s)} ds.
\end{equation}
Also recall that 
\begin{equation}
   \tilde \mu = \sum_{j=0}^{k-1} p_{j+1} \delta_{x_{j+1}} = \sum_{j=0}^{k-1} (p_{j+1} - \eps) \delta_{x_{j+1}} + \sum_{j=0}^{k-1} \eps \delta_{x_{j+1}}.
\end{equation}
It then follows that 
\begin{equation}
    \nonumber
    \begin{gathered}
                \E \dbl (\bar{\mu}^\eps, \tilde \mu) \leq \sum_{j=0}^{k-1} (p_{j+1} - \eps) \E \dbl \left( \frac{1}{p_{j+1}-\eps} \int_{P_j + \eps}^{P_{j+1}} \delta_{\bar{Y}^\eps(s)} ds, \delta_{x_{j+1}}    \right) \\
                + \eps \sum_{j=0}^{k-1} \E \dbl \left( \frac{1}{\eps} \int_{P_j}^{P_j + \eps} \delta_{\bar{Y}^\eps(s)} ds, \delta_{x_{j+1}}   \right).
    \end{gathered}
\end{equation}
The result now follows  on sending $\eps \rightarrow 0$ and using Lemma \ref{lemma:occupational-conv}.
\end{proof}
\subsubsection{Proof of the LDP lower bound.}
We now complete the proof of \eqref{ineq-single-lower}.
We will apply Theorem \ref{Variational-1} with $\mathcal{R} = \mathcal{A}$. By a similar argument as in Section \ref{subsec:single-upper},

\begin{equation}
\begin{gathered} \nonumber
       -\eps s^2(\eps)  \log \E \exp \left(- \frac{F(\mu^\eps)}{\eps s^2(\eps)} \right) = \inf_{v \in \mathcal{A}} \E \left(  F \circ \mathcal{G}^\eps (B)  + \frac{1}{2} \int_0^1 \|v(s)\|^2 ds \right) \\ 
        \leq \E \left(  F(\Bar{\mu}^\eps) + \frac{1}{2} \int_0^1 \|v^\eps(s)\|^2 ds \right),
\end{gathered}
\end{equation}
where $\bar{\mu}^\eps$ is as introduced in Proposition \ref{prop:measure-conv} and $v^\eps$ is as constructed in \eqref{eq:process}.
Thus,
\begin{equation*}
\begin{split}
\limsup_{\eps \rightarrow 0} -\eps s^2(\eps)  \log \E \exp \left(- \frac{F(\mu^\eps)}{\eps s^2(\eps)} \right) &\leq
        \limsup_{\eps \rightarrow 0}
        \E \left(  F(\Bar{\mu}^\eps) + \frac{1}{2} \int_0^1 \|v^\eps(s)\|^2 ds \right)\\
       &= F(\tilde \mu) + \frac{1}{2} \int \|\sigma^T (y) \nabla \phi (y) \|^2 \tilde \mu(dy)\\
		&\leq \inf_{\gamma \in \mathcal{P}(\mathbb{R}^d)} \left( F(\gamma) + I_1(\gamma)    \right)  + 3\delta,
\end{split}
\end{equation*}
where the first equality is from Propositions \ref{prop:convergence-costs} and \ref{prop:measure-conv}, and the last inequality is from \eqref{snomes2}. Since $\delta>0$ is arbitrary, the inequality in \eqref{single-near-optimal-measure-selection} follows. \hfill \qed

\subsection{Compactness of Level Sets of $I_1$}
 \label{sec:cptlvthm1} 
In this section we show that 
the function $I_1$ as defined in \eqref{eq:rate-function} is a rate function.

Fix $M>0$. It suffices to show that the set $\{ \gamma \in \mathcal{P}(\mathbb{R}^d): I_1(\gamma) \leq M \}$ is compact. Consider a sequence $\gamma_n \in \{ \gamma \in \mathcal{P}(\mathbb{R}^d): I_1(\gamma) \leq M \}$. Then
\[ \int_{\RR^d} \|\sigma^T(y)\nabla \phi (y) \|^2  \gamma_n (dy) \leq 2M, \quad \text{for all}  \quad n \in \mathbb{N}. \]
We now argue that $\{\gamma_n\}$ is a tight sequence of probability measures on $\mathbb{R}^d$. Fix $\kappa>0$. From Assumption \ref{assumptions-1} (parts 1 and 2(d)) there exists $M_1 \in (0,\infty)$ such that for $\|y\| > M_1$, $\|\sigma^T(y) \nabla \phi(y)\|^2 \geq \frac{2M}{\kappa}.$ Then for all $n \in \mathbb{N}$
\begin{equation}
 \int_{\|y\| \geq M_1 } \gamma_n (dy) = \frac{\kappa}{2M} \int_{\|y\| \geq M_1 } \frac{2M}{\kappa} \gamma_n (dy) \\ \leq \frac{\kappa}{2M} \int_{\|y\| \geq M_1 } \|\sigma^T(y) \nabla \phi(y)\|^2 \gamma_n (dy) \leq \frac{\kappa}{2M} 2M = \kappa.
\end{equation}  
Since $\kappa>0$ is arbitrary, we have that $\{\gamma_n\}$ is tight. Thus, we can find some subsequence (labeled again as $n$) along which $\gamma_n \to \bar\gamma$ for some $\bar\gamma \in \clp(\RR^d)$.
 From Fatou's lemma 
\[  \int_{\RR^d} \|\sigma^T(y) \nabla \phi(y)\|^2 \bar\gamma(dy) \leq \liminf_n \int_{\RR^d} \|\sigma^T(y) \nabla \phi(y)\|^2 \gamma_{n} (dy) \leq 2M,    \]
which shows that $\bar\gamma \in \{ \gamma \in \mathcal{P}(\mathbb{R}^d): I(\gamma) \leq M \}$. Thus $\{ \gamma \in \mathcal{P}(\mathbb{R}^d): I(\gamma) \leq M \}$ is compact. \hfill \qed

\section{Proof of Theorem \ref{theorem:ldp-2}.}
\label{sec:pfthmtwo}
We now turn to the multiscale system introduced in Section \ref{sec:m2} and prove our second main result, namely Theorem \ref{theorem:ldp-2}.
As before, the proof proceeds by establishing the associated Laplace asymptotics. Recall the definitions of the spaces $\clx$ and $\clm_1$ from Section \ref{sec:m2}.
Section  \ref{subsec:two-scale-upper} shows that for every $F \in \clc_b(\clx \times \clm_1)$
\begin{equation}
\label{ineq-two-upper}
    \liminf_{\eps \rightarrow 0}  \left[- \eps s^2(\eps) \log \E e^{-\frac{F(X^\eps, \Lambda^{\eps})}{\eps s^2(\eps)}} \right] \geq \inf_{(\xi, \nu) \in \clx \times \clm_1} \bigg( F(\xi, \nu) + I_2(\xi, \nu)  \bigg)
\end{equation}
which gives the LDP upper bound.
Then, in Section \ref{subsec:two-scale-lowern} we  prove the complementary lower bound: for every $F \in \clc_b(\clx \times \clm_1)$
\begin{equation}
\label{ineq-two-lower}
    \limsup_{\eps \rightarrow 0}  \left[- \eps s^2(\eps) \log \E e^{-\frac{F(X^\eps, \Lambda^{\eps} )}{\eps s^2(\eps)}} \right] \leq \inf_{(\xi, \nu) \in \clx \times \clm_1} \bigg( F(\xi, \nu) + I_2(\xi, \nu)   \bigg).
\end{equation}
Finally Section \ref{sec:cptlevi2} proves that $I_2$ is a rate function. Theorem \ref{theorem:ldp-2} is an immediate consequence of these three results.
Throughout this section, Assumption \ref{Assumptions:system} will be taken to hold.

\subsection{LDP Upper Bound}
\label{subsec:two-scale-upper}

In this Section, we prove the inequality in \eqref{ineq-two-upper}. Under Assumption \ref{Assumptions:system}, the SDE system in \eqref{eq:systemor}
has a unique pathwise solution $(X^{\eps}, Y^{\eps})$ given on a filtered probability space $(\Omega, \mathcal{F}, \{\clf_t\}, P)$ satisfying the usual conditions and equipped with mutually independent $k$ and $d$ dimensional $\{\clf_t\}$-Brownian motions $\{W(t)\}$ and $\{B(t)\}$.
This says that, there exists a measurable map $\tilde{\mathcal{G}}^\eps: \mathcal{C}([0,T]:\mathbb{R}^{d+k}) \rightarrow \clx \times \clm_1$ such that $(X^\eps, \Lambda^{\eps})= \tilde{\mathcal{G}}^\eps(B,W)$ where $\Lambda^{\eps}$ is defined as in \eqref{eq:557}.

Fix $\eps > 0$ and $F \in \clc_b(\clx \times \clm_1)$  and apply Theorem \ref{Variational-1} with $p = k+d$, $\beta = (B,W)$, and $G$ replaced by $G^\eps = F \circ \mathcal{\tilde{G}}^\eps$. Then, as in Section \ref{subsec:single-upper} we have that
\begin{equation}
\begin{aligned}
            &- \eps s^2(\eps) \log \E e^{-\frac{F(X^\eps, \La^{\eps})}{\eps s^2(\eps)}} \\
            &= \inf_{(v_1,v_2) \in \mathcal{A}_{b}} \E \left[ \frac{1}{2}\eps s^2 (\eps) \int_0^T (\|v_1(s)\|^2  + \|v_2(s)\|^2) ds + G^\eps \left( B + \int_0^\cdot  {v}_1(s)ds, W + \int_0^\cdot v_2(s)ds \right)   \right] \\
            &= \inf_{(v_1,v_2) \in \mathcal{A}_{b}} \E \left[ \frac{1}{2} \int_0^T (\|v_1(s)\|^2  + \|v_2(s)\|^2) ds  + G^\eps \left( B + \frac{1}{\sqrt{\eps} s(\eps)}\int_0^\cdot v_1(s)ds, W + \frac{1}{\sqrt{\eps} s(\eps)}\int_0^\cdot v_2(s) ds \right)   \right].
\end{aligned}
\end{equation}
The classes $\cla_b= \cla_b^p$ and $\cla=\cla^p$ are as in Section \ref{sec:varformula}. Note that any $v \in \cla^{d+k}$ can be written as $v=(v_1, v_2)$ where $v_1 \in \cla^d$ and $v_2 \in \cla^k$ and if $v \in \cla^{d+k}_{b,M}$ then $v_1 \in \cla^{d}_{b,M}$ and $v_2 \in \cla^{k}_{b,M}$.

Fix $\delta > 0$ and choose for each $\eps > 0$ a $(\bar{v}_1^\eps, \bar{v}_2^\eps) \in \mathcal{A}_b$ that is $\delta-$optimal for the right side, namely,
\begin{equation}
    \begin{aligned}
    - \eps s^2(\eps) \log \E e^{-\frac{F(X^\eps, \La^{\eps})}{\eps s^2(\eps)}} 
    &\geq \E \left[ \frac{1}{2}\int_0^T \|\bar{v}^{\eps}_1(s)\|^2 ds + \frac{1}{2}\int_0^T \|\bar{v}^{\eps}_2(s)\|^2 ds \right. \\ 
	&\quad + \left. G^\eps \left( B + \frac{1}{\sqrt{\eps} s(\eps)}\int_0^\cdot \bar{v}^{\eps}_1(s)ds,W + \frac{1}{\sqrt{\eps} s(\eps)}\int_0^\cdot \bar{v}^{\eps}_2(s)ds \right)   \right] - \delta.
    \end{aligned}
\end{equation}
A similar localization argument as invoked in  Section \ref{subsec:single-upper} shows that, there is a $M \in (0,\infty)$, and for each $\eps >0, (v_1^\eps,v_2^\eps) \in \mathcal{A}_{b,M}$ such that 
\begin{equation}
    \label{system-upper-variational}
    \begin{aligned}
        - \eps s^2(\eps) \log \E e^{-\frac{F(X^\eps, \La^{\eps})}{\eps s^2(\eps)}} &
        \geq \E \left[ \frac{1}{2}\int_0^T \|v_1^{\eps}(s)\|^2 ds + \frac{1}{2}\int_0^T \|v_2^{\eps}(s)\|^2 ds \right. \\
		& \quad  \left. + G^\eps \left( B + \frac{1}{\sqrt{\eps} s(\eps)}\int_0^\cdot v_1^{\eps}(s)ds, W + \frac{1}{\sqrt{\eps} s(\eps)}\int _0^\cdot v_2^{\eps}(s)ds \right)   \right] - 2 \delta.  
    \end{aligned}
\end{equation}
Also by application of Girsanov's Theorem, it follows that 
$$G^\eps \left( B + \frac{1}{\sqrt{\eps} s(\eps)}\int_0^\cdot {v}^{\eps}_1(s)ds, W + \int _0^\cdot {v}^{\eps}_2(s)ds \right) = (\Bar{X}^\eps, \Bar{\La}^{\eps}),$$
 where $(\Bar{X}^\eps, \Bar{Y}^\eps)$ is the solution to
\begin{equation}
	\begin{aligned}
d\bar{X}^\eps (t) &= \left(b(\bar{X}^\eps(t), \bar{Y}^\eps(t))+\alpha(\bar{X}^\eps(t))v_2^\eps(t)\right)dt + s(\eps) \sqrt{\eps} \alpha(\bar{X}^\eps(t))  dW(t), \, 
 \bar{X}^\eps(0) = x_0,\\
   d\bar{Y}^\eps (t) &= \left(-\frac{1}{\eps} \nabla_y U(\bar{X}^\eps (t) \textit{,} \bar{Y}^\eps (t))+ \frac{1}{\eps} v_1^\eps(t)\right)dt + \frac{s(\eps)}{\sqrt{\eps}} dB(t), \,     \bar{Y}^\eps(0) = y_0
 \end{aligned}
\label{eq:system-bar-def}
\end{equation}
and 
\begin{equation}\label{eq:laeps}
	\bar{\La}^\eps(A \times [0,t]) \doteq \int_{[0,t]} 1_{A}(\Bar{Y}^{\eps}(s)) ds, \; t \in [0,T], \; A \in \clb(\RR^d).\end{equation}
Denote for $t \in [0,T]$, $\mathbb{M}_t = [0,t] \times \mathbb{R}^d \times \mathbb{R}^d$ and define $\clp(\MM_T)$ valued random variables $\bar{\Gamma}^\eps$
as
\begin{equation}
\label{eq:gamma_bar}
    \bar{\Gamma}^\eps (A \times B \times C) = \frac{1}{T} \int_0^T 1_A(s) 1_B (\bar{Y}^\eps (s)) 1_C (v^\eps_1 (s)) ds, \, A\in \clb([0,T]), B \in \clb(\RR^d), C \in \clb(\RR^d).
\end{equation}
By denoting $(ds \times dy \times dz)$ as $d\textbf{v}$, the first equation in \eqref{eq:system-bar-def}  can now be written as
\begin{align}
    \bar{X}^\eps (t)  = x^0 + T \int_{\mathbb{M}_t}  b(\bar{X}^\eps (s),y) \bar{\Gamma}^\eps (d\textbf{v}) + s(\eps) \sqrt{\eps} \int_0^t \alpha(\bar{X}^\eps(s)) dW(s) + \int_0^t \alpha(\bar{X}^\eps(s))v_2^\eps(s) ds
    \label{eq:XQ}
\end{align}
We begin by establishing the following  useful moment bound.
\begin{lemma}
For some $\eps_0 \in (0,1),$

\begin{equation}
\nonumber
    \sup_{\eps \in (0,\eps_0)} \left[ \E \sup_{0 \leq t \leq T} \|\bar{X}^\eps(t)\|^2 + \E \int_0^T \|\bar{Y}^\eps(t)\|^2 dt + \eps \sup_{0 \leq t \leq T} \E \|\bar{Y}^\eps(t)\|^2     \right] < \infty.
\end{equation}
\label{system-upper-moment-bound}
\end{lemma}

\begin{proof}
	The constants $\kappa_i, i = 1,\dots,6$ in the proof will be positive reals that are independent of $t\in [0,T]$ and $\veps \in (0,1)$.
	Applying It\^{o}'s formula to $U(\bar X^{\veps}(t), \bar Y^{\veps}(t))$ we have for $t \in [0,T]$
	\begin{multline}\label{eq:eq1111}
	U(\bar X^{\veps}(t), \bar Y^{\veps}(t))	
	= U(x_0, y_0) + \int_0^t \nabla_x U(\bar X^{\veps}(s), \bar Y^{\veps}(s))\cdot [b(\bar X^{\veps}(s), \bar Y^{\veps}(s)) + \PZ{\alpha(\bar X^{\veps}(s))} v_2^{\veps}(s)]	ds\\
	\quad  + s(\veps)\veps^{1/2} \int_0^t \nabla_x U(\bar X^{\veps}(s), \bar Y^{\veps}(s)) \alpha(\bar X^{\veps}(s)) dW(s)
	+ \frac{s^2(\veps)\veps}{2}\int_0^t  \tr([\alpha^T\clh_xU\alpha](\bar X^{\veps}(s), \bar Y^{\veps}(s))) ds\\
	\quad - \frac{1}{\veps} \int_0^t \left\| \nabla_yU(\bar X^{\veps}(s), \bar Y^{\veps}(s))\right\|^2 ds +
	\frac{1}{\veps} \int_0^t \nabla_yU(\bar X^{\veps}(s), \bar Y^{\veps}(s))\cdot v_1^{\veps}(s) ds\\
	\quad + \frac{s(\veps)}{\veps^{1/2}} \int_0^t \nabla_yU(\bar X^{\veps}(s), \bar Y^{\veps}(s)) dB(s)
	+ \frac{s^2(\veps)}{2\veps} \int_0^t \tr (\clh_yU(\bar X^{\veps}(s), \bar Y^{\veps}(s))) ds.
	\end{multline}
	From our assumption that $b$ is Lipschitz, $\alpha$ is bounded and $\nabla_yU$ is Lipschitz,  it follows by standard Grönwall estimates that for each fixed $\veps>0$ and $t\in [0,T]$
	\begin{equation}\label{eq:mombd439}
		E[\|\bar X^{\veps}(t)\|^2 + \|\bar Y^{\veps}(t)\|^2] <\infty.
	\end{equation} 
	This, in particular, in view of linear growth of $\nabla_x U$ and $\nabla_yU$,  says that the expectation of the stochastic integrals in \eqref{eq:eq1111} is $0$.
	Next note that by Assumption \ref{Assumptions:system}(2c) 
	$$U(\bar X^{\veps}(t), \bar Y^{\veps}(t)) 
	\geq L^1_{low} \|\bar Y^{\veps}(t)\|^2 - L^2_{low}.$$
	Also, using the linear growth of $b$ and $\nabla_x U$, the boundedness of $\alpha$, and the fact that $v^{\eps}_2 \in \cla^k_{b,M}$, it follows that for some $\kappa_1 \in (0,\infty)$, 
\begin{align*}
	E\left\|\int_0^t \nabla_x U(\bar X^{\veps}(s), \bar Y^{\veps}(s))\cdot [b(\bar X^{\veps}(s), \bar Y^{\veps}(s)) + \alpha(\bar X^{\veps}(s))v_2^{\veps}(s)]\right\| 
	\le \kappa_1 \int_0^t E(1+ \| \bar Y^{\veps}(s)\|^2 + \| \bar X^{\veps}(s)\|^2).
\end{align*}
By Young's inequality
\begin{align*}
&- \frac{1}{\veps} \int_0^t \left\| \nabla_yU(\bar X^{\veps}(s), \bar Y^{\veps}(s))\right\|^2 ds +
	\frac{1}{\veps} \int_0^t \nabla_yU(\bar X^{\veps}(s), \bar Y^{\veps}(s))\cdot v_1^{\veps}(s) ds\\
	&\le - \frac{1}{2\veps} \int_0^t \left\| \nabla_yU(\bar X^{\veps}(s), \bar Y^{\veps}(s))\right\|^2 ds + \frac{M}{2\eps}.
\end{align*}
Combining the above observations with \eqref{eq:eq1111} we now have, for some $\kappa_2 \in (0,\infty)$,
\begin{align}
	L^1_{low}(E\| \bar Y^{\veps}(t)\|^2 -1) &\le
	\frac{\kappa_2}{\veps} + \kappa_2 \int_0^t  E(1+ \| \bar Y^{\veps}(s)\|^2 + \| \bar X^{\veps}(s)\|^2) ds
	- \frac{1}{2\veps} \int_0^t E\left\| \nabla_yU(\bar X^{\veps}(s), \bar Y^{\veps}(s))\right\|^2 ds.
	\label{eq:s1141}
\end{align}


Applying Assumption \ref{Assumptions:system}(2c)  once more we see that
\begin{equation}
\int_0^t E\left\| \nabla_yU(\bar X^{\veps}(s), \bar Y^{\veps}(s))\right\|^2 ds \ge 	
\int_0^t (L^1_{low} E\|\bar Y^{\veps}(s)\|^2 - L^2_{low}) ds.
\end{equation}
Using this inequality in \eqref{eq:s1141} and rearranging terms, for some $\kappa_3, \kappa_4 \in (0,\infty)$
\begin{align*}
	L^1_{low} \int_0^t E\| \bar Y^{\veps}(s)\|^2 ds 
	&\le \kappa_3 + \kappa_4\veps \int_0^t  E(\| \bar Y^{\veps}(s)\|^2 + \| \bar X^{\veps}(s)\|^2) ds.
\end{align*}
Choose  $\veps_1 \in (0,1)$ such that $\kappa_4\veps_1 \le \frac{L^1_{low}}{2}$.
Then there is a $\kappa_5 \in (0,\infty)$ such that for all $\veps \in (0, \veps_1)$ and $t \in [0,T]$
\begin{equation*}
	\int_0^t E\| \bar Y^{\veps}(s)\|^2 ds \le \kappa_5 (1 + \veps \sup_{0\le s \le t} E\| \bar X^{\veps}(s)\|^2).
\end{equation*}
Next, by standard Grönwall estimates on the first equation in \eqref{eq:system-bar-def}   there is a $\kappa_6 \in (0, \infty)$, such that for all $\veps \in (0,1)$ and $t \in [0,T]$
\begin{equation}
	\label{eq:mombd441}
	E\sup_{0\le s \le t} \|\bar X^{\veps}(s)\|^2 \le \kappa_6 E \int_0^t (1+ \|\bar Y^{\veps}(s)\|^2) ds.
\end{equation}	
Combining the last two estimates we see that with $\veps_0 = \min(\veps_1, 1/(2\kappa_5\kappa_6))$
\begin{equation*}
\sup_{\veps \in (0, \veps_0)}	\left(\int_0^ T E\| \bar Y^{\veps}(s)\|^2 ds + E \sup_{0\le s \le T}  \| \bar X^{\veps}(s)\|^2\right)  < \infty.
\end{equation*}
Finally, using the above estimate in \eqref{eq:s1141} we obtain 
$$\sup_{\veps \in (0,\veps_0)}
	 \veps \sup_{0\le t \le T} E\|\bar Y^{\veps}(t)\|^2 <\infty.
	 $$
	 This completes the proof of the lemma.
\end{proof}

For the rest of this subsection, we will assume  that $\eps \in (0, \eps_0)$ where
 $\eps_0$ is as in  the statement of previous lemma.
  In the next result we establish the tightness of various objects of interest. Recall the definition of  $S_M^p$, for $p \in \NN$, from \eqref{eq:954}.
\begin{proposition}
\label{system-upper-tightness}
The collection $(\bar{\Gamma}^\eps,\bar{X}^\eps,v_2^\eps)$ is tight in $\mathcal{P} (\mathbb{M}_T) \times \clx \times S^k_M$.
\end{proposition}

\begin{proof}
The tightness of $\{v_2^\eps\}$ is immediate from the compactness of $S_M^k$.
In order to prove the tightness of $\{\bar{\Gamma}^\eps \}$, it suffices to show that the collection of non-random probability measures on $\MM_T$ defined as $\bar\gamma^\eps(A) \doteq  \E \bar\Gamma^\eps(A)$, $A \in \clb(\MM_T)$ is relatively compact (cf. \cite[Theorem 2.11]{buddupbook}).
In turn, to prove the tightness of $\{\bar\gamma^\eps\}$ it suffices to prove the tightness of its three marginals. The first marginal $ [\bar\gamma^\eps]_1$ is the normalized Lebesgue measure on $[0,T]$ for each $\eps$ so it is automatically tight.
For the second marginal we have,  for $R \in (0,\infty)$
$$[\bar\gamma^\eps]_2(y \in \RR^d: \|y\| \ge R) = \frac{1}{T} \int_0^T P(\|\bar Y^{\eps}(t)\| \ge R) \le \frac{1}{TR^2}\int_0^T E\|\bar Y^{\eps}(t)\| ^2 dt.$$
Tightness of $\{[\bar\gamma^\eps]_2\}$ is now immediate from Lemma \ref{system-upper-moment-bound}.
Also, for the third marginal, for $R \in (0,\infty)$
$$[\bar\gamma^\eps]_3(z \in \RR^d: \|z\| \ge R) = \frac{1}{T} \int_0^T P(\|v_1^{\eps}(t)\| \ge R) \le
\frac{1}{TR^2}\int_0^T E\|v_1^{\eps}(t)\| ^2 dt \le \frac{M}{TR^2}.$$
The tightness of $\{[\bar\gamma^\eps]_3\}$  follows. Thus we have shown the tightness of $\{\bar{\Gamma}^\eps \}$.

Finally consider $\{\bar{X}^\eps\}$. 
Write $\bar{X}^\eps =  \bar \clb^{\eps} + \bar \cla^{\eps}$, where, for $ \in [0,T]$,
$$\bar \clb^{\eps}(t) = x_0 + \int_0^t b(\bar X^{\veps}(r), \bar Y^{\veps}(r))dr + \int_0^t \alpha(\bar X^{\veps}(r)) v_2^{\veps}(r) dr$$
and
$\bar \cla^{\eps}(t) = s(\eps)\sqrt{\eps}\int_0^t \alpha(\bar X^{\veps}(s)) dW(s)$.
The tightness of $\{\bar \clb^{\eps} \}$ in $\clx$ is immediate from the moment bounds in Lemma \ref{system-upper-moment-bound}, the linear growth of $b$, the boundedness of $\alpha$ and since $v_2^{\eps} \in \cla^r_{b,M}$, while the tightness of $\bar \cla^{\eps}$ in $\clx$  follows from the boundedness of $\alpha$. This proves the tightness of $\{\bar{X}^\eps\}$ in $\clx$ and completes the proof of the lemma.
\end{proof}

By Lemma \ref{system-upper-tightness}, it follows that every subsequence has a further  subsequence along which $(\bar\Gamma^{\veps}, \bar X^{\veps}, v_2^{\veps})$ converges in distribution   to $(\bar \Gamma, \bar X,  v_2)$. 
We disintegrate the measure $\bar \Gamma$ as follows
\begin{equation}\label{eq:eq859}
	\bar \Gamma(dt\, dy\, dz) = \frac{1}{T} dt\, \hat{\gamma}_t(dy)\, q(t, y, dz) = \gamma(dt\, dy) q(t,y,dz) =  \frac{1}{T}  dt\, \hat{\Gamma}_t(dy\, dz).
\end{equation}
\PZ{In the above identity, $\gamma(dt\, dy) = \bar \Gamma(dt\times dy\times \RR^d)$ is the marginal distribution of $\bar \Gamma$ on the first two coordinates and $q$ is the r.c.p.d. on the third coordinate given the first two coordinates. Also, since 
$\bar \Gamma(A \times \RR^d \times \RR^d) =  \gamma(A\times \RR^d)=\frac{1}{T}\lambda(A)$, where $\lambda$ is the Lebesgue measure and $A \in \clb([0,T])$, the probability measure $\gamma$ can be disintegrated as $\gamma(dt\, dy) = \frac{1}{T} dt\, \hat{\gamma}_t(dy)$, which give the first two identities in the display. For the third identity we disintegrate the probability measure $\bar \Gamma$ as the marginal on the first coordinate (which is the normalized Lebesgue measure $\frac{1}{T} dt$) and the r.c.p.d. on the last two coordinates given the first coordinate, denoted as $\hat{\Gamma}_t(dy\, dz)$}. 
 The following lemma gives a characterization of the limit points $(\bar \Gamma, \bar X,  v_2)$.
 Recall the class $\clu(\xi, \nu)$ defined in Section \ref{sec:m2}. 
\begin{lemma}
\label{system-upper-representation-xi}
Let $(\bar \Gamma, \bar X,  v_2)$ be a weak limit point of $(\bar\Gamma^{\veps}, \bar X^{\veps}, v_2^{\veps})$. Define $\clm_1-$valued random variable
$\bar \Lambda$ as
$$\bar \Lambda([0,t]\times A) \doteq T\, \bar \Gamma([0,t]\times A \times \RR^d), \mbox{ for } t \in [0,T], A \in \clb(\RR^d).$$
Then
$v_2 \in \clu(\bar X, \bar \Lambda)$ a.s.
\end{lemma}

\begin{proof}
	From \eqref{eq:XQ}, for $t \in [0,T]$,
\begin{align}
	\bar X^{\veps}(t) 
	&= x_0 + T\int_{\MM_t} b(\bar X^{\veps}(s), y) \bar{\Gamma}^{\eps}(d\textbf{v})
	 + \int_0^t \alpha(\bar X^{\veps}(s)) v_2^{\veps}(s) ds + \clr^{\veps}(t),
	\label{eq:341nn}
\end{align}	
where $\clr^{\veps}(t) \doteq s(\veps)\veps^{1/2} \int_0^t \alpha(\bar X^{\veps}(s)) dW(s)$.
Note that
$$E\sup_{0\le t \le T}\|\clr^{\veps}(t)\|^2 \le  4Ts^2(\veps) \veps \|\alpha\|_{\infty}^2 \to 0, \mbox{ as } \veps \to 0.$$
	We assume without loss of generality that convergence of $(\bar\Gamma^{\veps}, \bar X^{\veps}, v_2^{\veps},
	\clr^{\veps})$ to $(\bar\Gamma, \bar X, v_2,
	0)$ in $\clp(\MM_T)\times \clx \times S_M^k\times \clx$
	holds along the 
	full sequence, and, by appealing to Skorohod representation theorem,  that the convergence holds a.s.
We need to show that $(\bar{\Gamma},\bar{X},{v}_2)$ satisfy  a.e., for all $t \in [0,T]$
	\begin{equation}
	    \bar{X}(t)  = x^0 + T\int_{\mathbb{M}_t} b(\bar{X}(s),y) \bar{\Gamma}(d\textbf{v}) +   \int_0^t  \alpha(\bar X(s))  {v}_2(s) ds.
	    \label{eq:limitQ}
	\end{equation}
Note that
\begin{align}	\label{eq:341n}
	\bar X^{\veps}(t) 
	&= x_0 + T\int_{\MM_t} b(\bar X(s), y) \bar{\Gamma}^{\eps}(d\textbf{v})
	 + \int_0^t \alpha(\bar X(s)) v_2^{\veps}(s) ds + \clr^{\veps}(t) + \clr_1^{\veps}(t),
\end{align}	
where
$$\clr_1^{\veps}(t) \doteq T\int_{\MM_t} (b(\bar X^{\eps}(s), y) - b(\bar X(s), y))\bar{\Gamma}^{\eps}(d\textbf{v}) + \int_0^t (\alpha(\bar X^{\eps}(s)) -  \alpha(\bar X(s))) v_2^{\veps}(s) ds .
$$
Using the Lipschitz property of $b$ and $\alpha$  from Assumption \ref{Assumptions:system}(1)
\begin{align*}
\sup_{0\le t \le T}\|\clr_1^{\veps}(t)\| &\le \sup_{0\le t \le T}  T\left\|
\int_{\MM_t} (b(\bar X^{\veps}(s), y) -  b(\bar X(s), y))\bar{\Gamma}^{\eps}(d\textbf{v})\right\|
+  \int_0^T \left\|\alpha(\bar X^{\eps}(s)) -  \alpha(\bar X(s))\right\|  \|v_2^{\veps}(s) \| ds
\\
&\le T(L_b + L_{\alpha}\sqrt{TM} )\sup_{0\le t \le T} \| \bar X^{\veps}(t) - \bar X(t)\| \to 0, \mbox{ as } \veps \to 0.
\end{align*}
From Lemma \ref{system-upper-moment-bound}, and since $v_1^{\eps} \in \cla^d_{b,M}$, 
\begin{equation}
	\label{eq:eq402}
		\sup_{\veps} E \int_{\MM_T} (\|y\|^2 + \|z\|^2) \bar{\Gamma}^{\eps}(d\textbf{v}) <\infty.
\end{equation}
Also, by convergence of $\bar{\Gamma}^{\eps}$ to $\bar{\Gamma}$, the continuity  of the map $(s,y,z) \mapsto b(\bar X(s), y)$, 
the linear growth of $b$ and the estimate in \eqref{eq:eq402}, it follows that, for each $t \in [0,T]$
$$
\lim_{\eps \to 0} \int_{\MM_t} b(\bar X(s), y) \bar{\Gamma}^{\eps}(d\textbf{v})
= \int_{\MM_t} b(\bar X(s), y) \bar{\Gamma}(d\textbf{v}).$$
Finally, for each $t \in [0,T]$, as $\eps \to 0$,
$$ \int_0^t \alpha(\bar X(s)) v_2^{\veps}(s) ds  \to \int_0^t \alpha(\bar X(s)) v_2(s) ds.$$
The result now follows on sending $\eps \to 0$ in \eqref{eq:341n}.
\end{proof}

The following lemma gives an important inequality for the costs that will be useful for the proof of the upper bound.
Recall the disintegration in \eqref{eq:eq859}.
\begin{lemma}
	Let $(\bar \Gamma, \bar X,  v_2)$ be as in Lemma \ref{system-upper-representation-xi}.
The following inequality holds a.s.
\begin{equation}
    \nonumber
    \int_{\mathbb{M}_T} \|z\|^2 \bar{\Gamma}(dt\,  dy\,  dz) \geq \frac{1}{T}\int_0^T \left( \int_{\mathbb{R}^d} \| \nabla_y U (\bar{X}(t),y)\|^2 \hat{\gamma}_t(dy)  \right) dt.
\end{equation}
\label{system-upper-inequality-representation}
\end{lemma}

\begin{proof}
%
%
Let $\eta: \RR^d \to \RR$ be in $\clc_c^2$. Then It\^{o}'s formula applied to $\eta(\bar{Y}^\eps)$ gives
\begin{equation}
    \begin{aligned}
        \eta(\bar{Y}^\eps(t)) - \eta(y^0) &= \frac{s(\eps)}{\sqrt{\eps}} \int_0^t \nabla \eta(\bar{Y}^\eps(s)) dB(s) 
        -\frac{1}{\eps} \int_0^t \nabla \eta(\bar{Y}^\eps(s)) \cdot \nabla_y U(\bar{X}^\eps(s),\bar{Y}^\eps(s))  ds  \\ &+\frac{1}{\eps} \int_0^t \nabla \eta (\bar{Y}^\eps(s)) \cdot  v_1^\eps(s)  ds 
        + \frac{s^2(\eps)}{2\eps} \int_0^t \Delta \eta (\bar{Y}^\eps(s))  ds.
    \end{aligned}
\end{equation}
Multiplying by $\eps$ in the above equation and recalling the definition of the random measure $\bar{\Gamma}^\eps$, we have,
\begin{equation}
    \nonumber
    \begin{aligned}
        \eps \eta(\bar{Y}^\eps(t)) -\eps \eta(y^0) &= s(\eps)\sqrt{\eps} \int_0^t \nabla \eta(\bar{Y}^\eps(s)) dB(s) \\
       &- T\int_{\mathbb{M}_t}  \nabla \eta(y) \cdot ( \nabla_y  U(\bar{X}^\eps(s),y)- z)  \bar{\Gamma}^\eps(d\textbf{v}) + \frac{s^2(\eps)}{2} \int_0^t \Delta \eta (\bar{Y}^\eps(s))  ds.
    \end{aligned}
\end{equation}
Sending $\eps \rightarrow 0$ in the above display we have that, as $\eps \to 0$,
\begin{equation}
	\int_{\mathbb{M}_t}  \nabla \eta(y) \cdot ( \nabla_y  U(\bar{X}^\eps(s),y)- z) \bar{\Gamma}^\eps(d\textbf{v})  \to 0 \mbox{ in probability}.
\end{equation}
As in the proof of Lemma \ref{system-upper-representation-xi} we assume without loss of generality that
convergence of $(\bar\Gamma^{\veps}, \bar X^{\veps})$ to $(\bar\Gamma, \bar X)$
	holds along the 
	full sequence in a.s. sense.
From the Lipschitz property of $\nabla_y  U$ we now see that, as $\eps \to 0$,
$$
\int_{\mathbb{M}_t} \nabla \eta(y) \cdot ( \nabla_y  U(\bar{X}^\eps(s),y)- \nabla_y  U(\bar{X}(s),y)) \bar{\Gamma}^\eps(d\textbf{v})  \to 0.$$
Finally,   from the convergence of $\bar{\Gamma}^\eps$ to $\bar{\Gamma}$, the square integrability in \eqref{eq:eq402}, and the compact support property of $\eta$ we see that for all $t \in [0,T]$
\begin{equation}
    \nonumber
    \int_{\mathbb{M}_t} \nabla \eta (y) \cdot ( \nabla_y U (\bar{X}(s), y) - z) \bar{\Gamma}^{\eps} (d\textbf{v}) \to
	 \int_{\mathbb{M}_t}  \nabla \eta (y) \cdot (\nabla_y U (\bar{X}(s), y) - z)  \bar{\Gamma}(d\textbf{v}).
\end{equation}
Combining the last three convergence statements we have, a.s., for all $t\in [0,T]$, 
\begin{equation}
	 \label{eq:s2e8} \int_{\mathbb{M}_t}  \nabla \eta (y) \cdot ( \nabla_y U (\bar{X}(s), y) - z ) \bar{\Gamma}(d\textbf{v})=0.\end{equation}
Recall the disintegration in \eqref{eq:eq859} and define
$$\bar v_1(t,y) = \int_{\RR^d} z \, q(t,\, y,\, dz), \; t \in [0,T], y \in \RR^d.$$
Note that the integral is well defined a.s. for $\gamma$ a.e. $(t, y)$, since from \eqref{eq:eq402} and Fatou's lemma
\begin{equation}\label{eq:fatest}
	E \int_{[0,T]\times \RR^{2d}} (\|y\|^2 + \|z\|^2) \bar{\Gamma}(d\textbf{v}) =
	\frac{1}{T} E\int_{0}^T \int_{\RR^{2d}} (\|y\|^2 + \|z\|^2) \hat{\Gamma}_t(dy\, dz)  dt <\infty .
\end{equation}
From \eqref{eq:s2e8}, by a standard separability argument, a.s., for every $\eta \in \clc_c^2$ and $t\in [0,T]$
\begin{equation}\label{eq:1229}
\int_0^t \left[\int_{\RR^d} \left[\nabla_y U(\bar X(s),y) - \bar v_1(s,y)\right]\cdot \nabla \eta(y) \hat \gamma_s(dy)\right] ds=0.
\end{equation}
Also, from \eqref{eq:fatest}, a.s., for a.e. $t \in [0,T]$, 
\begin{equation}\label{eq:eq304}
	\int_{\RR^{2d}} (\|y\|^2 + \|z\|^2) \hat \Gamma_t(dy\, dz)   <\infty, 
\end{equation}
and for every $\eta \in \clc_c^2$,
\begin{equation}\label{eq:1229b}
\int_{\RR^d} \left[\nabla_y U(\bar X(t),y) - \bar v_1(t,y)\right]\cdot \nabla \eta(y) \hat \gamma_t(dy) =0.
\end{equation}
Fix $t \in [0,T]$ for which the above two equations  holds and denote $f(y)\doteq U(\bar X(t),y) $.
Now a similar argument as used in the proof of \eqref{eq:eq321} shows that we can replace $\eta$ with $f$ in the above identity.
Indeed, from Assumption \ref{Assumptions:system} the function $f$ satisfies the conditions  in Lemma \ref{single-upper-lemma-orthogonality-extension}. Also, from properties of $U$, an estimate similar to \eqref{eq:333} shows that, for some $\kappa_1 \in (0,\infty)$,
\begin{multline}
 \int_{\RR^{d}} \left|\left[\nabla f(y) - \bar v_1(t,y)\right]\cdot \nabla f(y)\right| \hat\gamma_t(dy)=	 \int_{\RR^{d}} \left|\left[\nabla_y U(\bar X(t),y) - \bar v_1(t,y)\right]\cdot \nabla f(y)\right| \hat\gamma_t(dy)\\
	 \le \kappa_1 \int_{\RR^d} (1+ \|y\|^2 + \|\bar v_1(t,y)\|^2) \hat \gamma_t(dy)
	 \le \kappa_1 \int_{\RR^{2d}} (1+ \|y\|^2 + \|z\|^2)  \hat \Gamma_t(dy\, dz) < \infty\label{eq:333n}
\end{multline}
where the second inequality on the second line is from Jensen's inequality while the finiteness asserted in the last inequality is from \eqref{eq:eq304}.
Now exactly as in \eqref{eq:849n} and discussion below it we see that \eqref{eq:1229b} holds with $\eta$ replaced by $f(\cdot) = U(X(t), \cdot)$, namely 
\begin{equation}
	\int_{\RR^d} \left[\nabla_y U(\bar X(t),y) - \bar v_1(t,y)\right]\nabla_y U(X(t), y) \hat \gamma_t(dy) =0 \mbox{ a.s.} \label{eq:eq321b}
\end{equation}
%
Now the proof is completed as in Lemma \ref{lem:ineqincos}:
\begin{multline*}
	\int_{\RR^{2d}} \|z\|^2 \hat \Gamma_t(dy\, dz) = \int_{\RR^{2d}} \|z\|^2 q_t(y,\, dz) \hat \gamma_t(dy)
	 \ge \int_{\RR^d} \|\bar v_1(t,y)\|^2 \hat \gamma_t(dy)\\
	= \int_{\RR^d}  \|\bar v_1(t,y) - \nabla_y U(\bar X(t),y)\|^2 \hat \gamma_t(dy)
	+ \int_{\RR^d}  \|\nabla_y U(\bar X(t),y)\|^2 \hat \gamma_t(dy)\\
	\quad + 2\int_{\RR^d}  (\bar v_1(t,y) - \nabla_y U(\bar X(t),y))\cdot \nabla f(y) \hat \gamma_t(dy).
\end{multline*}
Since the third term equals $0$ from \eqref{eq:eq321b}, the lemma follows.
\end{proof}
\subsubsection{Proof of the LDP upper bound.}
We now complete the proof of the inequality in \eqref{ineq-two-upper}. 
Recall the weak limit point $(\bar \Gamma, \bar X,  v_2)$  of $(\bar\Gamma^{\veps}, \bar X^{\veps}, v_2^{\veps})$ as in Lemmas \ref{system-upper-representation-xi} and \ref{system-upper-inequality-representation}. By a standard subsequential argument we can assume without loss of generality that the convergence holds along the full sequence.
From \eqref{system-upper-variational} (and the observation below it)
\begin{equation*}
\begin{split}
	\liminf_{\veps \to 0}- s^2(\veps)\veps &E e^{-\frac{F(X^{\veps}, \La^{\eps})}{s^2(\veps)\veps}}
	 \ge \liminf_{\veps \to 0} E \left [ F(\bar X^{\veps}, \La^{\eps}) + \frac{1}{2}\sum_{i=1,2}\int_0^T \|v_i^{\veps}(t)\|^2 dt\right] - 2\delta\\
	 &= \liminf_{\veps \to 0} E \Big [ F(\bar X^{\veps}, \bar \La^{\eps}) + \frac{T}{2} \int_{\MM_T} \|z\|^2 \bar\Gamma^{\veps}( d\textbf{v})
		+ \frac{1}{2} \int_{0}^T \|v_2^{\veps}(t)\|^2 dt \Big ] - 2\delta\\
	&\ge E \left [ F(\bar X, \bar \La) + \frac{T}{2} \int_{\MM_T} \|z\|^2 \bar \Gamma( d\textbf{v})
		+ \frac{1}{2} \int_{0} ^T\|  v_2(t)\|^2 dt\right] - 2\delta\\
	&\ge {E \left [ F(\bar X, \bar \La) + \frac{1}{2}\int_0^T 
			\left (\int_{\RR^{d}} \| \nabla_y U(\bar X(t), y)\|^2 \hat\gamma_t(dy)\right) dt
		+ \frac{1}{2} \int_{0} ^T\|  v_2(t)\|^2 dt\right] - 2\delta}\\	
	&\ge E \left [ F(\bar X , \bar \La) + I_2(\bar X, \bar \La)\right]-2\delta \ge \inf_{(\xi, \nu) \in \clx \times \clm_1} \left [ F(\xi, \nu) + I_2(\xi, \nu)\right]-2\delta.
\end{split}
\end{equation*}
where the third line uses lower semicontinutiy  of the $L^2$ norm and Fatou's lemma, the fourth line uses Lemma \ref{system-upper-inequality-representation}, the fifth line uses the definition of $I_2$, the definition of $\bar \La$, and the property $v_2 \in \clu(\bar X, \bar \La)$ a.s. shown in Lemma \ref{system-upper-representation-xi}. Since $\delta>0$ is arbitrary, the result follows.
\hfill \qed

\subsection{Simple form near optimal paths.}
\label{sec:lowbd}
In preparation for the proof of the LDP lower bound we first prove a preliminary result which provides simple form near optimal paths that can then be well approximated by suitable controlled processes.
\begin{lemma}\label{lem:approxdisc}
	For each $\delta_0 \in (0,1)$ and a bounded Lipschitz $F: \clx \times \clm_1 \to \RR$, there is a $\xi^* \in \clx$, $\nu^* \in \clm_1 $, 
	$v^* \in L^2([0,T]:\RR^k)$, a partition
	$$0=t_0<t_1 \cdots < t_{K+1}=T$$ of $[0,T]$
	and probability measures $\nu^*_i$, $i=0, 1, \ldots , K$ on $\RR^{d}$ 
	with finite support, such that
	\begin{enumerate}
		\item $\nu^*(dy\,ds) = \hat\nu_s(dy) ds$, with $\hat\nu_t \doteq \nu^*_0 1_{\{0\}}(t)+ \sum_{i=0}^K \nu^*_i 1_{(t_i, t_{i+1}]}(t)$, $0\le t \le T$.
		\item $v^*(s) = v^*(t_i) \doteq v^*_i$ for $s \in [t_i, t_{i+1})$,  $i=0, 1, \ldots K$.
		\item For all  $t \in [0,T]$,
		$$\xi^*(t) = x_0 + \int_0^t \int_{\RR^{d}} b(\xi^*(s), y) \hat\nu_s(dy) ds + \int_0^t \alpha(\xi^*(s)) v^*(s) ds.$$
		\item  $(\xi^*, \nu^*) $ is $\delta_0$-optimal, i.e.,
		\begin{align*} 
			& F(\xi^*, \nu^*) +  \frac{1}{2}  \int_0^T \left[
			\left (\int_{\RR^{d}} \| \nabla_y U(\xi^*(s), y)\|^2 \hat\nu_{s}(dy)\right)+ \|v^*(s)\|^2\right] ds\\
	&\le \inf_{(\xi, \nu) \in \clx \times  \clm_1} \left [ F(\xi, \nu) + I_2(\xi, \nu)\right] + \delta_0.\end{align*}
	\end{enumerate}
\end{lemma}
We now proceed with the proof of the lemma. This proof  will be completed at the end of Section \ref{sec:discappr} by constructing a series of approximations for a near optimal $(\tilde \xi, \tilde \nu)$.

Fix a bounded Lipschitz function $F:\clx \times \clm_1 \to \RR$ and $\delta_0\in (0,1)$. 
Choose $(\tilde\xi, \tilde\nu) \in \clx\times \clm_1$ such that
\begin{equation}
	\label{eq:lowbd1}
	F(\tilde\xi, \tilde\nu) + I_2(\tilde\xi, \tilde\nu) \le \inf_{(\xi, \nu) \in \clx \times  \clm_1} \left [ F(\xi, \nu) + I_2(\xi, \nu)\right] +  \delta_0/5.
\end{equation}
Next, using the definition of $I_2$, choose $\tilde v \in \clu (\tilde\xi, \tilde\nu)$ such that, with $\tilde\nu(dy\,ds) = {\tilde \nu_s(dy)} ds$ 
\begin{equation}
	\label{eq:lowbd2}
	\frac{1}{2}  \int_0^T \left[
			\left (\int_{\RR^{d}} \| \nabla_y U(\tilde \xi(s), y)\|^2 \hat\nu_s(dy)\right)
			+  \|\tilde v(s)\|^2\right] ds \le I_2(\tilde\xi, \tilde\nu) + \delta_0/4.
\end{equation}
Note that $\tilde \xi, \tilde \nu, \tilde v$ satisfy
$$\tilde \xi(t) = x_0 + \int_0^t \int_{\RR^d} b(\tilde \xi(s), y) \hat \nu_s(dy) ds + \int_0^t  \alpha(\tilde \xi(s))\tilde v(s) ds, \; t \in [0,T].$$
From the lower bound on $\|\nabla_yU(x,y)\|^2$ in Assumption \ref{Assumptions:system}(2c) we see that
\begin{equation}\label{eq;307n}
	 \int_0^T\int_{\RR^{d}} \|y\|^2  {\tilde\nu_s(dy)} ds \le A_1,
\end{equation}
where $A_1 \doteq  (L^1_{low})^{-1} (2(\inf_{(\xi, \nu) \in \clx \times  \clm_1} \left [ F(\xi, \nu) + I_2(\xi, \nu)\right] + \|F\|_{\infty} + 1)+L^2_{low}T)$.
We also remark that $\hat \nu_0$ can be taken to be an arbitrary probability measure on $\RR^d$ and we will assume without loss of generality that
\begin{equation}\label{eq;313n}\int_{\RR^d} \|y\|^2 \hat \nu_0(dy) = 1.\end{equation}
We now proceed to our series of approximations.

\subsubsection{Approximating with continuous $\nu, v$.}
Fix $\delta  \in (0,1)$. The choice of $\delta$ will be identified at the end of this section.  
Using the uniform continuity  of $\tilde \xi$, choose $0<\eta<\delta^2$ such that
\begin{equation}\label{eq:425n}
	\|\tilde \xi(s)- \tilde \xi(s')\| \le \delta \mbox{ whenever } |s-s'|\le \eta.
\end{equation}
Let
$$v_{\eta}^*(s) = \frac{1}{\eta} \int_{s-\eta}^s \tilde v(r) dr,\;  \mu_{\eta,s}^* = \frac{1}{\eta} \int_{s-\eta}^s {\tilde\nu}_{r} dr, \;  s \in [0,T],$$
where $\tilde v(r)\doteq 0$ and $\hat\nu_r \doteq \hat\nu_0$ for $r \le 0$.
Note that $v_{\eta}^*$ and $\mu_{\eta}^*$ are continuous maps on $[0,T]$ with values in $\RR^k$ and $\clp(\RR^d)$ respectively.
Let
$\xi^*_{\eta}$ be given as the solution of
\begin{equation}\label{eq:xe}
	\xi^*_{\eta}(t) = x_0 + \int_0^t \int_{\RR^d} b(\xi^*_{\eta}(s), y) \mu_{\eta,s}^*(dy) ds + \int_0^t \alpha (\xi^*_{\eta}(s))v_{\eta}^*(s) ds, \; t \in [0,T].\end{equation}
Note that due to the Lipschitz property of $b$  and $\alpha$ the above equation has a unique solution.
Also note that $\tilde \xi$ can be represented as
\begin{equation}\label{eq:xe2}
	\tilde \xi(t) = x_0 + \int_0^t \int_{\RR^d} b(\tilde \xi(s), y) \mu_{\eta,s}^*(dy) ds + \int_0^t \alpha(\tilde \xi(s)) v_{\eta}^*(s) ds -
\clr_{1,\eta}(\tilde \xi,t), \; t \in [0,T],\end{equation}
where
\begin{align*}
	\clr_{1,\eta}(\tilde\xi, t) &= \int_0^t \int_{\RR^d} b(\tilde \xi(s), y) \mu_{\eta,s}^*(dy) ds - \int_0^t \int_{\RR^d} b(\tilde \xi(s), y) {\tilde\nu_s(dy)} ds\\
	&\quad + \int_0^t \alpha(\tilde \xi(s))v_{\eta}^*(s) ds - \int_0^t \alpha(\tilde \xi(s))\tilde v(s) ds.
\end{align*}
Using  an 
interchange of the order of integration.
\begin{multline*}
\int_0^t \int_{\RR^d} b(\tilde \xi(s), y) \mu_{\eta,s}^*(dy) ds - \int_0^t \int_{\RR^d} b(\tilde \xi(s), y) {\tilde\nu_s(dy)} ds
= \int_{-\eta}^{0} \int_{\RR^d} \left(\frac{1}{\eta} \int_{0}^{(r+\eta)\wedge t} b(\tilde \xi(s), y) ds\right) {\tilde\nu_0(dy)} dr\\
+ \int_{0}^{(t-\eta)^+} \int_{\RR^d} \left(\frac{1}{\eta} \int_{r}^{r+\eta} b(\tilde \xi(s), y) ds\right) \hat\nu_{r}(dy) dr - \int_0^{(t-\eta)^+} \int_{\RR^d} b(\tilde \xi(s), y) {\tilde\nu_s(dy)} ds\\
+ \int_{(t-\eta)^+}^{t} \int_{\RR^d} \left(\frac{1}{\eta} \int_{r}^{t} b(\tilde \xi(s), y) ds\right) \hat\nu_{r}(dy) dr
 - \int_{(t-\eta)^+}^{t} \int_{\RR^d} b(\tilde \xi(s), y) {\tilde\nu_s(dy)}\, ds.
 \end{multline*}
From the Lipschitz property of $b$
$$\left\|\frac{1}{\eta} \int_{r}^{r+\eta} b(\tilde \xi(s),y) ds - b(\tilde \xi(r), y) \right\| \le L_b \delta, r \in [0, (t-\eta)^+].$$
Combining this with the linear growth of $b$, for some $C(b) \in (0,\infty)$ depending only on the function $b$
$$
\left\|\int_0^t \int_{\RR^d} b(\tilde \xi(s), y) \mu_{\eta,s}^*(dy) ds - \int_0^t \int_{\RR^d} b(\tilde \xi(s), y) {\tilde\nu_s(dy)} ds\right\| \le \left(C(b)(1+ \|\tilde \xi\|_{\infty}) \eta + TL_b \delta\right).$$
A similar calculation shows that
$$
\left\| \int_0^t \alpha(\tilde \xi(s))v_{\eta}^*(s) ds - \int_0^t \alpha(\tilde \xi(s))\tilde v(s) ds\right\|
\le \delta(L_{\alpha}\sqrt{T} +2\|\alpha\|_{\infty}) \|\tilde v\|_2.$$
Thus, for all $t \in [0,T]$,
$$\|\clr_{1,\eta}(\tilde \xi, t)\| \le \left(C(b)(1+ \|\tilde \xi\|_{\infty}) \eta +  \delta(L_{\alpha}\sqrt{T} +2\|\alpha\|_{\infty}) \|\tilde v\|_2 + TL_b \delta\right) \le
\kappa_1 \delta$$
where $\kappa_1 \doteq C(b)(1+ \|\tilde \xi\|_{\infty}) + T L_b + (L_{\alpha}\sqrt{T} +2\|\alpha\|_{\infty}) \|\tilde v\|_2)$.
Combining the above estimate with \eqref{eq:xe} and \eqref{eq:xe2}, and using the Lipschitz property of $b$, we have by Grönwall's lemma
\begin{equation}\label{eq:536n}
	\|\xi^*_{\eta} - \tilde\xi\|_{\infty} \le \kappa_2\delta,
\end{equation}
where $\kappa_2 \doteq \kappa_1 e^{L_bT + T^{1/2} L_{\alpha}\|\tilde v\|_2}$.

Now we estimate the cost.
From \eqref{eq;307n}  and \eqref{eq;313n} we see that
\begin{equation}\label{eq:902n}
	 \int_0^T\int_{\RR^{d}} \|y\|^2\mu_{\eta,s}^*(dy) ds \le A_1+1.
\end{equation}
Using the Lipschitz 
 property of $\nabla_y U$ we see that for a $C(U) \in (0,\infty)$ depending only on $U$,
 $$\| \nabla_y U(x, y)\| ^2\le C(U)(1+\|x\|^2 +\|y\|^2)$$
  and
 \begin{equation}\label{eq:424n}
 	 \left|\| \nabla_y U(x, y)\|^2  - \| \nabla_y U(x', y)\|^2 \right| \le C(U)\|x-x'\|(1+\|x\| +\|x'\|+\|y\|) \mbox{ for all } x, x' \in \RR^m, y \in \RR^d.
 \end{equation}
 From this it follows that
 \begin{align*}
	&\left|\frac{1}{2}\int_0^T \int_{\RR^{d}} \| \nabla_y U(\xi^*_{\eta}(s), y)\|^2 \mu_{\eta,s}^*(dy) ds
	- \frac{1}{2}\int_0^T \int_{\RR^{d}} \| \nabla_y U(\tilde\xi(s), y)\|^2 \mu_{\eta,s}^*(dy) ds\right|\\
	&\le \kappa_2C(U)\delta \int_0^T \int_{\RR^{d}} (1 +\|y\|+ \|\tilde\xi\|_{\infty} + \kappa_2)\mu_{\eta,s}^*(dy) ds
	\le \kappa_2C(U)\delta (T(2+ \|\tilde\xi\|_{\infty} + \kappa_2) + A_1+1).
 \end{align*}
 Thus with $\kappa_3 \doteq \kappa_2C(U) (T(2+ \|\tilde\xi\|_{\infty} + \kappa_2) + A_1+1)$,
 \begin{equation}
 	\frac{1}{2}\int_0^T \int_{\RR^{d}} \| \nabla_y U(\xi^*_{\eta}(s), y)\|^2 \mu_{\eta,s}^*(dy) ds
	\le \frac{1}{2} \int_0^T \int_{\RR^{d}} \| \nabla_y U(\tilde\xi(s), y)\|^2 \mu_{\eta,s}^*(dy) ds + \kappa_3\delta.
 \end{equation}
 Also, by an interchange of order of integration, and using \eqref{eq;313n}
 \begin{multline*}
 	\int_0^T \int_{\RR^{d}} \| \nabla_y U(\tilde\xi(s), y)\|^2 \mu_{\eta,s}^*(dy) ds  =
	\int_{-\eta}^T \int_{\RR^{d}}  \frac{1}{\eta} \int_{r \vee 0}^{(r+\eta)\wedge T} \| \nabla_y U(\tilde\xi(s), y)\|^2 ds\, {\tilde\nu_r(dy)}\, dr\\
	\le \int_{0}^T \int_{\RR^{d}}  \frac{1}{\eta} \int_{r \vee 0}^{(r+\eta)\wedge T} \| \nabla_y U(\tilde\xi(s), y)\|^2 ds\, {\tilde\nu_r(dy)}\, dr + C(U)\delta^2(2+\|\tilde\xi\|^2_{\infty}).
 \end{multline*}
 Next, from \eqref{eq:424n} and \eqref{eq:425n},
 \begin{multline*}
 \int_{0}^T \int_{\RR^{d}}  \frac{1}{\eta} \int_{r \vee 0}^{(r+\eta)\wedge T} (\| \nabla_y U(\tilde\xi(s), y)\|^2  - \| \nabla_y U(\tilde\xi(r), y)\|^2) ds\, {\tilde\nu_r(dy)}\, dr\\	
 \le 2C(U) \delta \int_0^T \int_{\RR^{d}} (1+ \|\tilde\xi\|_{\infty} +\|y\|^2){\tilde\nu_r(dy)}\, dr
 \le 2C(U)\delta (T(1+\|\tilde\xi\|_{\infty} )+A_1).
 \end{multline*}
 Combining the last three displays
 \begin{equation}
 	\frac{1}{2}\int_0^T \int_{\RR^{d}} \| \nabla_y U(\xi^*_{\eta}(s), y)\|^2 \mu_{\eta,s}^*(dy) ds
	\le \frac{1}{2}\int_0^T \int_{\RR^{d}} \| \nabla_y U(\tilde \xi(s), y)\|^2 {\tilde\nu_s(dy)}\, ds + \kappa_4 \delta,
 \end{equation}
 where
 $$\kappa_4 \doteq \kappa_3 + C(U)(2+\|\tilde\xi\|^2_{\infty}) + 2C(U)(T(1+\|\tilde\xi\|_{\infty} )+A_1).$$
 Also, an interchange of order of integration  and application of Jensen's inequality  shows that
 \begin{equation}\label{eq:942n}
 \frac{1}{2}  \int_0^T  \|v_{\eta}^*(s)\|^2 ds \le \frac{1}{2}  \int_0^T \|\tilde v(s)\|^2 ds.\end{equation}
 From the last two displays and \eqref{eq:lowbd2} we now see that
 \begin{equation}
\frac{1}{2}  \int_0^T \left[
		\left (\int_{\RR^{d}} \| \nabla_y U(\xi^*_{\eta}(s), y)\|^2 \mu_{\eta,s}^*(dy)\right)
		+  \|v_{\eta}^*(s)\|^2\right] ds \le I_2(\tilde \xi, \tilde \nu) + \kappa_4\delta + \delta_0/4.	
 \end{equation}
 Define $\mu_{\eta}^* \in \clm_1$ as $\mu_{\eta}^*(dy\, ds) = \mu_{\eta,s}^*(dy)\, ds$.
We now estimate the distance between  $\mu_{\eta}^* $ and $\tilde \nu$. Recall that the space $\clm_1$ is equipped with the bounded Lipschitz distance defined in \eqref{eq:510n}. Fix $f\in BL_1(\RR^d\times [0,T])$.
Then, by  an interchange of order of integration,
\begin{align*}
	\int_{\RR^d\times [0,T]} f(y,s) \mu_{\eta}^*(dy, ds) &= \int_0^T \int_{\RR^{d}} f(y,s) \mu_{\eta,s}^*(dy)\, ds
	= \int_{-\eta}^T \int_{\RR^{d}} \frac{1}{\eta} \int_{r\vee 0}^{(\eta+r)\wedge T} f(y,s) ds\, {\tilde\nu_r(dy)}\, dr.
\end{align*}
Also, since $f$ is bounded by $1$,
$$\left|\int_{-\eta}^T \int_{\RR^{d}} \frac{1}{\eta} [(\eta+r)\wedge T - r\vee 0] f(y,r) {\tilde\nu_r(dy)}\, dr
- \int_0^T \int_{\RR^{d}} f(y,s){\tilde\nu_r(dy)}\, ds \right| \le 4\eta.$$
Using the Lipschitz property of $f$
$$\left|\int_{-\eta}^T \int_{\RR^{d}} \frac{1}{\eta} \int_{r\vee 0}^{(\eta+r)\wedge T} f(y,s) ds\, {\tilde\nu_r(dy)}\, dr
- \int_{-\eta}^T \int_{\RR^{d}} \frac{1}{\eta} [(\eta+r)\wedge T - r\vee 0] f(y,r) {\tilde\nu_r(dy)}\, dr \right| \le \eta(T+1).$$
The last three estimates show that
\begin{equation}
	\dbl (\mu_{\eta}^*,\tilde \nu) \le \eta(T+5).
\end{equation}
Combining the above with \eqref{eq:536n} we now have that
\begin{equation}
\|F(\tilde\xi, \tilde \nu) - F(\xi^*_{\eta}, \mu_{\eta}^*)	\|_{\infty} \le \kappa_5 \delta,
\end{equation}
where $\kappa_5 \doteq \|F\|_{Lip} (\kappa_2 + T+5)$ and $\|F\|_{Lip}$ is the Lipschitz constant of $F$.

Now choose $\delta>0$ such that $\max\{\kappa_5 \delta, \kappa_4\delta \}\le \delta_0/20$. With this choice of $\delta$ (and the corresponding $\eta$), we have
\begin{equation}\label{eq:1016n}
	F(\xi^*_{\eta}, \mu_{\eta}^*) + \frac{1}{2}  \int_0^T \left[
			\left (\int_{\RR^{d}} \| \nabla_y U(\xi^*_{\eta}(s), y)\|^2 \mu_{\eta,s}^*(dy)\right)
			+  \|v_{\eta}^*(s)\|^2\right] ds \le F(\tilde\xi, \tilde \nu) + I_2(\tilde \xi, \tilde \nu) + \delta_0/10 + \delta_0/4.	
\end{equation}
Henceforth we fix such an $\eta$ and denote the corresponding $(\xi^*_{\eta}, \mu_{\eta,s}^*, \mu_{\eta}^*, v_{\eta}^*)$
as simply $(\xi^*, \mu_s^*, \mu^*, v^*)$. 
\subsubsection{Approximating with piecewise constant $\nu, v$}
Fix $\delta \in (0,1)$. Once again the choice of $\delta$ will be identified at the end of the section. By construction, the $(v^*, \xi^*)$ obtained at the end of previous section are continuous.
Choose $0<\gamma \le \delta^2$ such that
\begin{equation}\label{eq:815n}
	\|\xi^*(s)-\xi^*(s')\| + \|v^*(s)-v^*(s')\| \le \delta \mbox{ whenever } |s-s'| \le \gamma,\; s, s' \in [0,T].\end{equation}
Let
$K \doteq \lfloor T/\gamma\rfloor$ and define $t_i \doteq i\gamma$ for $i=0, 1, \ldots K$ and $t_{K+1}=T$.
Define $\tilde\xi^*_{\gamma}(0) \doteq \xi^*(0)$ and
$$\tilde\xi^*_{\gamma}(s) \doteq \xi^*(t_i),\; s \in (t_i, t_{i+1}], \; i=0, 1, \ldots K.$$
Using the estimates in \eqref{eq:424n}  and \eqref{eq:902n}, we have that 
\begin{align}
	\left|\frac{1}{2}  \int_0^T 
			\left (\int_{\RR^{d}} \| \nabla_y U(\tilde\xi^*_{\gamma}(s), y)\|^2 \mu^*_{s}(dy)\right) ds
			- \frac{1}{2}  \int_0^T 
			\left (\int_{\RR^{d}} \| \nabla_y U(\xi^*(s), y)\|^2 \mu^*_{s}(dy)\right) ds\right|
			 &\le \kappa_6\delta,
\end{align}
where
$\kappa_6 \doteq C(U) (T (1+ \|\xi^*\|_{\infty}) + A_1 + 1)$.
Next note that
\begin{multline*}
\int_0^T \left (\int_{\RR^{d}} \| \nabla_y U(\tilde\xi^*_{\gamma}(s), y)\|^2 \mu^*_{s}(dy)\right) ds	
= \sum_{i=0}^K \int_{t_i}^{t_{i+1}} \left (\int_{\RR^{d}} \| \nabla_y U(\tilde\xi^*_{\gamma}(s), y)\|^2 \mu^*_{s}(dy)\right) ds\\
= \sum_{i=0}^K  \int_{\RR^{d}}\| \nabla_y U(\tilde\xi^*_{\gamma}(t_i), y)\|^2 \left (\frac{1}{t_{i+1}-t_i}\int_{t_i}^{t_{i+1}}  \mu^*_{s}(dy) ds \right)(t_{i+1}-t_i).
\end{multline*}
Let
$$\mu^*_{s,\gamma}(dy) \doteq \frac{1}{t_{i+1}-t_i}\int_{t_i}^{t_{i+1}}  \mu^*_{s}(dy) ds,\; s \in (t_i, t_{i+1}], \; i=0, 1, \ldots K$$
and $\mu^*_{0,\gamma}(dy) \doteq \mu_{t_0,\gamma}(dy)$.
Then
\begin{align*}
\int_0^T \left (\int_{\RR^{d}} \| \nabla_y U(\tilde\xi^*_{\gamma}(s), y)\|^2 \mu^*_{s}(dy)\right) ds
&= \int_0^T \left (\int_{\RR^{d}} \| \nabla_y U(\tilde\xi^*_{\gamma}(s), y)\|^2 \mu^*_{s,\gamma}(dy)\right) ds.
\end{align*}
Now define
$$v^*_{\gamma}(s) \doteq v^*(t_i),\; s \in (t_i, t_{i+1}], \; i=0, 1, \ldots K$$
and $v^*_{\gamma}(0) \doteq v^*(t_{0})$.
Then, for $t\in [0,T]$,
\begin{align*}
	\xi^*(t) &= x_0 + \int_0^t \int_{\RR^d} b(\xi^*(s), y) \mu^*_{s}(dy) ds + \int_0^t \alpha(\xi^*(s))v^*(s) ds\\
	&= x_0 + \int_0^t \int_{\RR^d} b(\tilde\xi^*_{\gamma}(s), y) \mu^*_{s}(dy) ds + \int_0^t \alpha(\tilde \xi_{\gamma}^*(s)) v^*_{\gamma}(s) ds + \clr_1(t),
\end{align*}
where
$\|\clr_1\|_{\infty} \le (T(L_b+1) + \|\alpha\|_{\infty} T + L_{\alpha}\|v^*\|_2 T^{1/2})\delta.$
Note that 
$$\int_0^t \int_{\RR^d} b(\tilde\xi^*_{\gamma}(s), y) \mu^*_{s}(dy) ds = \int_0^t \int_{\RR^d} b(\tilde\xi^*_{\gamma}(s), y) \mu^*_{s,\gamma}(dy) ds
+\clr_2(t),
$$
where, for some $C(b) \in (0, \infty)$, depending only on $b$,
$\|\clr_2\|_{\infty} \le 2\gamma C(b)(\|\xi^*\|_{\infty}+1).$
Thus,  for $t\in [0,T]$,
\begin{equation}\label{eq:554n}
	\tilde\xi^*_{\gamma}(t) = x_0 + \int_0^t \int_{\RR^d} b(\tilde\xi^*_{\gamma}(s), y) \mu^*_{s,\gamma}(dy) ds + \int_0^t \alpha(\tilde \xi_{\gamma}^*(s)) v^*_{\gamma}(s) ds + \clr(t),\end{equation}
where $\|\clr\|_{\infty} \le \delta(T(L_b+1)+ \|\alpha\|_{\infty} T + L_{\alpha}\|v^*\|_2 T^{1/2}+ 2 C(b)(\|\xi^*\|_{\infty}+1)+1)\doteq \kappa_7\delta$.
Let $\xi^*_{\gamma}$ be the unique solution of  the equation
\begin{equation}\label{eq:555n}\xi^*_{\gamma}(t) = x_0 + \int_0^t \int_{\RR^d} b(\xi^*_{\gamma}(s), y) \mu^*_{s,\gamma}(dy) ds + \int_0^t \alpha(\xi^*_{\gamma}(s)) v^*_{\gamma}(s) ds.\end{equation}
Then, by \eqref{eq:554n}, \eqref{eq:555n}, Grönwall's lemma and the Lipschitz property of $b$,
\begin{equation}\label{eq:1002n}
	\|\tilde \xi^*_{\gamma} -  \xi^*_{\gamma}\|_{\infty,T} \le \kappa_8\delta,\end{equation}
where $\kappa_8 = \kappa_7\exp\{TL_b+ (T+1)(\|v^*\|_2+1) L_{\alpha}\}$. Next, using the estimates in \eqref{eq:424n}  and \eqref{eq:902n}, we have that 
\begin{align}\label{eq:1003n}
\left |\frac{1}{2}\int_0^T\int_{\RR^{d}} \| \nabla_y U(\tilde\xi^*_{\gamma}(s), y)\|^2 \mu^*_{s,\gamma}(dy) \, ds	
- \frac{1}{2}\int_0^T \int_{\RR^{d}} \| \nabla_y U(\xi^*_{\gamma}(s), y)\|^2 \mu^*_{s,\gamma}(dy)\, ds\right| \le \kappa_9 \delta,
\end{align}
where $\kappa_9 = C(U)\kappa_8[(1+ \|\xi^*\|_{\infty} + \kappa_8) T + (A_1+1)]$.
Also, using \eqref{eq:815n}
\begin{align}\label{eq:1004n}\left | \frac{1}{2}\int_0^T \|v^*(s)\|^2 ds - \frac{1}{2}\int_0^T \|v^*_{\gamma}(s)\|^2 ds\right | \le \kappa_{10} \delta\end{align}
where $\kappa_{10} = (T +  \sqrt{T}\|\tilde v\|_2 )$.
Define $\mu^*_{\gamma} \in \clm_1$ as $\mu^*_{\gamma} (dy\, ds) = \mu^*_{s,\gamma} (dy)\, ds$. We now estimate the distance between $\mu^*_{\gamma} $ and $\mu^*$. Consider $f\in BL_1(\RR^d\times [0,T])$.
Then, for $k = 0, 1 , \ldots K$,
$$\left|\int_{t_k}^{t_{k+1}} \int_{\RR^d} f(y,s) \mu^*_{s, \gamma}(dy) ds  - \int_{t_k}^{t_{k+1}}\int_{\RR^d} f(y,s)\mu^*_{s}(dy) ds\right| \le \gamma (t_{k+1}- t_k).$$
From this it follows that
\begin{align}\label{eq:1005n}\dbl(\mu^*_{\gamma}, \mu^*) \le \gamma T \le \delta T.\end{align}
Now take $\delta$ to be small enough so that
$$\delta(\|F\|_{Lip}(\kappa_8+1+T)+\kappa_{10} +\kappa_9+\kappa_6) \le \delta_0/10.$$
Then, using the Lipschitz property of $F$, together with \eqref{eq:815n}, \eqref{eq:1002n},  \eqref{eq:1003n}, \eqref{eq:1004n} and \eqref{eq:1005n} it follows that 
\begin{equation} 
\begin{split} \label{eq:313n}
&\frac{1}{2}  \int_0^T \left[
			\left (\int_{\RR^{d}} \| \nabla_y U(\xi^*_{\gamma}(s), y)\|^2 \mu^*_{\gamma,s}(dy)\right)+ \|v^*_{\gamma}(s)\|^2\right] ds
			+ F(\xi^*_{\gamma}, \mu^*_{\gamma})\nonumber \\
			&\le F(\xi^*, \mu^*) + \frac{1}{2}  \int_0^T \left[
			\left (\int_{\RR^{d}} \| \nabla_y U(\xi^*(s), y)\|^2 \mu_{s}^*(dy)\right)
			+  \|v^*(s)\|^2\right] ds + \delta_0/10\nonumber\\
			 &\le F(\tilde\xi, \tilde \nu) + I_2(\tilde \xi, \tilde \nu) + \delta_0/5 + \delta_0/4, 
\end{split}
\end{equation}
where the last line follows from \eqref{eq:1016n}.
Henceforth we will fix such a $\delta$ and suppress $\gamma$ in the notation for $\xi^*_{\gamma}, \mu^*_{\gamma,s}, \mu^*_{\gamma}, v^*_{\gamma}$. 
\subsubsection{Approximating using discrete measures}
\label{sec:discappr}
Finally, we will now approximate the piecewise constant trajectory of measures $\mu_s$ from the last section by a similar trajectory where the measures are discrete. Fix $\delta \in (0,1)$. An appropriate choice of $\delta$ will be identified at the end of the section.
Let
$0=t_0 <t_1 < \cdots t_{K+1} = T$ be the partition over which $\mu^*_s$ and $v^*(s)$ are piecewise constant.
Let $\hat\mu_i = \mu^*_{t_i}$, $i = 0, \ldots K$.
Note that, from \eqref{eq;307n}, for each $i$,
\begin{equation}\label{eq:cieq}
	\int_{t_i}^{t_{i+1}}\int_{\RR^{d}} \| \nabla_y U(\xi^*(s), y)\|^2 \hat \mu_{i}(dy)  ds\doteq C_i<\infty.\end{equation}

Recall the space $\mathcal{P}_{dis}$ defined in Subsection \ref{subsec:single-lower}. Then, as in \eqref{single-discrete-approximation}, for each $i$, there is a  $\hat \mu_{i,d} \in \mathcal{P}_{dis}$ such that, 
\begin{equation}\label{eq:nabuerra}
	\int_{t_i}^{t_{i+1}}\left|\int_{\RR^d} b(\xi^*(s), y) \hat\mu_{i,d}(dy) - \int_{\RR^d} b(\xi^*(s), y) \hat\mu_{i}(dy)\right| ds \le \delta,\end{equation}
\begin{equation}\label{eq:nabuerr}
	\left\| \int_{t_i}^{t_{i+1}}\int_{\RR^{d}} \| \nabla_y U(\xi^*(s), y)\|^2 \hat \mu_{i,d}(dy)\, ds - 
\int_{t_i}^{t_{i+1}}\int_{\RR^{d}} \| \nabla_y U(\xi^*(s), y)\|^2 \hat \mu_{i}(dy)\, ds\right\| \le \delta,
\end{equation}
and
\begin{equation}\label{eq:nabuer3}\dbl(\hat\mu_{i,d}, \hat \mu_{i})\le \delta.\end{equation}
Define
$$\mu_{t,d} \doteq \hat \mu_{i,d}, \; t \in (t_i, t_{i+1}], i =0, 1, \ldots K,$$
and $\mu_{0,d} \doteq \mu_{t_0,d}$.
Then, from \eqref{eq:nabuerra},
for $t\in [0,T]$,
\begin{align*}
	\xi^*(t) &= x_0 + \int_0^t \int_{\RR^d} b(\xi^*(s), y) \mu^*_{s}(dy) ds + \int_0^t \alpha(\xi^*(s)) v^*(s) ds\\
	&= x_0 + \int_0^t \int_{\RR^d} b(\xi^*(s), y) \mu_{s,d}(dy) ds + \int_0^t \alpha(\xi^*(s)) v^*(s) ds + \clr_1(t)
\end{align*}
where
$\|\clr_1\|_{\infty} \le K\delta$.
Let $\xi^*_d$ be the unique solution of
$$
\xi^*_d(t) = x_0 + \int_0^t \int_{\RR^d} b(\xi^*_d(s), y) \mu_{s,d}(dy) ds + \int_0^t \alpha(\xi_d^*(s)) v^*(s) ds.
$$
Then, from the Lipschitz property of $b$, 
\begin{equation}\label{eq:303n}
	\|\xi^*_d - \xi^*\|_{\infty} \le K\delta \exp\{L_bT +L_{\alpha}T^{1/2} \|v^*\|_2\} \doteq \kappa_{11}\delta .\end{equation}
Also, from \eqref{eq:cieq}, \eqref{eq:nabuerr}, and the lower bound in Assumption \ref{Assumptions:system}(2c)
$$\int_0^T \int_{\RR^{d}}  \|y\|^2 \mu_{s,d}(dy)\, ds \le (L^1_{low})^{-1}(\sum_{i=0}^K(C_i+1) + T L^2_{low}) \doteq \kappa_{12}.$$
From the last two estimates and \eqref{eq:424n}
\begin{align*}
\left |\frac{1}{2}\int_0^T \left (\int_{\RR^{d}} \| \nabla_y U(\xi^*_{d}(s), y)\|^2 \mu_{s,d}(dy)\right) ds	
- \frac{1}{2}\int_0^T \left (\int_{\RR^{d}} \| \nabla_y U(\xi^*(s), y)\|^2 \mu_{s,d}(dy)\right) ds\right| \le \kappa_{13} \delta,
\end{align*}
where
$$\kappa_{13} \doteq C(U)\kappa_{11}(T(1+  \|\xi^*\|_{\infty} + \kappa_{11}) + \kappa_{12}).$$
Now let $f \in BL_1(\RR^d\times [0,T])$.
Then, from \eqref{eq:nabuer3}, for $k = 0, 1 , \ldots K$,
$$
\left | \int_{\RR^d} \left(\int_{t_k}^{t_{k+1}} f(s,y) ds\right) \hat \mu_{k,d} - \int_{\RR^d} \left(\int_{t_k}^{t_{k+1}} f(s,y) ds\right) \hat \mu_{k}\right| \le (t_{k+1}-t_k) \delta.$$
This shows that, with $\mu_d(dy\, ds)\doteq \mu_{s,d}(dy)\, ds$,
\begin{equation} \label{eq:304n} \dbl(\mu_d, \mu^*) \le \delta T.\end{equation}
	Now choose $\delta \in (0,1)$ such that $\delta(\|F\|_{Lip}(\kappa_{11} +T) +\kappa_{13}+T) \le \delta_0/10$.
Then,  using the Lipschitz property of $F$, and  \eqref{eq:303n}, \eqref{eq:304n} and \eqref{eq:nabuerr} it follows that 
\begin{align}&\frac{1}{2}  \int_0^T \left[
			\left (\int_{\RR^{d}} \| \nabla_y U(\xi^*_{d}(s), y)\|^2 \mu_{s,d}(dy)\right)+ \|v^*(s)\|^2\right] ds
			+ F(\xi^*_{d}, \mu_{d})\nonumber \\
			&\le F(\xi^*, \mu^*) + \frac{1}{2}  \int_0^T \left[
			\left (\int_{\RR^{d}} \| \nabla_y U(\xi^*(s), y)\|^2 \mu_{s}^*(dy)\right)
			+  \|v^*(s)\|^2\right] ds + \delta_0/10\nonumber\\
			 &\le F(\tilde\xi, \tilde \nu) + I_2(\tilde \xi, \tilde \nu) + 22\delta_0/40,\label{eq:313nn}
\end{align}
where the last line uses \eqref{eq:313n}.\\
%

\noindent {\bf Proof of Lemma \ref{lem:approxdisc}.} Lemma \ref{lem:approxdisc} now follows on combining the above display with \eqref{eq:lowbd1} and setting  $\xi^*\doteq \xi^*_d$, $v^*\doteq v^*$, $\nu^*_i\doteq\mu_{i,d}$, and {$\nu^*=\mu_d $}. \hfill \qed
\subsection{Construction of Controlled process}
\label{sec:conscont}
We will now use the trajectory $\xi^*$, control $v^*$ and measures $\nu_i^*$ given by Lemma \ref{lem:approxdisc} to construct suitable controls for the pre-limit stochastic system in \eqref{eq:system-bar-def}.

Recall that  $0=t_0 <t_1 < \cdots t_{K+1} = T$ is the partition over which $\hat\nu_s$ and $v^*(s)$ are piecewise constant. Also recall  the facts that 
 $\nu^*_i = \nu_{t_i}$, $i = 0, \ldots K$ and that $\nu_i^* \in \clp(\RR^d)$ has finite support for every $i$.
Let this discrete measure be represented as
$$\nu^*_i = \sum_{l=1}^{m(i)} p_{i,l} \delta_{y_{i,l}},$$
where $p_{i,l}>0$ for every $i,l$ and $y_{i,l}\in \RR^d$.

Let $\{\Delta_{\veps}\}_{\veps>0}$ be a collection of positive reals such that
$\Delta_{\veps}\to 0$ and $\Delta_{\veps}^2/\veps \to \infty$, as $\veps \to 0$. Without loss of generality,
$1> \Delta_{\eps} > \sqrt{\eps} >\eps$.
Let $N^i_{\veps} = \lfloor (t_{i+1}-t_i)/\Delta_{\veps}\rfloor$ and define
a partition of $[t_i, t_{i+1}]$, $i = 0, 1, \ldots K$ as
$$t_i \doteq s_{i,0} \le s_{i,1} \le \cdots \le s_{i,(N^i_{\veps}+1)} \doteq t_{i+1}$$
where
$$s_{i,j} \doteq \begin{cases}
	t_i + j \Delta_{\veps}, & j = 0, 1, \ldots , N^{i}_{\veps}\\
	t_{i+1}, & j = N^i_{\veps}+1
\end{cases},\;i = 0, 1, \ldots K. 
$$
Also define, for $i= 0, 1, \ldots K$,
$$
\Delta^{i,j}_{\veps} \doteq 
\begin{cases}
	\Delta_{\veps}, & j = 0, 1, \ldots , N^{i}_{\veps}-1\\
	t_{i+1} - N^{i}_{\veps}\Delta_{\veps}, & j = N^i_{\veps}
\end{cases}.
$$
For $y, y' \in \RR^d$, let
\begin{equation}\label{eq:859n}
	\uu^*(y,y',t) \doteq y'-y, \; 0 \le t \le 1,\end{equation}
and define, for $i= 0, 1, \ldots K$, $j = 0, 1, \ldots N^i_{\veps}-1$,
$$\uu^{\eps}(y,y',t) \doteq \frac{1}{\Delta_{\veps}}\uu^*(y,y',\frac{t}{\veps \Delta_{\veps}}), \; 0 \le t \le \veps \Delta_{\veps}.$$
Define $P_{il}$ for $i= 0, 1, \ldots K$, $l = 0, 1, \ldots, m(i)$, as
$$P_{i,0}\doteq 0, \; P_{i,l} \doteq P_{i,(l-1)} + p_{i,l}, \; l = 1, \ldots, m(i).$$
Denote, associated with the interval $[t_i, t_{i+1}]$ and its subinterval $[s_{ij}, s_{i,j+1}]$, for $l = 0, 1, \ldots, m(i)$,
$$\sigma_{i,j,l} \doteq s_{i,j} + P_{i,l}\Delta_{\veps}^{i,j}.$$
Occasionally we will write $\sigma_{i,j,l},s_{i,j}, P_{i,l}, p_{i,l}, y_{i,l}$ as $\sigma_{ijl},s_{ij}, P_{il}, p_{il}, 
y_{il}$ respectively, for brevity. Without loss of generality we assume that $\eps$ is sufficiently small so that  for some  $\underline{p} \in (0, \infty)$, $p_{il}- \eps \ge \underline{p} $ for all $i,l$.

\begin{figure}
    \centering
    \includegraphics[width=.95\textwidth]{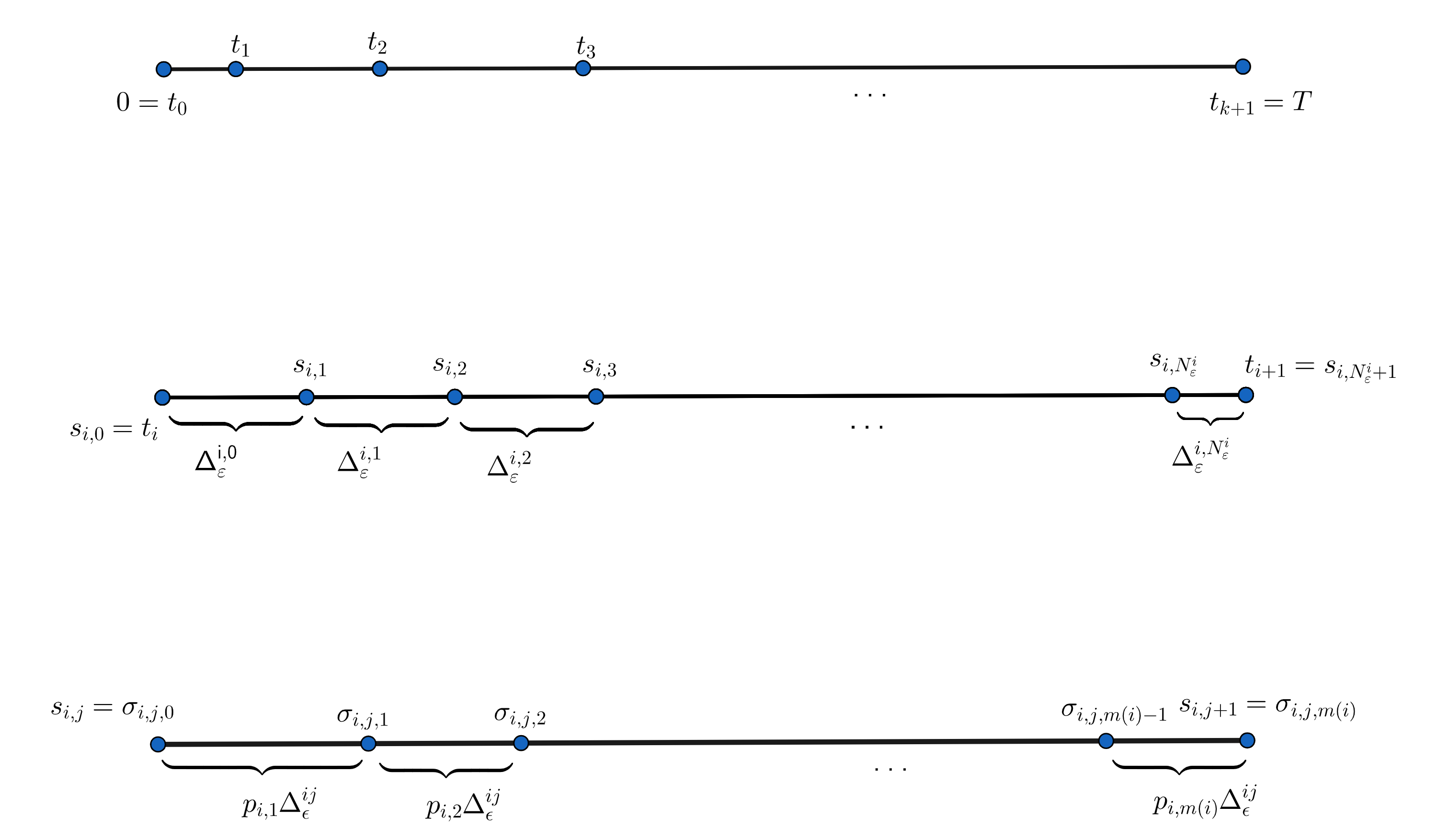}
    \caption{Construction of the successive partitions: first into intervals of the type $[t_i,t_{i+1}]$; then into subintervals of the type $[s_{i,j},s_{i,j+1}]$; finally into $[\sigma_{i,j,l},\sigma_{i,j,l+1}]$.}
    \label{fig:interval_construction}
\end{figure}

The state equations for the controlled processes are
\begin{equation}\label{eq:s2e6FS}
\begin{aligned}
	d\bar X^{\veps}(t) &= b(\bar X^{\veps}(t), \bar Y^{\veps}(t)) dt + s(\veps) \veps^{1/2} \alpha(\bar X^{\veps}(t)) dW(t) + \alpha(\bar X^{\veps}(t)) v^*(t) dt, \; \bar X^{\veps}(0) = x_0\\
	d\bar Y^{\veps}(t) &= - \frac{1}{\veps} [\nabla_y U(\bar X^{\veps}(t), \bar Y^{\veps}(t))  -  u^{\veps}(t)] dt + 
	\frac{s(\veps)}{\veps^{1/2}} d  B(t),  \; \bar Y^{\veps}(0) = y_0
\end{aligned}
\end{equation}
where $v^*$ is as in Lemma \ref{lem:approxdisc} and $u^{\veps}$ is given in state feedback form as follows.
For $i= 0, 1, \ldots K$, $j = 0, 1, \ldots N^i_{\veps}-1$, $l = 0, 1, \ldots, m(i)-1$,
\begin{equation}\label{eq:defncont}
	u^{\veps}(t) \doteq
	\begin{cases}
		\nabla_y U(\bar X^{\veps}(t), \bar Y^{\veps}(t)) +  \uu^{\eps}(\bar Y^{\veps}(\sigma_{ijl})), y_{i,l+1}, t - \sigma_{ijl}),\\
		\quad \quad \quad\quad\quad\quad \mbox{ if } t \in (\sigma_{ijl}, \sigma_{ijl}+\veps \Delta_{\veps}]\\
		\nabla_y U(\bar X^{\veps}(t), \bar Y^{\veps}(t)) - 
		\nabla_y U(\bar X^{\veps}(\sigma_{ijl}), \bar Y^{\veps}(t))
		+ \nabla_y U(\bar X^{\veps}(\sigma_{ijl}), y_{i,l+1}),\\
		\quad \quad \quad\quad\quad\quad \mbox{ if } t \in (\sigma_{ijl}+\veps \Delta_{\veps},
		\sigma_{ij(l+1)}]\\
		\nabla_y U(\bar X^{\veps}(t), \bar Y^{\veps}(t)) + \bar Y^{\veps}(t),
		\quad \quad \quad\quad\quad\quad \mbox{ if } t \in (s_{i, N^i_{\veps}}, t_{i+1}].
		\end{cases}
\end{equation}
From the Lipschitz property of $\nabla_yU$ and $\alpha$ we see that \eqref{eq:s2e6FS} has a unique solution and the feedback control $u^{\eps}$ in \eqref{eq:defncont} is well defined.

\begin{table}[h!]
\centering
 \begin{tabular}{||c|c|c||} 
 \hline
Time interval  & State of $\Bar{Y}^\veps$ &  Control Process $u^\veps(s)$  \\ [0.5ex] 
 \hline\hline
 $[\sigma_{ij0},\sigma_{ij0}+\veps \Delta_\veps]$ & $\bar Y^{\eps}(s_{ij}) \rightarrow y_{i,1}$ & \makecell{ $\nabla_y U(\bar{X}^\veps (s) \textit{,} \bar{Y}^\veps (s)) + \uu^{\eps}(\bar{Y}^\veps(\sigma_{ij0}),y_{i,1},t-\sigma_{ij0})$ } \\ 
 \hline
 $(\sigma_{ij0}+\veps \Delta_\veps,\sigma_{ij1}]$ & $y_{i,1}$  & \makecell{$ \nabla_y U (\Bar{X}^\veps(s),\Bar{Y}^\veps(s)) - \nabla_y U (\Bar{X}^\veps(\sigma_{ij0}),\Bar{Y}^\veps(s))$   \\ 
 $+ \nabla_y U (\Bar{X}^\veps(\sigma_{ij0}) ,y_{i,1}) $}  \\
 \hline
 $(\sigma_{ij1},\sigma_{ij1} + \veps \Delta_\veps]$ &  $y_{i,1} \rightarrow y_{i,2}$ & \makecell{ $\nabla_y U(\bar{X}^\veps (s) \textit{,} \bar{Y}^\veps (s)) + \uu^{\eps}(\bar{Y}^\veps(\sigma_{ij1}),y_{i,2},t-\sigma_{ij1})$ } \\
 \hline
 \dots & \dots & \dots \\ [1ex]  \hline
 $(\sigma_{i,j,m(i)-1} + \veps \Delta_\veps,s_{i,j+1}]$ &
 $y_{i,m(i)}$ &  \makecell{$ \nabla_y U (\Bar{X}^\veps(s),\Bar{Y}^\veps(s)) - \nabla_y U (\Bar{X}^\veps(\sigma_{i,j,m(i)-1}),\Bar{Y}^\veps(s)) $  \\$
 + \nabla_y U (\Bar{X}^\veps(\sigma_{i,j,m(i)-1}) ,y_{i,m(i)}) $} \\
  \hline
\end{tabular}
\label{table:process-travel-system}
\caption{Construction of the control process $u^\veps$ on the interval $[s_{ij},s_{i,j+1}]$, $j = 0, 1, \ldots, N^i_{\eps}-1$, $i=0, 1, \ldots, K$. Second column denotes the approximate states/transitions for $\bar{Y}^\veps$ under the selection of the controls, $u^\veps$. }
\end{table}

We will show that this controlled process $\bar X^{\veps}$ and the corresponding $\bar \Lambda^{\eps}$ defined by \eqref{eq:laeps} converge suitably to the near optimal  $(\xi^*, \nu^*)$
in Lemma \ref{lem:approxdisc} and the associated costs converge appropriately as well.
In preparation for this result we first prove a stabilization lemma analogous to Lemma \ref{theorem:auxillary} used in the proof of Theorem \ref{theorem:ldp}. We suppress some details in the proof that are similar to that in Lemma \ref{theorem:auxillary}.
\subsubsection{A stabilization lemma.}
\label{sec:stablem}
Recall, for $x \in \RR^m,z \in \RR^d$ the function $\clv_{x,z}: \RR^d \to \RR^d$ as introduced in Assumption \ref{Assumptions:system}(2d).
For $R <\infty$, let
\begin{equation}\label{eq:ar}
	A_R \doteq \{ (x,y) \in \RR^m \times \RR^d: \|x\|^2 + \|y\|^2 \le R\}.\end{equation}	
\begin{lemma}\label{lem:llbcgcene}
	 There exists $\kappa_1 \in (1, \infty)$ such that for every  $z\in \RR^d$, and  a collection $\{(x^{\veps}, y^{\veps})\}_{\veps \in (0,1)}$ in  $\RR^{m+d}$, with
	  $\tilde Y(s)\equiv Y^{\veps}(s,z, x^{\veps}, y^{\veps})$ given as the solution of
	\begin{equation}\label{eq:yepsx}
			d\tilde Y(t) = -\frac{1}{\veps} \clv_{x^{\veps},z}(\tilde Y(t)) dt
			+ \frac{s(\veps)}{\sqrt{\veps}} dB(t), \; \tilde Y(0) = y^{\veps},
	\end{equation}
	we have, for all $\veps \in (0,1)$,
	\begin{enumerate}[(i)]
	\item $\,$
	$$\sup_{0\le t \le \Delta_{\veps}} E\|Y^{\veps}(t,z, x^{\veps},y^{\veps} )\|^2 < 
	\kappa_1(1+ \|x^{\veps}\|^2+ \|y^{\veps}\|^2+\|z\|^2).$$
	\item 
	For every  $a \in (0,1)$,  $z\in \RR^d$, and $R <\infty$, as $\veps \to 0$,
	\begin{equation}\label{eq:contra}
		\sup_{t \in [a, 1]} \sup_{(\bar x, \bar y) \in A_R} E\left [ \dbl\left( \frac{1}{t\Delta_{\veps}} \int_0^{t\Delta_{\veps}} \delta_{Y^{\veps}(s,z, \bar x, \bar y )} ds, \delta_0\right)\right] \to 0.
\end{equation}
\end{enumerate}
\end{lemma}
\begin{proof}
	Let
	$$V_{x^{\veps},z}(y) = U(x^{\veps}, y+z) - y \cdot \nabla_y U(x^{\veps}, z), \; y \in \RR^d.$$
		Applying It\^{o}'s formula
\begin{equation}
\begin{split}
    \label{eq:256}
			V_{x^{\veps},z}(\tilde Y(t)) &= V_{x^{\veps},z}(y^{\veps}) -\frac{1}{\veps} \int_0^t \nabla V_{x^{\veps},z}(\tilde Y(s)) \cdot\clv_{x^{\veps},z}(\tilde Y(s))  ds \\
			&\quad + \frac{s(\veps)}{\sqrt{\veps}} \int_0^t  \nabla V_{x^{\veps},z}(\tilde Y(s))\cdot dB(s) + \frac{s^2(\veps)}{2\veps} \int_0^t  \Delta V_{x^{\veps},z}(\tilde Y(s)) ds.
\end{split}
\end{equation}
	Let $\tau_m = \inf\{t: \|\tilde Y(t)\| \ge m\}$. Then, 
\begin{equation}
\begin{split}
		 \label{eq:706}
     E\int_0^{t\wedge \tau_m} \nabla V_{x^{\veps},z}(\tilde Y(s))\cdot \clv_{x^{\veps},z}(\tilde Y(s)) ds
		 &= \veps (V_{x^{\veps},z}(y^{\veps})-EV_{x^{\veps},z}(\tilde Y(t\wedge \tau_m))) \\
   &+  \frac{s^2(\veps)}{2} E \int_0^{t\wedge \tau_m} \Delta V_{x^{\veps},z}(\tilde Y(s)) ds.
\end{split}
\end{equation}
Using the assumption on $U$ (Assumption \ref{Assumptions:system}(2b,2c)), we can find $c_1, c_2, c_3 \in (0, \infty)$ such that for all $x\in \RR^m$ and $y,z\in \RR^d$
	\begin{equation}\label{eq:1015}
		c_2\|y\|^2 - c_3 (1+\|x\|^2+\|z\|^2) \le V_{x,z}(y) \le c_1(1+ \|x\|^2+ \|y\|^2+ \|z\|^2)
	\end{equation}
	and
	\begin{equation*}
		c_2\|y\|^2 - c_3 (1+\|x\|^2+\|z\|^2) \le \left\|\nabla V_{x,z}(y)\right\|^2.
	\end{equation*}
	Using these observations, together with $\nabla V_{x^{\veps},z}\cdot \clv_{x^{\veps},z} = \|\nabla V_{x^{\veps},z}\|^2$, in \eqref{eq:706}, for $0\le t \le \Delta_{\veps}$
	\begin{align*}
	&E\int_0^{t\wedge \tau_m}	\left(c_2\|\tilde Y(s)\|^2 - c_3 (1+\|x^{\veps}\|^2+\|z\|^2)\right) ds\\
		&\le E\int_0^{t\wedge \tau_m} \nabla V_{x^{\veps},z}(\tilde Y(s))\cdot \clv_{x^{\veps},z}(\tilde Y(s)) ds\\
		&
		 \le \veps c_1(1+ \|x^{\veps}\|^2+ \|y^{\veps}\|^2+ \|z\|^2)
		 - \veps\left(c_2E\|\tilde Y(t\wedge \tau_m))\|^2 - c_3 (1+\|x^{\veps}\|^2+\|z\|^2)\right)
		  +  c_4s^2(\veps) \Delta_{\veps}.
	\end{align*}
	where $c_4 \doteq \sup_{x,y}\frac{1}{2}\|\clh_yU(x,y)\|$. 
	Thus, sending $m\to \infty$, for some $c_5 \in (0, \infty)$,
	\begin{equation}\label{eq:bdonint}
		E\int_0^{\Delta_{\veps}}	\|\tilde Y(s)\|^2  ds
	\le c_5\left( (1+ \|x^{\veps}\|^2+ \|z\|^2) \Delta_{\veps} +  \|y^{\veps}\|^2\veps\right).\end{equation}
	Also, from \eqref{eq:256} and \eqref{eq:1015} we see that 
	\begin{align*}
		c_2E\|\tilde Y(t)\|^2 - c_3 (1+\|x^{\veps}\|^2+\|z\|^2) &\le EV_{x^{\veps},z}(\tilde Y(t))\le
		c_1(1+ \|x^{\veps}\|^2+ \|y^{\veps}\|^2+ \|z\|^2)\\
		&\quad-\frac{1}{\veps} E\int_0^t \left(c_2\|\tilde Y(s)\|^2 - c_3 (1+\|x^{\veps}\|^2+\|z\|^2)\right) ds  +\frac{c_4s^2(\veps)}{\veps} t .
		\end{align*}
		This, using an argument similar to that below \eqref{eq:546n}, shows that for some $c_6 \in (0,\infty)$, for all $t \le \Delta_{\veps}$,
		$$E\|\tilde Y(t)\|^2 \le c_6(1+ \|x^{\veps}\|^2+ \|y^{\veps}\|^2+ \|z\|^2).$$
		The first part in the lemma now follows on taking $\kappa_1 = c_6$.
		
		In order to prove the second statement we argue via contradiction.
		Fix an $a \in (0,1)$ and an $R<\infty$.
		Suppose that the convergence in \eqref{eq:contra} fails to hold. Then there is a $\gamma>0$, a sequence $\veps_n \to 0$,  $t_n \in [a, 1]$ and $(x_n, y_n ) \in A_R$
		such that for every $n\ge 1$
		\begin{equation}\label{eq:contra2}
			 E\left [ \dbl\left( \frac{1}{t_n\Delta_{\veps_n}} \int_0^{t_n\Delta_{\veps_n}} \delta_{Y^{\veps_n}(s,z, x_{n}, y_{n})} ds, \delta_0\right)\right] > \gamma.
		\end{equation}
		Introduce
		random probability measures $\tilde Q^{n}$ on $\RR^d$  as
		\begin{equation}
			\label{eq:e7}
			\tilde Q^{n}(A ) \doteq \frac{1}{t_n\Delta_{\veps_n}}\int_0^{t_n\Delta_{\veps_n}} \one_A( Y^{\veps_n}(s,z, x_{n}, y_{n}))  ds, \;  A \in \calB(\RR^d).
		\end{equation}
		Using \eqref{eq:bdonint} and our assumption on $\{(x_n,y_n)\}$ we see that $\tilde Q^{n}$ is tight. Suppose that it converges in distribution along a subsequence to $\tilde Q$ along which we also have that $x_n \to x$ for some $x \in \RR^m$.
		Then, using the Lipschitz property of $\nabla_yU$, we have, for all $\eta: \RR^d \to \RR$ that are  in $\clc_b^2$,
		$$\int_{\RR^d} \left |\clv_{x_n,z}(y)  \cdot \nabla \eta(y) - \clv_{x,z}(y)  \cdot \nabla \eta(y)\right|  \tilde Q^n(dy) \to 0.$$
		Now, using  \eqref{eq:bdonint} and the property $\Delta(\eps)^2/\eps \to \infty$, along the lines of the proof of Lemma \ref{theorem:auxillary}  (see equation \eqref{eq:552n} therein), for all $\eta$ as above,
		\begin{equation}\label{eq:echhcrit}
			\int_{\RR^d} \clv_{x,z}(y)  \cdot \nabla \eta(y)  \tilde Q(dy) =0, \mbox{ a.s.}
		\end{equation}
		From Assumption \ref{Assumptions:system}(2d) we now have by similar arguments as as the end of Lemma \ref{theorem:auxillary} that $\tilde Q$ must be $\delta_0$ which contradicts \eqref{eq:contra2}.
		This completes the proof of the second statement in the lemma.
		\end{proof}
\subsubsection{Some preliminary estimates.}
In this subsection we collect some  estimates that will be used to show the convergence of $\bar X^{\veps}$ and the associated costs to appropriate limits.  We begin with the following lemma which gives a key moment bound. 
\begin{lemma}
\label{lem:keymombdne} 
Let $M_1 \doteq \sup_{0\le t \le T}\sup_{\veps \in (0,1)} 	E(\|\bar Y^{\veps}(t)\|^2+ \|\bar X^{\veps}(t)\|^2).$ Then, $M_1 < \infty$.
\end{lemma}
\begin{proof}
	By a straightforward conditioning argument and a recursion,  it suffices to prove the result with $T$ replaced with $t_1$.
	Henceforth, to reduce notation, we will denote, $p_{0,l}=p_l$, $P_{0,l}=P_l$ $y_{0,l}=y_l$, $\Delta_{\veps}^{0,j} =\Delta_{\veps}^{j}$,
	$s_{0,j}= s_j$, $\sigma_{0,j,l}= \sigma_{j,l}$
	for $j = 0, \ldots , N^0_{\veps}$ and $l = 1, \ldots , m(0)$. Also we write $m= m(0)$ and $N^0_{\veps} = N_{\veps}$.
	
	Consider an interval of the form $[\sigma_{j,l}, \sigma_{j, l+1}]$, $j = 0, \ldots , N_{\veps}-1$ and denote 
	as $\bar X^{\veps}(\sigma_{j,l}) = x^{\veps}$. 
	Then, with $c_0\ doteq 12$
	\begin{equation}\label{eq:eq323}
		E\sup_{\sigma_{j,l} \le t \le \sigma_{j,l}+\veps \Delta_{\veps}^j} \|\bar Y^{\veps}(t)\|^2
		\le c_0(E\|\bar Y^{\veps}(\sigma_{j,l})\|^2 +  \|y_{l+1}\|^2 + s^2(\veps)\Delta_{\veps}).
	\end{equation}
	and
	\begin{equation}\label{eq:eq902}
		E \|\bar Y^{\veps}(\sigma_{j,l}+\veps \Delta_{\veps}^j)\|^2 \le \|y_{l+1}\|^2 + s^2(\veps)\Delta_{\veps}.
	\end{equation}
Consider now the interval $[\sigma_{j,l}+\veps \Delta_{\veps}^j, \sigma_{j,l+1}]$.
Define $\bar Z^{\veps}(s) = \bar Y^{\veps}(s+\sigma_{j,l}+\veps \Delta_{\veps}^j) - y_{l+1}$ for
$s \in [0, \sigma_{j,l+1}-(\sigma_{j,l}+\veps \Delta_{\veps}^j)]$.
Then $\{\bar Z^{\veps}(s); s \in [0, \sigma_{j,l+1}-(\sigma_{j,l}+\veps \Delta_{\veps}^j)]\}$, 
solves
	\begin{align*}
		d\bar Z^{\veps}(t) &= -\frac{1}{\veps} \clv_{x^{\veps},y_{l+1}}(\bar Z^{\veps}(t)) dt + 
		\frac{s(\veps)}{\veps^{1/2}} 
	d [B(t+\sigma_{j,l}+\veps \Delta_{\veps}^j) - B(\sigma_{j,l}+\veps \Delta_{\veps}^j)],\\
	 \bar Z^{\veps}(0) &= \bar Y^{\veps}(\sigma_{j,l}+\veps \Delta_{\veps}^j)- y_{l+1}.
	 \end{align*}
	and conditioned on $\clf_{\sigma_{j,l}+ \veps \Delta_{\veps}^j}$,
	has the same law as $\{Y^{\veps}(s, y_{l+1}, x^{\veps}, y^{\veps}); s \in [0,(p_{l+1}-\veps)\Delta_{\veps}]\}$, with $y^{\veps} = \bar Y^{\veps}(\sigma_{j,l}+\veps \Delta_{\veps}^j)- y_{l+1}$, where $Y^{\veps}(s, z, x^{\veps}, y^{\veps})$ is as in Lemma \ref{lem:llbcgcene}.
	Thus, with $\kappa_1$ as in Lemma \ref{lem:llbcgcene}, and $j=0, \ldots , N_{\veps}-1$
\begin{equation}
\begin{split}
    	\sup_{ \sigma_{j,l}+\veps \Delta_{\veps}^j\le t\le \sigma_{j,l+1}}&\sup_{\veps \in (0,1)} 	E\|\bar Y^{\veps}(t)\|^2
		\le
	2 \|y_{l+1}\|^2 +
	2  \sup_{ \sigma_{j,l}+\veps \Delta_{\veps}^j\le t\le \sigma_{j,l+1}}\;\; \sup_{\veps \in (0,1)}	E\|\bar Y^{\veps}(t)- y_{l+1}\|^2\\
	&\le 2\|y_{l+1}\|^2 
	+ 2  \sup_{ 0 \le t \le (p_{l+1}-\veps)\Delta_{\veps}^j}\;\; \sup_{\veps \in (0,1)}	E\|\bar Z^{\veps}(t)\|^2\\
	&\le 2\|y_{l+1}\|^2 
	+ 2  \kappa_1 \sup_{\veps \in (0,1)}E(1+ \|\bar X^{\veps}(\sigma_{j,l})\|^2+ \|\bar Y^{\veps}(\sigma_{j,l}+\veps \Delta_{\veps}^j) - y_{l+1}\|^2+\|y_{l+1}\|^2)\\
	 &\le 2\|y_{l+1}\|^2 + 4  \kappa_1 \sup_{\veps \in (0,1)}E(1+ \|\bar X^{\veps}(\sigma_{j,l})\|^2+ \|\bar Y^{\veps}(\sigma_{j,l}+\veps \Delta_{\veps}^j)\|^2+\|y_{l+1}\|^2)\\
	 &\le 2\|y_{l+1}\|^2 + 8  \kappa_1\sup_{\veps \in (0,1)} E(1+ \|\bar X^{\veps}(\sigma_{j,l})\|^2+ \|y_{l+1}\|^2) \\
     &\le 10\kappa_1\sup_{\veps \in (0,1)}E(1+ \|\bar X^{\veps}(\sigma_{j,l})\|^2+ \|y_{l+1}\|^2),\label{eq:eq324}
\end{split}
\end{equation}
where the third inequality is from Lemma \ref{lem:llbcgcene} and second to last inequality is from \eqref{eq:eq902}. Next, for $t \in [\sigma_{j,l}, \sigma_{j, l+1}]$,
	$$\bar X^{\veps}(t) = \bar X^{\veps}(\sigma_{j,l}) + \int_{\sigma_{j,l}}^t b(\bar X^{\veps}(s), \bar Y^{\veps}(s)) ds
	+ s(\veps)\veps^{1/2}  \int_{\sigma_{j,l}}^t \alpha(\bar X^{\veps}(s))dW(s) +  \int_{\sigma_{j,l}}^t \alpha(\bar X^{\veps}(s)) v^* ds.$$
	Thus
	\begin{align*}
		\|\bar X^{\veps}(t)\|^2 &= \|\bar X^{\veps}(\sigma_{j,l})\|^2 +
		2\int_{\sigma_{j,l}}^t \bar X^{\veps}(s)\cdot b(\bar X^{\veps}(s), \bar Y^{\veps}(s)) ds
		+ 2s(\veps)\veps^{1/2} \int_{\sigma_{j,l}}^t  \bar X^{\veps}(s) \alpha(\bar X^{\veps}(s))dW(s)\\
		&\quad+  2\int_{\sigma_{j,l}}^t \bar X^{\veps}(s) \cdot  \alpha(\bar X^{\veps}(s)) v^* ds + s^2(\veps) \veps 
		\int_{\sigma_{j,l}}^t \tr[(\alpha^T\alpha)(\bar X^{\veps}(s))] ds.
	\end{align*}
	For some $c_1, c_2 \in (1,\infty)$, depending only on  $b, \alpha, v^*$ and $\{y_{il}\}$, for all $t \in [\sigma_{j,l},  \sigma_{j,l+1}]$, $j = 0, \ldots , N_{\veps}-1$, $l\ge 1$,
	\begin{equation} \label{eq:eq519n}
\begin{split}
\sup_{\sigma_{j,l} \le s \le t} E\|\bar X^{\veps}(s)\|^2  &\le E\|\bar X^{\veps}(\sigma_{j,l})\|^2 +
		c_1\int_{\sigma_{j,l}}^t \sup_{\sigma_{j,l} \le u \le s}E(1+ \|\bar X^{\veps}(u)\|^2 + \|\bar Y^{\veps}(u)\|^2) ds\\
		&\le E\|\bar X^{\veps}(\sigma_{j,l})\|^2 + c_2(1+ E\|\bar X^{\veps}(\sigma_{j,l})\|^2+ E\|\bar X^{\veps}(\sigma_{j,l-1})\|^2)\Delta_{\veps}\\
		&\quad + c_2\int_{\sigma_{j,l}}^t \sup_{\sigma_{j,l} \le u \le s}E\|\bar X^{\veps}(u)\|^2 ds.
\end{split}
\end{equation}

	where the last line is from \eqref{eq:eq323} and \eqref{eq:eq324}. The last estimate also holds for $l=0$ and $j>0$ by replacing $X^{\veps}(\sigma_{j,l-1})$ with $X^{\veps}(\sigma_{j-1,m})$.
	By Grönwall lemma, for $l\ge 1$,
	\begin{align}
		\sup_{\sigma_{j,l} \le s \le \sigma_{j,l+1}} E\|\bar X^{\veps}(s)\|^2 \le 
		e^{c_2\Delta_{\veps}}(E\|\bar X^{\veps}(\sigma_{j,l})\|^2 + c_2(1+ E\|\bar X^{\veps}(\sigma_{j,l})\|^2+ E\|\bar X^{\veps}(\sigma_{j,l-1})\|^2)\Delta_{\veps}).\nonumber\\
		\label{eq:eq536}
	\end{align}
	Letting $a_{l} = E\|\bar X^{\veps}(\sigma_{j,l})\|^2$, for $l\ge 1$,
	$$a_{l+1} \le e^{c_2\Delta_{\veps}}a_{l} + c_2(1+a_{l} +a_{l-1})\Delta_{\veps}e^{c_2\Delta_{\veps}},$$
	and so
	\begin{align*}a_{l+1} + c_2a_l \Delta_{\veps} &\le e^{c_2\Delta_{\veps}}a_{l} + c_2(1+a_{l} +a_{l-1})\Delta_{\veps}e^{c_2\Delta_{\veps}}
	+ c_2a_l \Delta_{\veps}\\
	&\le (a_l+c_2a_{l-1}\Delta_{\veps})e^{3c_2\Delta_{\veps}} + c_2 \Delta_{\veps}e^{c_2\Delta_{\veps}}.
	\end{align*}
	Thus, with $j\ge 0$ and $l\ge 1$ or $j\ge 1$ and $l\ge 0$,
	\begin{equation}\label{eq:eq527}
		(E\|\bar X^{\veps}(\sigma_{j,l+1})\|^2 +  c_2 \Delta_{\veps}E\|\bar X^{\veps}(\sigma_{j,l})\|^2) \le 
		e^{3c_2\Delta_{\veps}}(E\|\bar X^{\veps}(\sigma_{j,l})\|^2 +  c_2 \Delta_{\veps}E\|\bar X^{\veps}(\sigma_{j,l-1})\|^2) + c_2\Delta_{\veps} e^{c_2\Delta_{\veps}},
	\end{equation}
	where, as before, $X^{\veps}(\sigma_{j,l-1})$ is replaced with $X^{\veps}(\sigma_{j-1,m})$ when $l=0$ and $j\ge 1$.
	A similar calculation shows that, for some $\kappa \in (0,\infty)$ depending only on  $b, \alpha, v^*$ and $\{y_{i,l}\}$
	$$E\|\bar X^{\veps}(\sigma_{0,1})\|^2  \le \kappa (1+ \|x_0\|^2 + \|y_0\|^2).$$
	Combining this with \eqref{eq:eq323} and\eqref{eq:eq324}, and a recursion based on \eqref{eq:eq527},  it follows that, for some $c_3\in (0,\infty)$ and all $\eps$
	$$ \sup_{0 \le t \le \Delta_{\veps}} E\|\bar Y^{\veps}(t)\|^2 \le c_3(\|y_0\|^2+\|x_0\|^2+1).$$
	Using this estimate in the first line of \eqref{eq:eq519n} and Grönwall lemma, we have that for some $c_4<\infty$, for all $\eps$,
	\begin{equation}\label{eq:eq642}
		\sup_{0 \le t \le \Delta_{\veps}} E \|\bar X^{\veps}(t)\|^2 \le c_4(\|y_0\|^2+\|x_0\|^2+1).\end{equation}
	Using the above estimate in the recursion in \eqref{eq:eq527} and recalling that $N_{\veps}\Delta_{\veps}\le t_1 \le  T$, we see that, for some $c_5<\infty$, and all $\eps$
	$$\max_{0\le j \le N_{\veps}-1, 1 \le l \le m}E\|\bar X^{\veps}(\sigma_{j,l})\|^2 \le c_5(\|y_0\|^2+\|x_0\|^2+1).$$
	Applying the above inequality in \eqref{eq:eq536} we now see that, for some $c_6<\infty$ and all $\eps$
	\begin{equation}\label{eq:845n}\sup_{0 \le t\le s_{N_{\veps}}}E\|\bar X^{\veps}(t)\|^2 \le c_6(\|y_0\|^2+\|x_0\|^2+1).\end{equation}
	This together with \eqref{eq:eq323} and \eqref{eq:eq324} shows that, for some $c_7<\infty$, and all $\eps$
	$$\sup_{0 \le t\le s_{ N_{\veps}}}E\|\bar Y^{\veps}(t)\|^2 \le c_7(\|y_0\|^2+\|x_0\|^2+1).$$
	
	Next, by our choice of the control, note that for any $t \in [s_{0, N_{\veps}}, t_1]$
	$$\bar Y^{\veps}(t) = \bar Y^{\veps}(s_{0, N_{\veps}}) -\frac{1}{\veps} \int_{s_{0, N_{\veps}}}^t \bar Y^{\veps}(u) du + \frac{s(\veps)}{\veps^{1/2}} (B(t)-B(s_{0, N_{\veps}})).$$
	Applying It\^{o}'s formula with the function $f(y)=\|y\|^2$ and using an argument similar to that below \eqref{eq:546n}  we now see that, for some $c_8, c_9 \in (0,\infty)$ and all $\eps$ and  $t \in [s_{0, N_{\veps}}, t_1]$
	$$E\|\bar Y^{\veps}(t)\|^2 \le c_8(E\|\bar Y^{\veps}(s_{0, N_{\veps}})\|^2 + 1) \le c_9(\|y_0\|^2+\|x_0\|^2+1).$$
	Using this bound together with \eqref{eq:845n} and \eqref{eq:s2e6FS}, and the linear growth property of $b$, we now see by a straightforward application of
	Grönwall's lemma that, for some $c_{10} \in (0,\infty)$ and all $\eps$ and $t \in [s_{0, N_{\veps}}, t_1]$, 
	$$E\|\bar X^{\veps}(t)\|^2 \le 
	c_{10}(\|y_0\|^2+\|x_0\|^2+1).$$
The result follows.
\end{proof}
The next lemma is analogous to Lemma \ref{single-lemma-process-movement} and uses the properties of the control $\uu^*$ in \eqref{eq:859n}.
\begin{lemma}
	For $i= 0, 1, \ldots K$, $j = 0, 1, \ldots N^i_{\veps}-1$, $l= 0, 1, \ldots , m(i)-1$,
	$\bar Y^{\veps}(\sigma_{ijl} + \veps \Delta_{\veps}) \to y_{i,l+1}$ in $L^2(P)$.
\end{lemma}

\begin{proof}
	For $t \in [\sigma_{ijl} , \sigma_{ijl} + \veps \Delta^{ij}_{\veps}]$, 
	\begin{align*}
		\bar Y^{\veps}(t) &=  \bar Y^{\veps}(\sigma_{ijl})+ \frac{1}{\veps} \int_{\sigma_{ijl}}^t \uu^{\eps}(\bar Y^{\veps}(\sigma_{ijl}), y_{i,l+1}, s - \sigma_{ijl}) ds + 
		\frac{s(\veps)}{\veps^{1/2}} [B(t)- B(\sigma_{ijl})].
	\end{align*}
Thus
\begin{align*}\bar Y^{\veps}(\sigma_{ijl} + \veps \Delta_{\veps})& = \bar Y^{\veps}(\sigma_{ijl}) + \frac{1}{\veps \Delta_{\veps}} [y_{i,l+1}-\bar Y^{\veps}(\sigma_{ijl})]\veps \Delta_{\veps} + \frac{s(\veps)}{\veps^{1/2}} [B(\sigma_{ijl} + \veps \Delta_{\veps})- B(\sigma_{ijl})]\\
	&= y_{i,l+1} + \frac{s(\veps)}{\veps^{1/2}} [B(\sigma_{ijl} + \veps \Delta_{\veps})- B(\sigma_{ijl})] \to y_{i,l+1},
\end{align*}
in $L^2(P)$ as $\veps \to 0$.	
\end{proof}
Let 
\begin{equation}\label{eq:eq1030}\cli \doteq \{(i,j,l): i= 0, 1, \ldots K, \; j = 0, 1, \ldots N^i_{\veps}-1, \; l= 0, 1, \ldots , m(i)-1\}.
	\end{equation}
The next lemma is analogous to Lemma \ref{lemma:occupational-conv} and uses the stabilization lemma from Section \ref{sec:stablem}.
\begin{lemma}
	\label{lem:lem1006}
	As $\veps \to 0$, 
	\begin{equation*}
		\sup_{(i,j,l)\in \cli} E d_{bl}\left ( \frac{1}{(p_{i,l+1}-\veps) \Delta_{\veps}} \int_{\sigma_{ijl} +\veps\Delta_{\veps}}^{\sigma_{i,j,l+1}} \delta_{\bar Y^{\veps}(s)}ds , \delta_{y_{i,l+1}}\right) \to 0
	\end{equation*}
	and
	\begin{equation*}
		\sup_{(i,j,l)\in \cli}E d_{bl}\left ( \frac{1}{p_{i,l+1} \Delta_{\veps}} \int_{\sigma_{ijl} }^{\sigma_{i,j,l+1}} \delta_{\bar Y^{\veps}(s)}ds , \delta_{y_{i,l+1}}\right)  \to 0.
	\end{equation*}
	
\end{lemma}
\begin{proof}
 	Recall the set $A_R$ defined in \eqref{eq:ar}  for $R>0$. As in \eqref{eq:expectation-bounded Lipschitz} 
	we see that, for $R \in (0,\infty)$,
	\begin{align*}
		&Ed_{bl}\left ( \frac{1}{(p_{i,l+1}-\veps) \Delta_{\veps}} \int_{\sigma_{ijl} +\veps\Delta_{\veps}}^{\sigma_{i,j,l+1}} \delta_{\bar Y^{\veps}(s)}ds , \delta_{y_{i,l+1}}\right) \\
		&= E\dbl\left(\frac{1}{(p_{i,l+1}-\veps) \Delta_{\veps}} \int_{\sigma_{ijl} +\veps\Delta_{\veps}}^{\sigma_{i,j,l+1}}  \delta_{\bar Y^{\veps}(s)- y_{i,l+1}} ds, \delta_{0}\right)\\
		&= E\dbl \left(\frac{1}{(p_{i,l+1}-\veps) \Delta_{\veps}} \int_{0}^{\sigma_{i,j,l+1}-(\sigma_{ijl} +\veps\Delta_{\veps} )} 
		\delta_{\bar Z^{\veps}(s)} ds, \delta_{0}\right)\\
		&\le \sup_{(\bar x, \bar y) \in A_R}E\left[\dbl\left(\frac{1}{(p_{i,l+1}-\veps) \Delta_{\veps}}  \int_{0}^{(p_{i,l+1}-\veps)\Delta_{\veps} } \delta_{ Y^{\veps}(s,y_{i,l+1}, \bar x, \bar y)} ds, \delta_{0}\right)\right]\\
		&\quad\quad\quad\quad\times  P\{(\bar X^{\veps}(\sigma_{ijl}),  \bar Y^{\veps}(\sigma_{ijl} + \veps \Delta_{\veps} - y_{i,l+1})) \in A_R\}\\
		&\quad + \frac{2}{R} \sup_{0\le t \le T}E(\|\bar X^{\veps}(t)\|^2 + \|\bar Y^{\veps}(t)\|^2 + \max_{i,l} \|y_{il}\|^2)
	\end{align*}
	where, as in Lemma \ref{lem:keymombdne}, 
	$\bar Z^{\veps}(s) \doteq \bar Y^{\veps}(s+ \sigma_{ijl} +\veps\Delta_{\veps})- y_{i,l+1}$
	and, $Y^{\veps}(s, y_{i, l+1}, \bar x, \bar y)$ is as in Lemma \ref{lem:llbcgcene}, and the last line follows on noting that  $\sigma_{ij,l+1} -\sigma_{ijl} +\veps\Delta_{\veps}=(p_{i,l+1}-\veps)\Delta_{\veps}$
and that  
	conditioned on $\clf_{\sigma_{i,j,l}+ \veps \Delta_{\veps}}$,
	 $\{\bar Z^{\veps}(s); s \in [0, \sigma_{i,j,l+1}-(\sigma_{i,j,l}+\veps \Delta_{\veps})]\}$
	has the same law as $\{Y^{\veps}(s, y_{i,l+1}, x^{\veps}, y^{\veps}); s \in [0,(p_{i,l+1}-\veps)\Delta_{\veps}]\}$, with $y^{\veps} = \bar Y^{\veps}(\sigma_{i,j,l}+\veps \Delta_{\veps})- y_{i,l+1}$, $x^{\eps}=  \bar X^{\veps}(\sigma_{i,j,l})$.

	The first statement in the lemma now follows on applying Lemma \ref{lem:llbcgcene} and Lemma \ref{lem:keymombdne} by sending $\veps \to 0$ and then $R\to \infty$.
	The second statement is immediate from the first.
\end{proof}
\begin{corollary}
	\label{cor:cor545}
	For $i=0, 1, \ldots K$
	$$\limsup_{\veps \to 0} \max_{ j = 0, 1, \ldots N^i_{\veps}-1} E d_{bl}\left ( \frac{1}{ \Delta_{\veps}} \int_{s_{ij} }^{s_{i, j+1} } \delta_{\bar Y^{\veps}(s)} ds, \nu^*_{i}\right)  \to 0.$$
\end{corollary}
\begin{proof}
	Fix $i=0, 1, \ldots K$. Then, for all $j = 0,\dots,N_\veps^i-1$,
	\begin{multline*}
		E d_{bl}\left ( \frac{1}{ \Delta_{\veps}} \int_{s_{ij} }^{s_{i, j+1} } \delta_{\bar Y^{\veps}(s)} ds,  \nu^*_{i}\right) 
	= 	E d_{bl}\left (\sum_{l=0}^{m(i)-1}p_{i, l+1}  \frac{1}{ \Delta_{\veps} p_{i, l+1} } \int_{\sigma_{ijl}}^{\sigma_{i, j, l+1} } \delta_{\bar Y^{\veps}(s)} ds,  \sum_{l=0}^{m(i)-1}p_{i, l+1} \delta_{y_{i, l+1}}\right)\\
	\le \sum_{l=0}^{m(i)-1}p_{i, l+1} 
	E d_{bl}\left (\frac{1}{ \Delta_{\veps} p_{i, l+1} } \int_{\sigma_{ijl}}^{\sigma_{i, j, l+1} } \delta_{\bar Y^{\veps}(s)} ds, \delta_{y_{i, l+1}}\right)\\
	\le \sup_{(i,j,l)\in \cli} E d_{bl}\left (\frac{1}{ \Delta_{\veps} p_{i, l+1} } \int_{\sigma_{ijl}}^{\sigma_{i, j, l+1} } \delta_{\bar Y^{\veps}(s)} ds, \delta_{y_{i, l+1}}\right).
	\end{multline*}
	The result now follows from Lemma  \ref{lem:lem1006} on sending $\veps \to 0$.
\end{proof}
In order to prove the convergence of $\bar X^{\eps}$ to $\xi^*$, we now introduce an approximation to $\bar X^{\veps} $. Fix $L_{\veps}>0$ such that, as $\veps \to 0$, $L_{\veps}\to 0$ and
$L_{\veps}/\Delta_{\veps} \to \infty$.
Define for $t \in [0,T]$,
\begin{align*}
\hat  X^{\veps}(t) = x_0+ \int_0^t \frac{1}{L_{\veps}} \int_{(s-L_{\veps})}^{s} b(\hat X^{\veps}(s), \bar Y^{\veps}(r)) dr ds + s(\veps) \veps^{1/2} \int_0^t  \alpha(\hat X^{\veps}(s)) dW(s) + \int_0^t \alpha(\hat X^{\veps}(s)) v^*(s) ds,
\end{align*}
where we take $\bar Y^{\veps}(r) \doteq  y_0$ for $r\le 0$.
Unique solvability of the above equation follows from the Lipschitz property of 
$b$ and $\alpha$.
The following lemma shows that $\hat  X^{\veps}$ is close to $\bar  X^{\veps}$ for small $\veps$.
\begin{lemma}
	\label{lem:lem1023}
	As $\veps \to 0$, $E\|\hat  X^{\veps} - \bar  X^{\veps}\|^2_{\infty} \to 0$.
\end{lemma}
\begin{proof}
	We begin by noting that 
	\begin{equation*}
\begin{split}
		\int_0^t \frac{1}{L_{\veps}} \int_{(s-L_{\veps})}^{s} b(\hat X^{\veps}(s), \bar Y^{\veps}(r)) dr ds &=
		\int_0^{t} \frac{1}{L_{\veps}} \int_{r}^{(r+L_{\veps})\wedge t} b(\hat X^{\veps}(s), \bar Y^{\veps}(r)) ds dr\\
		 &+ \int_{- L_{\veps}}^0  \frac{1}{L_{\veps}} \int_{0}^{r+L_{\veps}} b(\hat X^{\veps}(s), \bar Y^{\veps}(r)) ds dr.
\end{split}
\end{equation*}
Next note that
\begin{equation*}
\begin{split}
	\int_0^t b(\bar X^{\veps}(r), \bar Y^{\veps}(r)) dr
	&= \int_0^{t}   \frac{1}{L_{\veps}} \int_{r}^{(r+L_{\veps})\wedge t}  b(\bar X^{\veps}(r), \bar Y^{\veps}(r))  ds dr  + \clr^{\veps}_1(t)\\
	&= \int_0^{t}   \frac{1}{L_{\veps}} \int_{r}^{(r+L_{\veps})\wedge t}  b(\bar X^{\veps}(s), \bar Y^{\veps}(r))  ds dr \\
	 &+ \int_{- L_{\veps}}^0  \frac{1}{L_{\veps}} \int_{0}^{r+L_{\veps} } b(\bar X^{\veps}(s), \bar Y^{\veps}(r)) ds dr + \clr^{\veps}_2(t) + \clr^{\veps}_3(t)  + \clr^{\veps}_1(t),
\end{split}
\end{equation*}
where
$$\clr^{\veps}_1(t) \doteq \int_{t-L_{\veps}} ^t \frac{1}{L_{\veps}} \int_{t}^{r+L_{\eps}} b(\bar X^{\veps}(r), \bar Y^{\veps}(r)) ds dr,$$
$$\clr^{\veps}_2(t)\doteq \int_0^{t}   \frac{1}{L_{\veps}} \int_{r}^{(r+L_{\veps})\wedge t}  (b(\bar X^{\veps}(r), \bar Y^{\veps}(r)) - b(\bar X^{\veps}(s), \bar Y^{\veps}(r)))  ds dr,$$
$$\clr^{\veps}_3(t)\doteq  -\int_{- L_{\veps}}^0  \frac{1}{L_{\veps}} \int_{0}^{r+L_{\veps}} b(\bar X^{\veps}(s), \bar Y^{\veps}(r)) ds dr.$$
Thus, using the Lipschitz property of $b$ 
\begin{equation}
	\|\hat X^{\veps}(t)-\bar X^{\veps}(t)\| \le L_b 
	\int_{-L_{\eps}}^{t}   \frac{1}{L_{\veps}} \int_{r\vee 0}^{(r+L_{\veps})\wedge t}  \|\hat X^{\veps}(s) - \bar X^{\veps}(s)\| ds dr +  D^{\eps} +  C^{\eps}(t) + \sum_{i=1}^3 \|\clr^{\veps}_i(t)\| 
\end{equation}
where
$$D^{\eps} =  s(\eps)\sqrt{\eps}\sup_{0\le t\le T}\|\int_0^t\alpha(\hat X^{\veps}(s)) dW(s) - \int_0^t\alpha(\bar X^{\veps}(s)) dW(s)\|$$
and
$$C^{\veps}(t) = \| \int_0^t (\alpha(\hat X^{\veps}(s)) - \alpha(\bar X^{\veps}(s))) v^*(s) ds \|.$$
We now consider the remainder terms.
Note that for some $c_1 \in (0,\infty)$ depending only on the coefficient $b$
\begin{align*}
	\|\clr^{\veps}_1(t)\| 
	\le  \int_{t-L_{\veps}} ^t \|b(\bar X^{\veps}(r), \bar Y^{\veps}(r))\|  dr
	\le c_1 L_{\veps}^{1/2} \left[\int_0^T (1 + \|\bar X^{\veps}(s)\|^2 + \|\bar Y^{\veps}(s)\|^2) ds\right]^{1/2} \doteq  \tilde \clr^{\veps}_1.
\end{align*}
Next, there is a $c_2 \in (0,\infty)$ depending only on $b,\alpha, v^*$ and $T$, such that  for $\lambda <\infty$, $0\le u \le s \le T$ with $|u-s|\le \lambda$
\begin{align}
\|\bar X^{\veps}(s) - \bar X^{\veps}(u)\| 
\le c_2 \lambda ^{1/2} + s(\veps) \veps^{1/2} \varpi
+ c_2 \lambda ^{1/2} \left[\int_0^T (1 + \|\bar X^{\veps}(s)\|^2 + \|\bar Y^{\veps}(s)\|^2) ds\right]^{1/2} \doteq \clt_{\veps}(\lambda)\nonumber\\
\label{eq:eq1020}
\end{align}
where
$\varpi \doteq 2\sup_{0\le u \le  T} \|\int_0^u \alpha(\bar X^{\veps}(\tau)) dW(\tau)\|.$
Note that
$
\|\clr^{\veps}_2(t)\| \le  TL_b \clt_{\veps} (L_{\veps})\doteq \tilde \clr^{\veps}_2$.
	
Also, from linear growth of $b$,
\begin{align*}
\|\clr^{\veps}_3(t)\|  &\le c_1\int_{- L_{\veps}}^0  \frac{1}{L_{\veps}} \int_{0}^{(r+L_{\veps})}  (1 + \|\bar X^{\veps}(s)\| + \|y_0\|) ds dr	\\
&\le c_1(1 +\|y_0\|) L_{\veps} + c_1 L_{\veps}^{1/2} \left[\int_0^T (1 + \|\bar X^{\veps}(s)\|^2 ) ds\right]^{1/2} \doteq \tilde \clr^{\veps}_3.
\end{align*}
Finally,
$$ C^{\veps}(t)  \le L_{\alpha} \int_0^t  \sup_{0\le u \le s} \|\hat X^{\veps}(u) -\bar X^{\veps}(u))\| \|v^*(s)\| ds.$$
Thus
\begin{align*}
	\sup_{0\le u \le t} \|\hat X^{\veps}(u)-\bar X^{\veps}(u)\| 
	&\le   \int_0^t  \sup_{0\le u \le s} \|\hat X^{\veps}(u) -\bar X^{\veps}(u))\| (L_b + L_{\alpha} \|v^*(s)\| ) ds + D^{\eps}+ \sum_{i=1}^3 \tilde\clr^{\veps}_i.
\end{align*}
The result now follows on taking squared expectations and applying Grönwall lemma together with Lemma \ref{lem:keymombdne}.
\end{proof}

\subsubsection{Convergence of controlled process and costs.}
Let $\bar \La^{\eps}$ be defined by \eqref{eq:laeps} with $\bar Y^{\eps}$ given by \eqref{eq:s2e6FS}.
In this section we show  the convergence of the controlled process $(\bar X^{\eps}, \bar\La^{\eps})$ and estimate the  cost
$E \frac{1}{2} \int_0^T \|u^{\veps}(t)\|^2 dt$ from above as $\eps \to 0$.
\begin{lemma}\label{lem:xvepstoxi}
	As $\veps \to 0$, $(\bar X^{\veps}, \bar \La^{\eps}) \to (\xi^*, \nu^*)$ in probability in $\clx \times \clm_1$, where $(\xi^*, \nu^*)$ is as given by Lemma \ref{lem:approxdisc}.
\end{lemma}
\begin{proof}
	By Lemma \ref{lem:keymombdne}  and a calculation similar to \eqref{eq:eq1020}   we see that $\bar X^{\veps}$ is tight in $\clx$. The tightness of $\bar \Lambda^{\eps}$ is immediate from Lemma \ref{lem:keymombdne} once more.
	Suppose that $(\bar X^{\veps}, \bar \Lambda^{\eps})$ converges in distribution along a subsequence to  $(X^*, \Lambda^*)$. Then
	from Lemma \ref{lem:lem1023}, along the same subsequence, $\hat X^{\veps}$ converges in distribution to $X^*$ as well.
	Define for $s \in [0,T]$,  a $\clp(\RR^d)$ valued random variable $\mu^{\veps}_s $ as
	$$\mu^{\veps}_s(dy) = \frac{1}{L_{\veps}} \int_{(s-L_{\veps})}^{s} \delta_{\bar Y^{\veps}(r)} dr,$$
	where $L_{\veps}$ is as introduced above Lemma \ref{lem:lem1023} and, as before, $\bar Y^{\veps}(r) \doteq y_0$ for $r\le 0$.
	Then, for $t \in  [0,T]$,
	\begin{align}\label{eq:eq1239}
	\hat  X^{\veps}(t) = x_0+ \int_0^t  \int_{\RR^d} b(\hat X^{\veps}(s), y) \mu^{\veps}_s(dy) ds + s(\veps) \veps^{1/2} \int_0^t  \alpha(\hat X^{\veps}(s)) dW(s) + \int_0^t \alpha(\hat X^{\veps}(s)) v^*(s) ds.
	\end{align}
	Also, from Lemma \ref{lem:keymombdne},
	\begin{align}
		\sup_{\veps>0} \sup_{t\in [0,T]} E \int_{\RR^d} \|y\|^2 \mu^{\veps}_t(dy) 
	= \sup_{\veps>0} \sup_{t\in [0,T]} \frac{1}{L_{\veps}} \int_{(t-L_{\veps})}^{t}  E\|\bar Y^{\veps}(r)\|^2 dr \\
		 \le  \sup_{\veps>0} \sup_{t\in [0,T]} E\|\bar Y^{\veps}(t)\|^2 \doteq \kappa_1 <\infty.
		  \label{eq:eq1231}
	\end{align}
	
We now show that for each $i=0, 1\ldots K$, and  a compact $[\theta,\beta]\in (t_i, t_{i+1})$
\begin{equation}\label{eq:1201}
	\sup_{s \in [\theta,\beta]} E[\dbl(\mu^{\veps}_s,  \nu^*_i)] \to 0, \quad \text{as} \; \veps \to 0.
\end{equation}
Let  for $s \in [\theta,\beta]$
$$ \underline j^{\veps} = \min\{j: s_{ij} \ge s- L_{\veps}\}, \;  \bar j^{\veps} = \max\{j: s_{ij} \le s\}$$
and
$\bar L_{\veps} = \Delta_{\veps} (\bar j^{\veps}- \underline  j^{\veps})$.
Then
\begin{align*}
	\mu^{\veps}_s &= \frac{1}{L_{\veps}} \int_{(s-L_{\veps})}^{s} \delta_{\bar Y^{\veps}(r)} dr
	= \frac{1}{L_{\veps}} \sum_{j=\underline j^{\veps} }^{\bar j^{\veps}-1}   \Delta_{\veps} \frac{1}{\Delta_{\veps}} \int_{s_{ij}}^{s_{i,j+1}} \delta_{\bar Y^{\veps}(r)} dr + \frac{L_{\veps}-\bar L_{\veps}}{L_{\veps}} \gamma^{\veps}_s
\end{align*}
for some $\clp(\RR^d)$ valued $\gamma^{\veps}_s$.
Noting that
$\nu^*_i = \frac{1}{L_{\veps}} \sum_{j=\underline j^{\veps} }^{\bar j^{\veps}-1}   \Delta_{\veps} \nu^*_i
	+ \frac{L_{\veps}-\bar L_{\veps}}{L_{\veps}} \nu^*_i$
we see that
\begin{align*}
	\dbl(\mu^{\veps}_s , \nu^*_i) &\le \frac{1}{L_{\veps}} \sum_{j=\underline j^{\veps} }^{\bar j^{\veps}-1}   \Delta_{\veps} \dbl\left(\frac{1}{\Delta_{\veps}} \int_{s_{ij}}^{s_{i,j+1}} \delta_{\bar Y^{\veps}(r)} dr , \nu^*_i  \right ) + \frac{4\Delta_{\veps}}{L_{\veps}}.
\end{align*}	
It then follows that, for $\eps$ sufficiently small,
\begin{align*}
	\sup_{s \in [\theta,\beta]} E\dbl(\mu^{\veps}_s , \nu^*_i) 
	&\le 
	\max_{ j = 0, 1, \ldots N^i_{\veps}-1} E\dbl \left(\frac{1}{\Delta_{\veps}} \int_{s_{ij}}^{s_{i,j+1}} \delta_{\bar Y^{\veps}(r)} dr , \nu^*_i  \right ) + \frac{4\Delta_{\veps}}{L_{\veps}}.
\end{align*}
The statement in \eqref{eq:1201} is now immediate from Corollary \ref{cor:cor545} and recalling that $L_{\veps}/\Delta_{\veps} \to \infty$ as $\veps \to 0$.

We now argue that  $X^*$ solves for  $i=0, 1, \ldots K$ and $t \in (t_i, t_{i+1}]$, 
\begin{equation}\label{eq:eq1243}
X^*(t) = X^*(t_i) + \int_{t_i}^t  \int_{\RR^d} b(X^*(s), y) \nu^*_i(dy) ds  + v^*_i  \int_{t_i}^t \alpha(X^*(s))ds.
\end{equation}
From the continuity of $X^*$,  it suffices to argue that for all  $[\theta,\beta]\in (t_i, t_{i+1})$
\begin{equation}\label{eq:eq1241}
X^*(\beta) = X^*(\theta) + \int_{\theta}^{\beta} \int_{\RR^d} b(X^*(s), y) \nu^*_i(dy) ds  + v^*_i \int_{\theta}^{\beta} \alpha(X^*(s))ds.
\end{equation}

Define $\clp(\RR^d\times [0,T])$ valued random variable $\nu^{\eps}$ as
$\nu^{\veps}(dy dt) \doteq  \frac{1}{T}\mu^{\veps}_t(dy) dt$ and with $\nu^*$ as in Lemma \ref{lem:approxdisc}, let
$\nu(dy\, dt) \doteq   \frac{1}{T}  \nu^*(dy\, dt)  =  \frac{1}{T} \hat\nu_t(dy) dt$  where $\hat\nu_t = \nu^*_i$ for $t \in (t_i, t_{i+1}]$.
Then, from \eqref{eq:1201}, as $\veps \to 0$,
\begin{align}\label{eq:759n}
	E \dbl(\nu^{\veps}, \nu) &\le \frac{1}{T} \int_0^T E\dbl(\mu^{\veps}_t, \hat\nu_t ) dt
	=\frac{1}{T} \sum_{i=0}^K\int_{t_i}^{t_{i+1}} E\dbl(\mu^{\veps}_t, \nu^*_i ) dt
	\to 0.
\end{align}
 By appealing to Skorohod representation theorem we can assume that
\begin{equation}\label{eq:236n}
	(\hat X^{\veps}(\cdot),  s(\veps) \veps^{1/2} 
\int_0^{\cdot}  \alpha(\hat X^{\veps}(s)) dW(s), \nu^{\veps}) \to (X^*(\cdot), 0, \nu)\end{equation}
a.s. in $C([0,T]:  \RR^{2m}) \times \clp(\RR^d\times [0,T])$.
Using the Lipschitz property of $b$
\begin{align*}
	&\left\|  \int_{\theta}^{\beta}  \int_{\RR^d} b(\hat X^{\veps}(s), y) \nu^{\veps}(dy ds)  - \int_{\theta}^{\beta} \int_{\RR^d} b(X^*(s), y) \nu^{\veps}(dy ds) \right\|\\
	&\le L_b\frac{1}{T}\int_{\theta}^{\beta} \|\hat X^{\veps}(s) - X^*(s)\| ds
	 \le \frac{(\beta-\theta)}{T}\sup_{t \in [\theta, \beta]} \|\hat X^{\veps}(t) - X^*(t)\|  \to 0.
\end{align*}
Next using the fact that $\nu^{\veps} \to  \nu$, the moment bound in \eqref{eq:eq1231}, the continuity and linear growth of $b$, and a standard uniform integrability argument, we have
$$
\left\|\int_{\theta}^{\beta} \int_{\RR^d} b(X^*(s), y) \nu^{\veps}(dy ds)
- \int_{\theta}^{\beta} \int_{\RR^d} b(X^*(s), y) \nu(dy ds)\right\| \to 0.$$
Note that
\begin{equation}\label{eq:eq533}
\hat  X^{\veps}(\beta) = \hat  X^{\veps}(\theta)+ T\int_{\theta}^{\beta} \int_{\RR^d} b(\hat X^{\veps}(s), y) \nu^{\veps}(dy  ds) + s(\veps) \veps^{1/2} \int_{\theta}^{\beta}  \alpha(\hat X^{\veps}(s)) dW(s) + v^*_i  \int_{\theta}^{\beta}  \alpha(\hat X^{\veps}(s)) ds.	
\end{equation}
Combining the last three convergence statements and taking limit as $\veps \to 0$ in \eqref{eq:eq533}
we obtain \eqref{eq:eq1241}, which as observed previously proves \eqref{eq:eq1243}.
From the unique solvability of \eqref{eq:eq1243} and the definition of $\xi^*$ we have that $X^* = \xi^*$ which proves the convergence $\bar X^{\eps} \to \xi^*$.
Now we argue that $\La^* =\nu^*$. In view of \eqref{eq:759n}, it suffices to show that, as $\eps \to 0$,
\begin{align}\label{eq:800m}
	E \dbl(\tilde \La^{\eps} , \bar \La^{\eps} ) \to 0,
\end{align}
where $\tilde \La^{\eps} \doteq T\nu^{\veps}(dy dt) =\mu^{\veps}_t(dy) dt$.
Consider $f\in BL_1(\RR^d\times [0,T])$.  
Then, using the properties of $f$, an argument similar to that used in the proof of Lemma \ref{lem:lem1023} shows that
\begin{align*}
	&\left| \int_{\RR^d\times [0,T]} f(s,y) \bar \La^{\eps}(dy\, ds) - \int_{\RR^d\times [0,T]} f(s,y) \tilde \La^{\eps}(dy\, ds)\right|\\
	&= \left| \int_0^T f(r, \bar Y^{\eps}(r)) dr - \int_0^T \frac{1}{L_{\eps}} \int_{s-L_{\eps}}^s f(s, \bar Y^{\eps}(r)) dr\, ds\right|\\
	&\le \int_0^T \frac{1}{L_{\eps}} \int_{r}^{(r+L_{\eps}) \wedge T} |f(r, \bar Y^{\eps}(r))- f(s, \bar Y^{\eps}(r))| ds \, dr + \int_{-L_{\eps}}^0 \frac{1}{L_{\eps}} \int_0^{(r+L_{\eps})\wedge T} |f(s, \bar Y^{\eps}(r))| ds \, dr+ L_{\eps}\\
	&\le (2+T)L_{\eps}.
\end{align*}
Since $L_{\eps}\to 0$ as $\eps \to 0$ we have \eqref{eq:800m}, which, as discussed previously, shows $\La^* = \nu^*$.
This completes the proof of the lemma.
\end{proof}

The next lemma estimates the cost.

\begin{lemma}\label{lem:bdoncost}
	With $u^{\veps}$ defined as in \eqref{eq:defncont} and $\xi^*$ as in Lemma \ref{lem:approxdisc}
	$$\limsup_{\veps \to 0} E \frac{1}{2} \int_0^T \|u^{\veps}(t)\|^2 dt \le
	\frac{1}{2}  \int_0^T
				\int_{\RR^{d}} \| \nabla_y U(\xi^*(s), y)\|^2 \hat\nu_{s}(dy)ds.$$
\end{lemma}
\begin{proof}
Note that
\begin{align*}
	E\int_0^T \|u^{\veps}(t)\|^2 dt = \sum_{i=0}^K \sum_{j=0}^{N^i_{\veps}-1} \sum_{l=0}^{m(i)-1} E \int_{\sigma_{ijl}}^{\sigma_{ij,l+1}} \|u^{\veps}(t)\|^2 dt + \sum_{i=0}^K   E \int_{s_{i, N^i_{\veps}}}^{t_{i+1}} \|u^{\veps}(t)\|^2 dt 
\end{align*}
Recall the notation $\cli$ from \eqref{eq:eq1030}.
	Then, for $(i,j,l) \in \cli$
	\begin{align*}
	 E \int_{\sigma_{ijl}}^{\sigma_{ij,l+1}} \|u^{\veps}(t)\|^2 dt = 	E \int_{\sigma_{ijl}}^{\sigma_{ijl}+\veps \Delta_{\veps}} \|u^{\veps}(t)\|^2 dt + E \int_{\sigma_{ijl}+\veps \Delta_{\veps}}^{\sigma_{ij,l+1}} \|u^{\veps}(t)\|^2 dt
	\end{align*}
For $t \in (\sigma_{ijl}, \sigma_{ijl}+\veps \Delta_{\veps})$, with $C(U)$ as introduced above \eqref{eq:424n},
\begin{align*}
\|u^{\veps}(t)\|^2  &= \|\nabla_y U(\bar X^{\veps}(t), \bar Y^{\veps}(t)) +  \uu^{\eps}(\bar Y^{\veps}(\sigma_{ijl}), y_{i,l+1}, t - \sigma_{ijl})\|^2\\
&\le 2(C(U)+2) \left(1 + \|\bar X^{\veps}(t)\|^2 + \|\bar Y^{\veps}(t)\|^2 +  \frac{1}{\Delta_{\veps}^2} (\|\bar Y^{\veps}(\sigma_{ijl})\|^2 + \|y_{i,l+1}\|^2 )\right)
\end{align*}
Thus, with $c_1 = 4(C(U)+2+ \max_{i,l} \|y_{i,l}\|^2)$,
\begin{align*}
E \int_{\sigma_{ijl}}^{\sigma_{ijl}+\veps \Delta_{\veps}} \|u^{\veps}(t)\|^2 dt 
\le \frac{c_1}{\Delta_{\veps}^2}\veps \Delta_{\veps}\left (1 + \sup_{0\le t \le T} E[ \|\bar X^{\veps}(t)\|^2 + \|\bar Y^{\veps}(t)\|^2]\right) \le \frac{c_2\veps}{\Delta_{\veps}},
\end{align*}
where $c_2= c_1(1+M_1)$ and $M_1$ is as in Lemma \ref{lem:keymombdne}.
Consequently, since $\Delta_{\eps}^2/\eps \to \infty$, as $\veps \to 0$, we have
\begin{align*}
\sum_{i=0}^K \sum_{j=0}^{N^i_{\veps}-1} \sum_{l=0}^{m(i)-1} 	E \int_{\sigma_{ijl}}^{\sigma_{ijl}+\veps \Delta_{\veps}} \|u^{\veps}(t)\|^2 dt  \le \frac{c_3T\veps \sum_{i=0}^K m(i)}{\Delta_{\veps}^2}  \doteq \theta(\eps)\to 0.
\end{align*}
Next consider $t \in (\sigma_{ijl}+\veps \Delta_{\veps}, \sigma_{ij,l+1})$. Fix $\vs \in (0,1)$. Then by Young's inequality,
\begin{align*}
\|u^{\veps}(t)\|^2  &=  \|\nabla_y U(\bar X^{\veps}(t), \bar Y^{\veps}(t)) - 
		\nabla_y U(\bar X^{\veps}(\sigma_{ijl}), \bar Y^{\veps}(t))
		+ \nabla_y U(\bar X^{\veps}(\sigma_{ijl}), y_{i,l+1})\|^2\\	
		&\le (1+\vs)T_3(t) + \frac{4}{\vs} (T_1 (t)+ T_2(t))
\end{align*}
where
\begin{align*}
T_1(t) &= 	\|\nabla_y U(\bar X^{\veps}(t), \bar Y^{\veps}(t)) - 
		\nabla_y U(\bar X^{\veps}(\sigma_{ijl}), \bar Y^{\veps}(t))\|^2\\
T_2(t)&= \| \nabla_y U(\bar X^{\veps}(\sigma_{ijl}), y_{i,l+1}) -  \nabla_y U(\xi^*(s_{ij}), y_{i,l+1})\|^2\\
T_3(t) &= \|\nabla_y U(\xi^*(s_{ij}), y_{i,l+1})\|^2.
\end{align*}
Let
$$\varpi(\xi^*, \Delta_{\veps}) \doteq  \sup_{0\le u \le s \le (u+\Delta_{\veps})\wedge T} \|\xi^*(s)- \xi^*(u)\|.$$
Observe that, for $(i,j,l) \in \cli$,
$$E\int_{\sigma_{ijl}+\veps \Delta_{\veps}}^{\sigma_{ij,l+1}}	T_3(t) dt = \|\nabla_y U(\xi^*(s_{ij}), y_{i,l+1})\|^2 (p_{i,l+1}-\veps) \Delta_{\veps}.$$
Thus, recalling Assumption \ref{Assumptions:system}(2a),
\begin{align*}
\sum_{l=0}^{m(i)-1} 	E\int_{\sigma_{ijl}+\veps \Delta_{\veps}}^{\sigma_{ij,l+1}}	T_3(t) dt
&\le \Delta_{\veps}\int_{\RR^d}		\|\nabla_y U(\xi^*(s_{ij}), y)\|^2 \nu^*_i(dy)\\
&=\int_{s_{ij}}^{s_{i,j+1}}\int_{\RR^d}\|\nabla_y U(\xi^*(s_{ij}), y)\|^2 \hat \nu_s(dy) ds\\
&\le \int_{s_{ij}}^{s_{i,j+1}}\int_{\RR^d} \left((1+ \vs)\|\nabla_y U(\xi^*(s), y)\|^2  + \frac{2}{\vs}L^2_{\mathcal{H}U}  [\varpi(\xi^*,\Delta_{\veps})]^2\right) \hat \nu_s(dy) ds.
\end{align*}
Consequently
$$
\sum_{i=0}^K \sum_{j=0}^{N^i_{\veps}-1} \sum_{l=0}^{m(i)-1} 	E\int_{\sigma_{ijl}+\veps \Delta_{\veps}^{ij}}^{\sigma_{ij,l+1}}	T_3(t) dt \le (1+\vs)\int_0^T \int_{\RR^d}\|\nabla_y U(\xi^*(s, y))\|^2 \hat \nu_s(dy) ds + \frac{2T}{\vs}L^2_{\mathcal{H}U}  [\varpi(\xi^*, \Delta_{\veps})]^2.$$
Next note that for $t \in (\sigma_{ijl}+\veps \Delta_{\veps}, \sigma_{ij,l+1})$, again using
Assumption \ref{Assumptions:system}(2a),
\begin{align*}
T_1(t) &= \|\nabla_y U(\bar X^{\veps}(t), \bar Y^{\veps}(t)) - 
		\nabla_y U(\bar X^{\veps}(\sigma_{ijl}), \bar Y^{\veps}(t))\|^2
		\le L^2_{\mathcal{H}U} [\clt_{\veps}(\Delta_{\veps})]^2,
\end{align*}
where $\clt_{\veps}(\cdot)$ is as in 
\eqref{eq:eq1020}.
Furthermore, using Assumption \ref{Assumptions:system}(2e), for some $C \in (0,\infty)$,
\begin{align*}
T_2(t) &=\| \nabla_y U(\bar X^{\veps}(\sigma_{ijl}), y_{i,l+1}) -  \nabla_y U(\xi^*(s_{ij}), y_{i,l+1})\|^2\\
&\le 2L^2_{\mathcal{H}U}\left(\sup_{0\le t \le T} \|\bar X^{\veps}(t) - \xi^*(t)\|^2\wedge C
 + [\varpi(\xi^*, \Delta_{\veps})]^2\right),
\end{align*}
for all $t\in [0,1]$ and $\eps \in (0,1)$.
Finally, using Lemma \ref{lem:keymombdne} and linear growth of $\nabla_yU$,
\begin{align*}
	\sum_{i=0}^K   E \int_{s_{i, N^i_{\veps}}}^{t_{i+1}} \|u^{\veps}(t)\|^2 dt \le c_3 \Delta_{\veps},
\end{align*}
where $c_3 = 2(K+1)(C(U)+1)(M_1+1)$.
Combining the above estimates 
\begin{multline*}
	E\int_0^T \|u^{\veps}(t)\|^2 dt
	\le (1+ \vs)^2\int_0^T \int_{\RR^d}\|\nabla_y U(\xi^*(s, y))\|^2 \hat \nu_s(dy) ds + (1+\vs)\frac{2T}{\vs} L^2_{\mathcal{H}U} [\varpi(\xi^*, \Delta_{\veps})]^2\\
	+ \frac{4T}{\vs} L^2_{\mathcal{H}U}\left(E[\clt_{\veps}(\Delta_{\veps})]^2 + 2E\left(\sup_{0\le t \le T} \|\bar X^{\veps}(t) - \xi^*(t)\|^2\wedge C
	+
  [\varpi(\xi^*, \Delta_{\veps})]^2\right)\right)
+ c_3 \Delta_{\veps} + \theta(\eps).
\end{multline*}
The result follows on  using Lemmas \ref{lem:keymombdne} and \ref{lem:xvepstoxi} upon first sending $\veps \to 0$ and then $\vs \to 0$.
\end{proof}
\subsubsection{Proof of the LDP lower bound.}
\label{subsec:two-scale-lowern}
Now we complete the proof of the lower bound in  
\eqref{ineq-two-lower}. From  \cite[Corollary 1.10]{buddupbook},  without loss of generality we can assume that $F$ is a bounded Lipschitz function.

From the variational formula in Section \ref{sec:varformula} it follows that
\begin{equation}
\begin{aligned}
            &- \eps s^2(\eps) \log \E e^{-\frac{F(X^\eps, \La^{\eps})}{\eps s^2(\eps)}} \\
            &= \inf_{(v_1,v_2) \in \mathcal{A}} \E \left[ \frac{1}{2} \int_0^T (\|v_1(s)\|^2 ds + \|v_2(s)\|^2) ds  + G^\eps \left( B + \frac{1}{\sqrt{\eps} s(\eps)}\int_0^\cdot v_1(s)ds, W + \frac{1}{\sqrt{\eps} s(\eps)}\int_0^\cdot v_2(s) ds \right)   \right],
\end{aligned}
\end{equation}
where $G^{\eps}$ is as in Section \ref{subsec:two-scale-upper} and $\cla$ is  as in Section \ref{sec:varformula} with $p=d+k$.
Then, since $(v_1, v_2) = (u^{\eps}, v^*) \in \cla$, where $u^{\eps}$ is as constructed in \eqref{eq:defncont} and $v^*$ is as in Lemma \ref{lem:approxdisc},
we have from the above variational formula that, with $\bar X^{\veps}$
as in \eqref{eq:s2e6FS} and  $\bar \La^{\eps}$ defined via \eqref{eq:laeps} with $\bar Y^{\eps}$ as in \eqref{eq:s2e6FS},
\begin{align*}
	&\limsup_{\veps \to 0}-s^2(\veps)\veps E \left[ \exp \left\{-\frac{1}{s^2(\veps)\veps} F(X^{\veps}, \La^{\eps})\right\} \right]\\
	&\le  \limsup_{\veps \to 0} E \left [ F(\bar X^{\veps}, \bar \La^{\eps}) + \frac{1}{2}\int_0^T \|v^*(t)\|^2 dt + \frac{1}{2}\int_0^T \|u^{\veps}(t)\|^2 dt\right]\\
	&\le  \left [ F(\xi^*, \nu^*) + \frac{1}{2}\int_0^T \|v^*(t)\|^2 dt + \frac{1}{2}\int_0^T \int_{\RR^d}\|\nabla_y U(\xi^*(s, y))\|^2 \hat \nu_s(dy) ds\right]\\
	&\le  \inf_{(\xi, \nu) \in \clx \times  \clm_1} \left [ F(\xi, \nu) + I_2(\xi, \nu)\right] + \delta_0.
\end{align*}
where the third line is from Lemmas \ref{lem:xvepstoxi} and \ref{lem:bdoncost}, and the last line follows from Lemma \ref{lem:approxdisc}(4). 
 The bound in \eqref{ineq-two-lower} now follows on sending $\delta_0 \to 0$. \hfill \qed

\subsection{Compactness of level sets of $I_2$.}
\label{sec:cptlevi2}
 In this section we show that the function $I_2$ defined in \eqref{eq:rate-function-2} is a rate function.
 For this it suffices to show that for every $M<\infty$, the set $B_M \doteq \{(\xi, \nu)\in \clx \times \clm_1: I_2(\xi, \nu) \le M\}$ is a compact subset of $\clx \times \clm_1$.
 Now  fix  a $M \in (0, \infty)$ and consider a sequence $\{\xi^n, \nu^n\} \subset B_M$. It suffices to argue that the sequence is relatively compact and there is a limit point of this sequence that belongs to $B_M$. From the definition of $I_2$, there is a $v^n \in \mathcal{U}(\xi^n, \nu^n)$ such that 
\begin{equation}\label{eq:909n}
      \frac{1}{2} \int_0^T \|v^n(s)\|^2 ds + \frac{1}{2} \int_0^T \int_{\mathbb{R}^d} \|\nabla_y U (\xi^n(s),y) \|^2 \nu^n(dy\, ds) \leq M + \frac{1}{n}, \quad \text{for all} \quad n \in \mathbb{N}.
 \end{equation}
Note that $\{v^n\}_{n \in \mathbb{N}} \subset S_{2(M+1)}$, where the latter space is defined as in \eqref{eq:954} (with $p= k$) and so $\{v^n\}_{n \in \mathbb{N}}$ is trivially relatively compact.
 Also,  from Assumption \ref{Assumptions:system}(2c), we have
\begin{equation}
     M+1 \geq \frac{1}{2}\int_0^T \int_{\mathbb{R}^d} \|\nabla_yU(\xi^n(s),y)\|^2 \nu^n(dy\,ds) \geq \frac{1}{2}\int_0^T \int_{\mathbb{R}^d} (L^1_{low}\|y\|^2-L^2_{low}) \nu^n(dy\, ds)
 \end{equation}
 and so 
 \begin{equation}
 \label{rate-system-gamman}
     \sup_{n \in \mathbb{N}} \int_{[0,T] \times \mathbb{R}^d } \|y\|^2 \nu^n (ds dy) \leq 
	 (L^1_{low})^{-1} [2(M+1) + L^2_{low} T]\doteq  \kappa_1 < \infty.
 \end{equation}
This proves that $\{\nu^n\}$ is relatively compact in $\mathcal{M}_1$. Since $v^n \in \mathcal{U}(\xi^n, \nu^n)$, we have that
\begin{equation}
 \label{rate-system-xin}
     \xi^n(t) = x_0 + \int_0^t \int_{\mathbb{R}^d} b(\xi^n(s),y) \nu^n (dy\, ds) + \int_0^t \alpha(\xi^n(s)) v^n(s) ds, \quad \text{for all} \quad t \in [0,T].
 \end{equation}
Using  Grönwall lemma, the linear growth of $b$, the boundedness of $\alpha$, and \eqref{rate-system-gamman}, \eqref{eq:909n}, we now see that
\begin{equation}
     \label{rate-system-total-bound-xi}
     \sup_{n \in \mathbb{N}} \sup_{0 \leq t \leq T} \|\xi^n(t)\|^2 \doteq \kappa_\xi < \infty.
 \end{equation}
Also, using \eqref{eq:909n}, \eqref{rate-system-total-bound-xi}, the linear growth of $b$, and boundedness of $\alpha$, we can find $\kappa_2 \in (0,\infty)$ such that, for all $n \in \mathbb{N}$,  and for $0 \leq s \leq t \leq T$,
\begin{equation}
 \|\xi^n(t)-\xi^n(s)\| \leq \| \int_s^t \int_{\mathbb{R}^d} b(\xi^n(u),y) \nu^n(du dy) \| + \| \int_s^t \alpha(\xi^n(u)) v^n(u) du\| 
     \leq \kappa_2 ( (t-s) + (t-s)^{1/2}).
 \end{equation}
 This estimate, together with \eqref{rate-system-total-bound-xi} shows that $\{ \xi^n \}$ is relatively compact in $\clx$. Now let $\{(\xi^n,\nu^n, v^n)\}_n$ converge along some subsequence (labeled again as $n$) in $\clx \times \mathcal{M}_1 \times \mathcal{S}_{2(M+1)}$ to $(\xi,\nu,v)$. Note that for every $t \in [0,T]$
\begin{equation}
     \begin{aligned}
     &\left\| \int_0^t \alpha(\xi^n(s)) v^n(s) ds - \int_0^t \alpha(\xi(s)) v(s) ds \right\| \\
    & \leq  \int_0^t \left\| \alpha(\xi^n(s)) - \alpha(\xi(s)) \right\| \|v^n(s)\| ds + \left\| \int_0^t \alpha(\xi(s)) (v^n(s) - v(s)) ds \right\| \\
     &\leq  L_{\alpha}(T(2M+1))^{1/2}\sup_{0 \leq s \leq T} \|\xi^n(s) - \xi(s)\| + \left\| \int_0^t \alpha(\xi(s)) (v^n(s) - v(s)) ds    \right\|.
     \end{aligned}
 \end{equation}
 From the convergence of $(\xi^n, v^n) \rightarrow (\xi,v)$ in $\clx \times \mathcal{S}_{2(M+1)}$ the last term converges to $0$ as $n \rightarrow \infty$. This shows that, as $n \rightarrow \infty$, for each $t \in [0,T]$
\begin{equation}
 \label{rate-system-alpha-conv}
     \int_0^t \alpha (\xi^n(s)) v^n(s) ds \rightarrow \int_0^t \alpha (\xi(s)) v(s) ds.
 \end{equation}
  Next,
\begin{equation}
 \label{rate-system-b-conv-1}
     \begin{aligned}
      &\left\| \int_0^t \int_{\mathbb{R}^d} b(\xi^n(s),y) \nu^n(dy\, ds) - \int_0^t \int_{\mathbb{R}^d} b(\xi(s),y) \nu(dy\, ds)   \right\| \\
      &\leq 
	  L_bT \sup_{0 \leq t \leq T} \|\xi^n(s) - \xi(s) \| + \left\| \int_0^t \int_{\mathbb{R}^d} b(\xi(s),y) \nu^n(dy\, ds) - \int_0^t \int_{\mathbb{R}^d} b(\xi(s),y) \nu(dy\, ds)   \right\|.
     \end{aligned}
 \end{equation}
 From the convergence  $(\xi^n, \nu^n) \to (\xi, \nu)$, the continuity of $b$,  the linear growth of $b$ and the square integrability estimate in \eqref{rate-system-gamman} we get that he above quantity converges to $0$ as $n\to \infty$.
Thus we have shown that for all $t \in [0,T]$, as $n \rightarrow \infty$,
\begin{equation}
     \int_0^t \int_{\mathbb{R}^d} b(\xi^n(s),y) \nu^n(dy\, ds) \rightarrow \int_0^t \int_{\mathbb{R}^d} b(\xi(s),y) \nu (dy\, ds).
 \end{equation}
%
Combining the last two convergence statements with \eqref{rate-system-xin} we now see that
 $v \in \clu(\xi, \nu)$. 
%
%
Next, using \eqref{eq:424n}, as $n \rightarrow \infty$
\begin{equation}
     \label{rate-system-conv-U}
     \begin{gathered}
     \| \frac{1}{2} \int_{[0,T] \times {\mathbb{R}^d}} \left( \| \nabla_y U(\xi^n(s),y)\|^2 - \|\nabla_y U(\xi(s),y)\|^2 \right)  \nu^n(dy\, ds) \\
     \leq \frac{1}{2} C(U)\sup_{0 \leq s \leq T} \|\xi^n(s)-\xi(s)\| \int_{[0,T] \times {\mathbb{R}^d}} (1 + \|y\| + 2 \kappa_{\xi}) \nu^n(dy\, ds) \rightarrow 0,
     \end{gathered}
 \end{equation}
 where the last convergence uses \eqref{rate-system-gamman}.
 Finally,  using lower semicontinuity of $u \mapsto \int_0^T \|u(s)\|^2 ds$,  and the fact that $v \in \clu(\xi, \nu)$,
\begin{align*}
     I_2(\xi, \nu) &\le \frac{1}{2} \int_0^T \|v(s)\|^2 ds + \frac{1}{2} \int_{[0,T] \times {\mathbb{R}^d}} \| \nabla_y U(\xi(s),y)\|^2 \nu (dy\, ds) \\
     &\leq \liminf_{n \rightarrow \infty} \left[ \frac{1}{2} \int_0^T \|v^n(s)\|^2 ds + \frac{1}{2} \int_{[0,T] \times {\mathbb{R}^d}} \| \nabla_y U(\xi(s),y)\|^2 \nu^n (dy\, ds)\right] \\
     &= \liminf_{n \rightarrow \infty}  \left[\frac{1}{2} \int_0^T \|v^n(s)\|^2 ds + \frac{1}{2} \int_{[0,T] \times {\mathbb{R}^d}} \| \nabla_y U(\xi^n(s),y)\|^2 \nu^n (dy\, ds) \right]
     \leq \liminf_{n \rightarrow \infty} \left[ M + \frac{1}{n} \right] = M,
     \end{align*}
 where the inequality on second line is from Fatou's lemma and the equality on the third line is from \eqref{rate-system-conv-U}. Thus we have shown that the limit point $(\xi, \nu)$ of $(\xi^n, \nu^n)$ is in $B_M$.
 This completes the proof that $I_2$ is a rate function. \hfill \qed
\paragraph{Acknowledgement}
We gratefully acknowledge several useful conversations with Vivek Borkar and Siva Athreya on this problem. Research of AB supported in part by the NSF (DMS-1814894, DMS-1853968, DMS-2134107 and DMS-2152577). Research of PZ was partly supported by the 2022 Summer Fellowship awarded through UNC's Graduate School.

{\footnotesize 
\bibliographystyle{is-abbrv}

}

\noindent{\scriptsize {\textsc{\noindent  A. Budhiraja and P. Zoubouloglou,\newline
Department of Statistics and Operations Research\newline
University of North Carolina\newline
Chapel Hill, NC 27599, USA\newline
email: budhiraj@email.unc.edu
\newline
email: pavlos@live.unc.edu
\vspace{\baselineskip} } }}

{\footnotesize 
%
%
%
}

{\footnotesize 
%
}

\end{document}